\documentclass[11pt]{amsart}

\usepackage{amssymb,latexsym,amsmath,extarrows, mathrsfs, amsthm,color,amsrefs}
\usepackage{graphicx}
\usepackage{MnSymbol}

\usepackage[margin=1in, centering]{geometry}
\usepackage[bookmarksnumbered, colorlinks, plainpages]{hyperref}
\usepackage{hyperref} 
\hypersetup{
    colorlinks=true,       
    linkcolor=blue,          
    citecolor=magenta,        
    filecolor=magenta,      
    urlcolor=cyan           
}

\usepackage{mathtools}
\mathtoolsset{showonlyrefs,showmanualtags}

\newtheorem{theorem}{Theorem}[section]

\newtheorem{lemma}[theorem]{Lemma}

\newtheorem{remark}[theorem]{Remark}

\newtheorem{definition}{Definition}[section]
\newtheorem{corollary}[theorem]{Corollary}

\newtheorem{fact}[theorem]{Fact}

\newcommand{\Z}{\mathbb Z}
\newcommand{\R}{\mathbb R}
\newcommand{\C}{\mathbb C}
\newcommand{\T}{\mathbb T}
\newcommand{\N}{\mathbb N}
\newcommand{\Q}{\mathbb Q}

\definecolor{deepgreen}{cmyk}{1,0,1,0.5}

\title{Sharp localization on the first supercritical stratum for Liouville frequencies}

\author[R.\ Han]{Rui Han}
\address{Department of Mathematics \\ Louisiana State University  \\  Baton Rouge, LA 70803, USA}
\email{rhan@lsu.edu}

\thanks{
R.\ Han is partially supported by NSF DMS-2143369. 
}
\begin{document}

\begin{abstract}
    We establish Anderson localization for Schr\"odinger operators with even analytic potentials on the first supercritical stratum for Liouville frequencies in the sharp regime $\{E: L(\omega,E)>\beta(\omega)>0, \kappa(\omega,E)=1\}$, with $\kappa(\omega,E)$ being Avila's acceleration.
    This paper builds on the large deviation measure estimate and complexity bound scheme, originally developed for Diophantine frequencies by Bourgain, Goldstein and Schlag \cites{BG,BGS1,BGS2}, and the improved complexity bounds in \cite{HS1}. Additionally, it strengthens the large deviation estimates for weak Liouville frequencies in \cite{HZ}. We also introduce new ideas to handle Liouville frequencies in a sharp way.
\end{abstract}

\maketitle

\section{Introduction} 
We study one-dimensional quasi-periodic Schr\"odinger operators with analytic potentials $v$:
\begin{align}\label{def:H_operator}
    (H_{\omega,\theta}\phi)_n=\phi_{n+1}+\phi_{n-1}+v(\theta+n\omega) \phi_n,
\end{align}
in which $\theta\in \T$ is called the phase, $\omega\in \T\setminus\Q$ is called the frequency.
Throughout the paper, for $\theta\in \T$, $\|\theta\|_{\T}:=\mathrm{dist}(\theta,\Z)$. We shall simply write it as $\|\theta\|$.

The main result of this paper is the following:
\begin{theorem}\label{thm:main}
    Let $v$ be an even analytic function on $\T$. 
    For non-resonant $\theta\in \Theta$, see \eqref{def:Theta}, $H_{\omega,\theta}$ exhibits Anderson localization on 
    $$\{E: L(\omega,E)>\beta(\omega)\geq 0,\text{ and } \kappa(\omega,E)=1\},$$ 
    in which $\beta(\omega)$ is as in \eqref{def:beta}, $L(\omega,E)$ is the Lyapunov exponent and $\kappa(\omega,E)$ is Avila's acceleration number.
\end{theorem}
We call the collection of those energies satisfying $\kappa(\omega,E)=1$ and $L(\omega,E)>0$ {\it the first supercritical stratum}, according to Avila's stratification of the spectrum in \cite{Global}.
Theorem \ref{thm:main} has been established recently in \cite{HS1} for Diophantine $\omega\in \mathrm{DC}$ \footnote{In some of literature, $\mathrm{DC}$ has a slightly different definition, we omit such technicality.},
where
\begin{align}
    \mathrm{DC}:=\{\omega\in \T:\, \|n\omega\|\geq \frac{c}{n (\log n)^{A}}, \text{ for some } A>1 \text{ and } c>0\}.
\end{align}
$\mathrm{DC}$ is a full Lebesgue measure set and is a proper subset of $\{\omega: \beta(\omega)=0\}$.

The set of non-resonant $\theta$ is defined to be
\begin{align}\label{def:Theta}
    \Theta:=\left\{\theta\in \T:\, \limsup_{n\to\infty} \frac{-\log \|2\theta+n\omega\|}{|n|}=0\right\}.
\end{align}
It is well-known that $\Theta$ is a full Lebesgue measure set. 

This paper builds and expands upon the large deviation measure estimates and the complexity bounds for quasi-periodic Schr\"odinger operators. We will review the literature and discuss our contribution in this paper below.

\subsection{Large deviation and complexity bound for analytic potentials}
Following the scheme originally developed by Bourgain, Goldstein and Schlag \cites{BG,BGS1,BGS2}, the main difficulty in proving AL is to control both the large deviation measure estimates (LDM), and {\it at the same time}, the {\it complexity (number of connected components)} of the large deviation sets.

\begin{itemize}
    \item For general analytic potentials, in their seminal paper \cite{BG}, Bourgain and Goldstein proved LDM for the norm of transfer matrices $u_{m,E}(\theta):=m^{-1}\log \|M^{\omega}_{m,E}(\theta)\|$ for Diophantine $\omega$ in the following form:
    \begin{align}\label{eq:BG_LDM}
    \mathrm{mes}(\{\theta: |u_{m,E}(\theta)-\langle u_{m,E}\rangle|>m^{-\sigma}\})\leq e^{-m^{\sigma}}, \text{ for some } \sigma>0,
    \end{align}
    in which $\langle u_{m,E}\rangle$ is the average value of $u_{m,E}$.
    Bourgain and Goldstein also introduced a semi-algebraic set argument which, roughly speaking, controls the {\it complexity} of the large deviation sets in a polynomial manner $m^C$, $C>1$.
    Combining LDM with the complexity bound, Bourgain and Goldstein proved AL for general analytic potentials, for a.e.~(non-arithmetic) $\omega$.
    The approach is very robust, and has been further developed to establish AL in a variety of more general settings, including the higher-dimensional torus $\T^d$ ($d\geq 2$) \cite{BG}, the skew-shift dynamics \cite{BGS1}, the challenging higher-dimensional operators \cites{BGS2,Bo}, block-valued operators \cites{BJ1,DK1,Kl,HS2,HS3}, and continuous quasi-periodic operator \cite{BKV}.
    The proof of LDM for $u_{m,E}$, relying crucially on its almost shift invariant property, in fact generalizes beyond $\omega\in \mathrm{DC}$. For example, You-Zhang \cite{YZ} extended it to the case $\beta(\omega)=o(1)$, and Han-Zhang \cite{HZ} further generalized it to $\beta(\omega)=o(L(\omega,E))$. 
    However, it remained a challenging problem to prove arithmetic AL following this scheme, even for all Diophantine $\omega$'s, essentially due to the lack of a tight complexity bound.

    \item In their landmark work \cite{GS1}, Goldstein and Schlag introduced a new important tool, the Avalanche Principle (AP), into the study of quasi-periodic operators. 
    AP has since played an indispensable role and established a foundation for solving many major problems such as \cites{GS1,GS2,GS3,BJ2} and others. 
    It has also been generalized to higher-dimensional cocycles in \cites{Sch1,DK}, which play an increasingly important role in the study of block-valued operators \cites{DK1,Kl,HS3}.
    In \cite{GS1}, Goldstein and Schlag were able to combine the powers of LDM for $v_{m,E}$ and AP to prove, for the first time, H\"older continuity of IDS for one-dimensional quasi-periodic operators. 
    Later, to study the more challenging problem of determining the exact H\"older exponent, Goldstein and Schlag studied a more subtle alternative object $v^P_{m,E}(\theta):=m^{-1}\log |P_{m,E}^{\omega}(\theta)|$ in \cite{GS2}, where $P_{m,E}$ is the Dirichlet determinant and is an entry of $M^{\omega}_{m,E}$. 
    Their paper, for the first time, revealed the importance of zero count of $P_{m,E}$ and established the connection between the H\"older exponent at energy $E$ and the local zero count of $P_{m,E}$. The local zero count is controlled by the global zero count $N^P_{m,E}$ through a delicate AP argument, leading to H\"older exponent $(m/N^P_{m,E})-0$.
    As part of their deep analysis, they proved LDM for $v^P_{m,E}$ for all Diophantine $\omega\in \mathrm{DC}$, for which they overcame the significant difficulty due to the lack of almost shift invariance for $v^P_{m,E}$. 
    
    \item Goldstein and Schlag in \cite{GS3} proved remarkably LDM for $v^{\mathrm{tr}}_{m,E}(\theta):=m^{-1}\log |\mathrm{tr}(M^{\omega}_{m,E}(\theta))|$, which played a crucial role in their proof of Cantor spectrum for quasi-periodic operators.
    LDM of trace turns out to have profound impacts in other open problems as well, including Han and Schlag's solution \cite{HS3} of a problem on non-perturbative AL for block-valued operators, and more importantly a {\it quantitative version} of Avila's almost reducible conjecture for Diophantine $\omega$'s \cite{HS2}. The latter builds crucially on AL for the dual model and an important formula of Haro and Puig for dual Lyapunov exponents \cite{HP}.
    As in the study of the entries, the proof LDM for trace also presents significant difficulty, due to the lack of almost shift invariance. 
    Goldstein and Schlag's proofs in \cites{GS2,GS3} consist of a series of technical steps, including an additional Cartan estimate, which requires controlling quantitative small deviations for $u_{m,E}(\theta)$, for example, of the form~\eqref{eq:BG_LDM} with deviation $m^{-\sigma}$.
    However, to establish such kinds of small deviations one often requires $\omega\in \mathrm{DC}$.
    
    \item Recently, Han and Schlag \cite{HS1} uncovered the mystery behind Avila's acceleration number and the global zero count of $P_{m,E}^{\omega}$ by establishing a sharp characterization (up to quantitative small error):
    \begin{align}\label{eq:HS1_main}
    \kappa(\omega,E)=N^P_{m,E}/(2m),
    \end{align}
    for all Diophantine $\omega\in \mathrm{DC}$.
    This leads to a sharp complexity bound $2m\cdot \kappa(\omega,E)$ for $v^P_{m,E}$.
    Using this sharp bound, Han and Schlag \cite{HS1} were able to prove a conjectured H\"older exponent for IDS (up to the endpoint), and also arithmetic AL for all $\omega\in \mathrm{DC}$ on the first supercritical stratum.   
\end{itemize}

    This paper builds mainly on \cite{HS1} and earlier works \cites{BG,GS1,GS2,GS3,HZ}, but new ideas need to be introduced to overcome three main difficulties.
    Before we enter detailed discussion, let us briefly introduce the framework.
    Our analysis is multi-scale: for each $n$, we study the decay of eigenfunctions, roughly speaking, on the scale of $(q_n/5, q_{n+1}/5)$. Depending on the growth from $q_n$ to $q_{n+1}$, we divide the scales into weakly Liouvillian scales (when $(\log q_{n+1})/q_n\leq \delta_1$, where $\delta_1$ is as in \eqref{def:delta_1}), and strong Liouville scales (when $(\log q_{n+1})/q_n> \delta_1$). 
    In a strong Liouville scale, to prove AL in the sharp regime, we need to further divide into strongly resonant regimes, those of the form $((\ell-\sigma)q_n, (\ell+\sigma)q_n)$, and weakly resonant regimes $((\ell+\sigma)q_n, (\ell+1-\sigma)q_n)$, where $\sigma>0$ is a small constant. 
    The study of the strongly resonant regimes is the technical core of the paper, which also relies heavily on a sharp analysis of the weakly resonant regimes. 
    Next, we explain the difficulties in each weak/strong Liouville scale and weakly/strongly resonant regime in more details below.
    \begin{enumerate}
        \item 
        For a weak Liouville scale, the complexity bound for $u_{m,E}$, for $m\in (q_n/5, q_{n+1}/5)$, is not tight enough for arithmetic AL. 
        The challenges remain even if one considers $v^P_{m,E}$ or $v^{\mathrm{tr}}_{m,E}$ since their LDM in the literature required $\omega\in \mathrm{DC}$.
        Hence one needs new ideas on controlling LDM and sharp complexity {\it simultaneously} for such weak Liouville scales.
        We overcome this difficulty by studying an alternative object, an analytic function defined by:
\begin{align}\label{def:g}
    g_{m,E}^{\omega}(z)=(P_{m,E}^{\omega}(z))^2+(P_{m-1,E}^{\omega}(z))^2+(P_{m-1,E}^{\omega}(ze^{2\pi i\omega}))^2+(P_{m-2,E}^{\omega}(ze^{2\pi i\omega}))^2,
\end{align}
and the associated subharmonic $v^g_{m,E}(\theta)=(2m)^{-1}\log |g^{\omega}_{m,E}(e^{2\pi i\theta})|$.
        Clearly $g_{m,E}^{\omega}(e^{2\pi i\theta})=\|M_{m,E}^{\omega}(\omega,\theta)\|_{\mathrm{HS}}^2$ for $\theta\in \T$. 
        This function was in fact introduced in Sec.7 of \cite{GS1}, to improve the complexity bound for $u_{m,E}$ from $m^C$ in \cite{BG} to $C'm$, with some implicit large constant $C'$, for $\omega\in \mathrm{DC}$. 
        In this paper, we show that $g_{m,E}(z)$ is, surprisingly, a perfect candidate for arithmetic AL (except possibly in strong Liouville case). 
        On one hand, it is equivalent to the operator norm, hence it inherits directly the LDM for $u_{m,E}$ on the scale $m\in (q_n/5, q_{n+1}/5)$. 
        On the other hand, it is an analytic function, thus making it possible to try to adapt a similar strategy in \cite{HS1} for this Liouville setting, to improve Goldstein and Schlag's complexity bound in \cite{GS1} from $C'm$ to exactly $2m$ on the first supercritical stratum. 
        In \cite{HS1}, the complexity of $v^P_{m,E}$ was bounded by the number of zeros of $P_{m,E}^{\omega}$, which in turn is determined by Avila's acceleration as in \eqref{eq:HS1_main}. Such bound is possible because by LDM and an additional Cartan estimate, each connected component of the large deviation set must be close to at least one zero of $P_{m,E}^{\omega}$.
        However, as we mentioned, an additional Cartan argument often requires small deviation $\sim m^{-\sigma}$ and hence $\omega\in \mathrm{DC}$, which is not the case here.
        We overcome this difficulty by developing a different approach: directly bounding the complexity of $v^g_{m,E}$ by the number of zeros for a perturbed function $\tilde{g}_{m,E}^{\omega}$, whose zeros are also determined by the acceleration number.
        Finally, combining LDM and the sharp complexity bound $2m$ with the {\it pigeonhole principle argument} as in \cite{HS1} one can prove exponential decay of eigenfunctions in such a scale.

        \item For strong Liouville scales, there had been no quantitative LDM for $u_{m,E}$ for $m\in [\sigma q_n, \sigma^{-1} q_n]$, where $\sigma>0$ small.
        This is because LDM is essentially a quantitative ergodic theorem, hence being exponentially close to rationals makes the system less ergodic, thus causing significant difficulties.
        The only LDM on such a scale goes back to key lemma of Bourgain and Jitomirskaya, \cite[Lemma 4]{BJ2}, where LDM was proved for $u_{m,E}$ for $m>\sigma^{-1} q_n$.
        In this paper we are able to prove LDM for $v^g_{m,E}$ for $m\in [\sigma q_n,\sigma^{-1} q_n]$, an important regime that is out of the reach of \cite{BJ2}.
        The applicability of LDM for $m\in [\sigma q_n, q_n/2]$ is crucial, and is in fact the only range we need, for our analysis of the weakly resonant regimes in a strong Liouville scale.

        \item To study the strongly resonant regimes at a strong Liouville scale, we do not consider $v^g_{m,E}$; instead, we study $v^{\mathrm{tr}}_{q_n,E}$. 
        The big advantage of studying the trace is that we can directly derive the structure of its zeros using rational approximation, and obtain {\it sharp} LDMs and complexity bounds directly from such a zero structure. 
        This is a completely new approach for LDMs and complexity bounds; indeed, it allows one to truly utilize the Liouville feature instead of working against it as in a conventional resolution of a small divisor problem. 
        This rational approximation idea is inspired by Avila's proof \cite{arc} of the almost reducibility conjecture for Liouville frequencies on the subcritical stratum, a stratum complementary to the supercritical ones we study in this paper.
        Based on our analysis of $v^{\mathrm{tr}}_{m,E}$, we further develop a new approach to AL for strong Liouville frequencies, by directly studying the closeness of the orbit $\{\theta+k\omega\}_{k\in \Z}$ to the zeros of $\mathrm{tr}(M^{\omega}_{q_n,E})$, without using the pigeonhole principle argument as in~\cite{HS1}.
        This approach is deterministic and yields a good control of the resolvent on the interval $[-[q_n/2], [q_n/2]]$, instead of merely establishing the existence of a good interval by the pigeonhole principle.
        Note that with a LDM for $g_{m,E}$, $m\in [q_n,2q_n]$, and its associated complexity bound $2m$, one can perhaps develop an alternative approach to the strongly resonant scales. But we do not pursue this idea here.
        \end{enumerate}

The purely singular continuous spectrum in the complementary regime $\{L(\omega,E)<\beta(\omega)\}$ established in \cite{AYZ} demonstrates the {\it sharpness} of our result on the first supercritical stratum.

\subsection{Related results in special cases}
In the past, sharp arithmetic localization for Liouville frequencies has only been established for two specific models, one being the Maryland model, when $v(\theta)=\lambda\tan(\pi\theta)$, the other one being the almost Mathieu operator (AMO), when $v(\theta)=2\lambda\cos(2\pi\theta)$ with $\lambda>1$. 

In the literature, the Maryland model had been studied via the Cayley transform, an indirect approach that works uniquely for the $\tan$ potential. 
Those studies had led to a complete spectral characterization \cites{GFP,FP,Si2,JL2}, with a proof of AL for all Diophantine frequencies back in the 1980s' \cites{FP,Si2}.
An alternative direct approach, based on the Green's function expansion, has been quite recently developed for the Maryland model \cites{JY}, and further extended in \cite{HJY} revealing some novel structures of eigenfunctions due to presence of large potential barriers. 
But that approach is also $\tan$ specific. 
The proofs of AL for AMO are also restricted to the $\cos$ potential: \cite{AYZ} relies on the reducibility for the dual operator and the unique ``self-dual" feature of AMO; \cites{Ji,JL2} utilize crucially the Lagrange interpolation argument (an algebraic property) for the Dirichlet determinant $P_{m,E}^{\omega}(\theta)$, which is a {\it polynomial in $\cos$ of degree exactly $m$}.
To the best of our knowledge, neither the duality nor the Lagrange interpolation argument currently apply to general analytic potentials.
But it might be possible to incorporate the new ideas we introduce in this paper to further extend these two approaches beyond AMO.

It is desirable to understand the phase transition phenomenon, beyond special cases, in a more robust and physically relevant setting, e.g., for models with general even analytic potentials.
In fact, AMO originates from studying the motion of a single electron on the $\Z^2$ lattice, subjected to a transversal magnetic field. There, the electron is only allowed to hop to its nearest neighbors, which gives rise to the $\cos$ potential upon taking the Fourier transform.
However, in the real world, the electron indeed exhibits infinite-distance hopping with an exponentially decaying symmetric hopping strength, which gives rise to a true {\it even analytic potential}. 
In this paper, we show that arithmetic phase transition phenomena are indeed ``{\it topological invariant}": they are robust within the class of $\kappa(\omega,E)=1$ for general even potentials. 
It is worth-noting that AMO with $\lambda>1$ together with its even analytic perturbations are a special case of the operators we study here: the entire spectrum of AMO with $\lambda>1$ (and any analytic perturbation) is contained in the first supercritical stratum, see \cite[Lemma 25]{Global}.
Our analysis also allows one to establish the hierarchical structure of eigenfunctions in the localization regime, similar to that of AMO in \cite{JL2}. 
To achieve that, one replaces our study around the global maximal, roughly speaking $|\phi_0|$, with an arbitrary local max. We leave it for interested readers.

We introduce some notations before we proceed.
For any $R>1$, let $\mathcal{A}_R:=\{z\in \C: 1/R<|z|<R\}$ be the annulus. Let $\mathcal{C}_r:=\{z\in \C: |z|=r\}$ be the circle with radius $r>0$.
For $z\in \C$ and $r>0$, let $B_r(z)$ be the open ball centered at $z$ with radius $r$.
For $x\in \R$, let $[x]$ be the largest integer such that $[x]\leq x$.

We organize the rest of the paper as follows: Sec.~\ref{sec:pre} is devoted to preliminaries, Sec.~\ref{sec:overview} contains an overview of Anderson localization, in particular how to incorporate the weak/strong Liouville scale analysis presented in Sec.~\ref{sec:non_exp} and~\ref{sec:strong} respectively. 
Sec.~\ref{sec:LDT} is on the proofs of the large deviation estimates in Sec.~\ref{sec:non_exp}.

\section{Preliminaries}\label{sec:pre}

\subsection{Non-resonant $\theta$'s}
Clearly, for any $\theta\in \Theta$, and any small constant $\delta'>0$, for $n$ large enough, we have
\begin{align}\label{eq:theta_non_res}
    \|2\theta+n\omega\|\geq e^{-\delta' |n|}.
\end{align}
By restricting to $\theta\in \Theta$, one essentially rules out the resonance caused by reflection invariance.

\subsection{Continued fraction expansion}\label{sec:continued_fraction}
Give $\omega\in (0,1)$, let $[a_1,a_2,...]$ be the continued fraction expansion of $\omega$. For $n\geq 1$, let $p_n/q_n=[a_1,a_2,...,a_n]$ be the continued fraction approximants of $\omega$. 
The following property is well-known for $n\geq 1$,
\begin{align}\label{eq:qn_omega_min}
\|q_n\omega\|_{\T}=\min_{1\leq k<q_{n+1}} \|k\omega\|_{\T},
\end{align}
The $\beta(\omega)$ exponent measures the exponential closeness of $\omega$ to rational numbers:
\begin{align}\label{def:beta}
\beta(\omega):=\limsup_{n\to\infty}\frac{\log q_{n+1}}{q_n}=\limsup_{n\to\infty}\left(-\frac{\log \|n\omega\|_{\T}}{n}\right).
\end{align}
Let 
\begin{align}\label{def:betan}
\beta_n(\omega):=\frac{\log q_{n+1}}{q_n}. 
\end{align}
It is well-known that
\begin{align}\label{eq:qn_omega}
\|q_n\omega\|\in (1/(2q_{n+1}), 1/q_{n+1})=(e^{-\beta_n(\omega) q_n}/2, e^{-\beta_n(\omega) q_n}).
\end{align}
The $p_n, q_n$'s are determined by $a_n$'s in the following way:
\begin{align}\label{eq:qn+1=qn+qn-1}
    q_{n+1}=a_{n+1}q_n+q_{n-1}, \text{ and } p_{n+1}=a_{n+1}p_n+p_{n-1},
\end{align}
and
\begin{align}\label{eq:qn-1_norm=qn+qn+1}
\|q_{n-1}\omega\|=a_{n+1}\|q_n\omega\|+\|q_{n+1}\omega\|.
\end{align}
To see the latter is true, note that by \eqref{eq:qn+1=qn+qn-1},
\begin{align}
    q_{n+1}\omega-p_{n+1}=a_{n+1}(q_n\omega-p_n)+(q_{n-1}\omega-p_{n-1}),
\end{align}
which yields \eqref{eq:qn-1_norm=qn+qn+1} since $\|q_n\omega\|=|q_n\omega-p_n|$ and
$$(q_{n+1}\omega-p_{n+1})\cdot (q_n\omega-p_n)<0,$$
for any $n\geq 1$.

\subsection{Green's function for the annulus.}We state the precise Green's kernel, which can be derived  by the method of images. 

\begin{lemma}\label{lem:Green_AR}\cite{HS1}
The Green's function on the annulus $\mathcal{A}_R$ is given by
\begin{align}\label{def:G}
G_R(z,w)=\frac{1}{2\pi}\log |z-w| + \Gamma_R(z,w),
\end{align}
where
\begin{align}\label{def:GammaR}
\Gamma_R(z,w)=\frac{\log( |z|/ R) \log (|w|/ R)}{4\pi\log R}+\frac{1}{2\pi}\log \left( \frac{\prod_{k=1}^{\infty} |1-\frac{1}{R^{4k}}\frac{z}{w}| \cdot |1-\frac{1}{R^{4k}} \frac{w}{z}|}{R\cdot \prod_{k=1}^{\infty} |1-\frac{1}{R^{4k-2}}w\overline{z}|\cdot |1-\frac{1}{R^{4k-2}} \frac{1}{\overline{z}w}|}\right).
\end{align}
The Green's function is symmetric and invariant under rotations: $G_R(z,w)=G_R(w,z)$ and $G_R(z,w)=G_R(e^{i\phi} z,e^{i\phi} w)$. 
\end{lemma}
It is also easy to check that
\begin{align}\label{eq:GR_even}
    &2\pi G_R(1/\overline{z}, 1/\overline{w})\\
    =&
    \log |1/\overline{z}-1/\overline{w}|+\frac{\log (1/(R|z|))\log (1/(R|w|))}{2\log R}
    +\log \left( \frac{\prod_{k=1}^{\infty} |1-\frac{1}{R^{4k}}\frac{\overline{w}}{\overline{z}}| \cdot |1-\frac{1}{R^{4k}} \frac{\overline{z}}{\overline{w}}|}{R\cdot \prod_{k=1}^{\infty} |1-\frac{1}{R^{4k-2}}\frac{1}{\overline{w}z}|\cdot |1-\frac{1}{R^{4k-2}} z\overline{w}|}\right)\\
    =&2\pi G_R(z,w).
\end{align}

The following integral is useful, see \cite[Lemma 3.2]{HS1}, note $\Gamma_R=H_R$ therein.
\begin{align}\label{eq:int_HR}
\int_0^{1}   \Gamma_R(re^{ 2\pi i\theta},w) \, \mathrm{d}\theta= \frac{\log(r/R)}{4\pi\log  R} \log(|w|/R) - \frac{\log R}{2\pi}.
\end{align}

\subsection{Cartan set and estimate}
\begin{definition}[Cartan set]
For an arbitrary subset $\mathcal{P}\subset \mathcal{D}(z_0,1)\subset \C$, where $\mathcal{D}(z_0,1)$ is the disk, we say that $\mathcal{P}\in \mathrm{Car}(H,K)$ if $\mathcal{P}\subset \bigcup_{k=1}^{k_0} \mathcal{D}(z_k, r_k)$ with $k_0\leq K$,  and
\begin{align}\label{eq:Hr}
\sum_{j} r_j<e^{-H}.
\end{align}
\end{definition}

By Wiener's covering lemma  we can assume that $\mathcal{D}(z_k,r_k)$ are pairwise disjoint,
 at the expense of a factor of~$3$ in~\eqref{eq:Hr}. 

\begin{lemma}\label{lem:Cartan}
Let $\varphi$ be an analytic function defined in a disk $\mathcal{D}:=\mathcal{D}(z_0,1)$. Let 
$M\geq \sup_{z\in \mathcal{D}} \log |\varphi(z)|$, $m\leq \log |\varphi(z_0)|$. Given $H\gg 1$, there exists a set $\mathcal{P}\subset \mathcal{D}$, $\mathcal{P}\in \mathrm{Car}(H,K)$, $K=CH(M-m)$ for some absolute constant $C>0$, such that 
\begin{align}
\log |\varphi(z)|>M-CH(M-m),
\end{align}
for any $z\in \mathcal{D}(z_0, 1/6)\setminus \mathcal{P}$.
\end{lemma}

\subsection{Cocycles and Lyapunov exponents}

Let $(\omega, A)\in (\T, C^{\omega}(\T, \mathrm{SL}(2,\R)))$. 
Let 
\begin{align}
A_n^{\omega}(\theta)=A(\theta+(n-1)\omega)\cdots A(\theta).
\end{align}
Let the finite-scale Lyapunov exponents be defined as 
\begin{align}\label{def:LE}
L_n(\omega, A):=\frac{1}{n}\int_{\T} \log \|A_n^{\omega}(\theta)\|\, \mathrm{d}\theta,
\end{align} 
and the infinite-scale Lyapunov exponents as
\begin{align}
L(\omega, A)=\lim_{n\to\infty}L_n(\omega, A).
\end{align}

We denote the phase-complexified Lyapunov exponents as $$L_n(\omega,A(\cdot+i\varepsilon))=:L_n(\omega,A,\varepsilon), \text{ and }L(\omega,A(\cdot+i\varepsilon))=:L(\omega,A,\varepsilon),$$ respectively.

If $\omega=p/q\in \Q$, we define
\begin{align}\label{def:Lpq}
L(p/q,A,\theta)=\lim_{n\to\infty} \frac{1}{n}\log \|A_n^{p/q}(\theta)\|=\frac{1}{q}\log (\rho(A_q^{p/q}(\theta))),
\end{align}
where $\rho(A)$ is the spectral radius of $A$.

The Schr\"odinger cocycle $(\omega,M_E^{\omega})$ associated to the operator \eqref{def:H_operator} is defined with
\begin{align}
M_E^{\omega}(\theta)=\left(\begin{matrix}
    E-v(\theta)\, &-1\\
    1 &0
\end{matrix}\right).
\end{align}
The matrix $M_E^{\omega}(\theta)$ is called transfer matrix, and 
\begin{align}\label{def:MnE}
M_{n,E}^{\omega}(\theta)=M_E^{\omega}(\theta+(n-1)\omega)\cdots M_E^{\omega}(\theta)
\end{align}
is called $n$-step transfer matrix.

\subsection{Avila's acceleration}
Let $(\omega, A)\in (\T, C^{\omega}(\T, \mathrm{SL}(2, \R)))$.
The Lyapunov exponent $L(\omega,A,\varepsilon)$ is a convex and even function in $\varepsilon$.
Avila defined the acceleration to be the right-derivative as follows:
\begin{align}
\kappa(\omega, A,\varepsilon):=\lim_{\varepsilon'\to 0^+} \frac{L(\omega, A,\varepsilon+\varepsilon')-L(\omega, A,\varepsilon)}{2\pi \varepsilon'}.
\end{align}
As a cornerstone of his global theory \cite{Global}, he showed that for analytic $A\in \mathrm{SL}(2,\R)$ and irrational $\omega$, $\kappa(\omega, A,\varepsilon)\in \Z$ is always quantized.

Recall that $v$ is an analytic functions on $\T_{\varepsilon_0}$ for some $\varepsilon_0>0$.  We may shrink $\varepsilon_0$ when necessary such that 
\begin{align}\label{eq:L_linear}
L(\omega,M_E^{\omega},\varepsilon)=L(\omega,M_E^{\omega},0)+2\pi \kappa(\omega,M_E^{\omega},0)\cdot |\varepsilon|, \text{ holds for any } |\varepsilon|\leq \varepsilon_0.
\end{align}
For the rest of the paper, when $\varepsilon=0$, we shall omit $\varepsilon$ from various notations of Lyapunov exponents and accelerations.
We will also write $L(\omega,E,\varepsilon)$ instead of $L(\omega,M_E^{\omega},\varepsilon)$, and sometimes even omit $\omega,E$ in the notation.

\subsection{Regular/Dominated cocycles}
Following \cite{AJS}, we say an analytic cocycle $(\omega,A)$ is regular if $\kappa(\omega,A,\varepsilon)$ is constant for $\varepsilon$ in a small neighborhood of $\varepsilon=0$. 
We will use the following theorem from \cite{AJS},  note we restrict to the $2\times 2$ cocycle case below.
\begin{theorem}\cite[Theorem 5.2]{AJS}\label{thm:dominated}
Assume $L(\omega,A)>-\infty$ and that $(\alpha,A)$ is regular. Then for any rational approximant $p/q$ of $\omega$, one has uniformly for small $\varepsilon$ and all $\theta\in \T$
\begin{align}
    L(p/q,A,\theta+i\varepsilon)=L(\omega,A,\varepsilon) + o(1).
\end{align}
\end{theorem}

\subsection{Transfer matrix and Dirichlet determinants}
It is well-known that the entries of the transfer matrices are the Dirichlet determinants:
\begin{align}\label{eq:M=P}
    M_{n,E}^{\omega}(\theta)=\left(\begin{matrix}
        P_{n,E}^{\omega}(\theta) & -P_{n-1,E}^{\omega}(\theta+\omega)\\
        P_{n-1,E}^{\omega}(\theta) &-P_{n-2,E}^{\omega}(\theta+\omega)
    \end{matrix}\right)\in \mathrm{SL}(2,\R),
\end{align}
in which $P_{k,E}^{\omega}(\theta):=\det(H_{[0,k-1],\omega,\theta}-E)$ is the Dirichlet determinant on the interval $[0,k-1]$.

Let 
\begin{align}\label{def:f=2-tr}
f_{n,E}^{\omega}(\theta):=\det(M_{n,E}^{\omega}(\theta)-I_2)=2-\mathrm{tr}(M_{n,E}^{\omega}(\theta)).
\end{align}
In fact $f_{n,E}^{\omega}$ is the determinant of $H_{[0,n-1],\omega,\theta}$ with periodic boundary condition, see \cite[Lemma 5.1]{HS2}.
This played an important role in Goldstein-Schlag's proof of Cantor spectrum \cite{GS3}, and Han-Schlag's proof of 
a quantitative version of Avila's almost reducibility conjecture \cite{HS2}.
$f^{p/q}_{q,E}(\theta)$ is a $1/q$-periodic function in $\theta$. This fact plays a crucial role in Avila's global theory, and our proof of Anderson localization for Liouville frequencies.

Since $M_{n,E}^{\omega}\in \mathrm{SL}(2,\R)$, one has
\begin{align}\label{eq:f=M+M}
    2I_2-M_{n,E}^{\omega}(\theta)-(M_{n,E}^{\omega}(\theta))^{-1}=f^{\omega}_{n,E}(\theta) \cdot I_2.
\end{align}
This implies
\begin{align}\label{eq:f=M2/M}
    \frac{\|(M_{n,E}^{\omega}(\theta))^2\|}{\|M_{n,E}^{\omega}(\theta)\|}-3\leq |f^{\omega}_{n,E}(\theta)|\leq \frac{\|(M_{n,E}^{\omega}(\theta))^2\|}{\|M_{n,E}^{\omega}(\theta)\|}+3\leq \|M_{n,E}^{\omega}(\theta)\|+3.
\end{align}

We shall also frequently write $f_{k,E}^{\omega}(z):=f_{k,E}^{\omega}(\theta)$ and $P_{k,E}^{\omega}(z):=P_{k,E}^{\omega}(\theta)$, with the obvious identification $z=e^{2\pi i\theta}$.

\subsection{Green's function expansion}
Let $\phi$ be such that it solves $H_{\omega,\theta}\phi=E\phi$ for some $E\in \R$.
Then for any interval $[m_1,m_2]\subset \Z$, we have for any $h\in [m_1,m_2]$ that
\begin{align}\label{eq:Green_exp}
    |\phi_h|\leq \frac{|P_{m_2-h,E}^{\omega}(\theta+(h+1)\omega)|}{|P_{m_2-m_1+1,E}^{\omega}(\theta+m_1\omega)|}|\phi_{m_1-1}|+\frac{|P_{h-m_1}^{\omega}(\theta+m_1\omega)|}{|P_{m_2-m_1+1,E}^{\omega}(\theta+m_1\omega)|}|\phi_{m_2+1}|.
\end{align}
Note that we avoid introducing the Green's function for the Schr\"odinger operator, rather, we use directly its connection to the Dirichlet determinants to avoid confusion with the Green's function on the annulus in \eqref{def:G}.

\subsection{Uniform upper semi-continuity}
The following lemma is an easy corollary of the arguments, essentially a subadditivity argument, in the proof of Lemma 5.1 in \cite{AJS}. 
\begin{lemma}\label{lem:upper_semi_cont}
For any small $\tau>0$, there exists $N=N(\tau, \omega, v, E)\geq 1$ and $\delta=\delta(\tau,\omega,v,E)>0$ such that for any $\|\omega'-\omega\|_{\T}\leq \delta$ and $n\geq N$,
\begin{align}
\frac{1}{n} \log \left\|M_{n,E}^{\omega'}(\theta)\right\|\leq L(\omega,M_{E})+\tau,
\end{align}
uniformly in $\theta\in \T$.
\end{lemma}

\subsection{Symmetry of zeros}
\begin{fact}\label{fact:zero_real}\cite[Fact 2.1]{HS2}
For any $\omega\in \T$, any $n\in \N$ and $E\in \R$. If $z\notin \mathcal{C}_1$ is a zero of $f_{n,E}^{\omega}(z)$ (or $P_{n,E}^{\omega}(z)$), then $1/\overline{z}$ is also a zero.
\end{fact}
\begin{proof}
    Since the potential function $v$ is real-valued, $f^{\omega}_{n,E}(e^{2\pi i\theta})\in \R$ for $\theta\in \T$. Hence the two analytic function $f_{n,E}^{\omega}(z)=\overline{f_{n,E}^{\omega}(1/\overline{z})}$ coincide on the unit circle $z\in \mathcal{C}_1$, which implies they are identical. 
\end{proof}

\subsection{Shnol' theorem and generalized eigenfunction}
By Shnol's theorem \cites{Be,Si1,Ha,Sch}, to prove Anderson localization, it suffices to show an arbitrary generalized eigenfunction $\phi$ to $H_{\omega,\theta}\phi=E\phi$, satisfying 
\begin{align}\label{eq:Shnol}
    1=\max(|\phi_0|, |\phi_{-1}|), \text{ and } |\phi_k|\leq C|k|, \text{ for any } k\neq 0,
\end{align}
decays exponentially. 
Throughout the rest of the paper, we fix such a generalized eigenpair $(E,\phi)$.

\section{An overview of the proof of Anderson localization}\label{sec:overview}
In this section, we give an overview of the proof, in particular, how to incorporate the two different weak/strong Liouville scale analysis presented in Sec. \ref{sec:non_exp} and \ref{sec:strong} respectively to prove decay of eigenfunction on the whole $\Z$. 
We need to pay extra attention to the applicability of the Theorems \ref{thm:AL_weak_resonant} and \ref{thm:AL_main_beta} at two consecutive scales.

Let $C_v\geq 1$ be the constant in \eqref{def:C_v}. Note that $C_v$ depends solely on $\|v\|_{L^{\infty}(\T_{\varepsilon_0})}$. 
Recall $\varepsilon_0$ is as in \eqref{eq:L_linear}.

Throughout the rest of the paper, let $\delta_1>0$ be a small constant such that
\begin{align}\label{def:delta_1}
    \delta_1^{1/4}=\frac{\min(\varepsilon_0,1, L(\omega,E)-\beta(\omega))}{10^5 C_v \max(1, L(\omega,E))}.
\end{align}
Let 
\begin{align}\label{def:c0}
    c_0=10\delta_1^{1/4}/(1+\beta(\omega)),
\end{align}
and
\begin{align}\label{def:eta}
    \eta=1000C_v\varepsilon_0^{-1}\delta_1^{1/4}<\frac{1}{100}.
\end{align}

Below, let $*=q_{k-1}/10$ if $q_{k-1}\leq e^{\delta_1 q_{k-2}}$, and $*=q_{k-1}^{1-c_0}$ if $q_{k-1}>e^{\delta_1 q_{k-2}}$. 

We divide into four cases:

\underline{Case 1.} If $q_k\leq e^{\delta_1 q_{k-1}}$ and $q_{k+1}\leq e^{\delta_1 q_k}$. 
One applies Theorem \ref{thm:AL_weak_resonant} to both the scales $n=k-1$ and $n=k$ so that $|\phi_y|$ decays on $[*, q_k/10]\bigcup [q_k/10, q_{k+1}/10]$. 
We note that the two consecutive scales leave no space uncovered around their connection at $q_k/10$.

\underline{Case 2.} If $q_k\geq e^{\delta_1 q_{k-1}}$ and $q_{k+1}\leq e^{\delta_1 q_k}$. 
One applies Theorem \ref{thm:AL_main_beta} to the scale $n=k-1$ so that $|\phi_y|$ decays on $[*, q_k^{1-c_0}]$.  One applies Theorem \ref{thm:AL_weak_resonant} to the scale $n=k$ so that $|\phi_y|$ decays on $[q_k^{1-c_0}, q_{k+1}/10]$.

\underline{Case 3.} If $q_k\leq e^{\delta_1 q_{k-1}}$ and $q_{k+1}\geq e^{\delta_1 q_k}$. 
One applies Theorem \ref{thm:AL_weak_resonant} to the scale $n=k-1$ so that $|\phi_y|$ decays on $[*, q_k/10]$, and Theorem \ref{thm:AL_main_beta} to the scale $n=k$ so that $|\phi_y|$ decays on $[q_k/10, q_{k+1}^{1-c_0}]$.

\underline{Case 4.} If $q_k\geq e^{\delta_1 q_{k-1}}$ and $q_{k+1}\geq e^{\delta_1 q_k}$. 
One applies Theorem \ref{thm:AL_main_beta} to both scales $n=k-1$ and $n=k$ so that $|\phi_y|$ decays on $[*, q_k^{1-c_0}]\bigcup [q_k^{1-c_0}, q_{k+1}^{1-c_0}]$.

Therefore by gluing the scales together, we obtain exponential decay of $|\phi_y|$ on the whole $\Z$.

\section{Weak Liouville scales}\label{sec:non_exp}
Throughout this section, we assume that $q_{n+1}\leq e^{\delta_1 q_n}$, except in Lemma \ref{lem:LDT_non_exp_m=qn}. This is what we called a weak Liouville scale.
Our goal of this section is to prove exponential decay of the generalized eigenfunction, roughly speaking on the scale of $[q_n/10, q_{n+1}/10]$.
\begin{theorem}\label{thm:AL_weak_resonant}
If $q_{n+1}\leq e^{\delta_1 q_n}$, then for 
\begin{align}
|y|\in 
    \begin{cases}
        [q_n^{1-c_0}, q_{n+1}/10], \text{ if } q_n\geq e^{\delta_1 q_{n-1}}\\
        [q_n/10, q_{n+1}/10], \text{ if } q_n\leq e^{\delta_1 q_{n-1}},
    \end{cases}
\end{align} 
we have 
$$|\phi_y|\leq e^{-L(\omega,E) |y|/40}.$$
\end{theorem}
\begin{remark}
It follows from a standard argument that iterating the Green's function expansion in \eqref{eq:Green_exp} leads to a sharp decay
$$|\phi_y|\leq e^{-(L(\omega,E)-CC_v\delta_1^{1/4})|y|},$$
for some absolute constant $C>0$. We leave this for interested readers.
\end{remark}
\begin{proof} 
In the proof, we shall sometimes write $L(\omega,E,\varepsilon)$ as $L(\varepsilon)$ and $L_m(\omega,E,\varepsilon)$ as $L_m(\varepsilon)$ for simplicity.
Without loss of generality, we assume $y>0$. 

Recall that $g_{m,E}^{\omega}$ is as in \eqref{def:g}.
We will in fact bound the number of zeros of $g_{m,E}^{\omega}(z)-e^{2m(L_m-\delta)}$ near the unit circle, for small constant $\delta>0$, and prove a large deviation estimate for the following function:
\begin{align}
    v_{m,E}(\theta):=\frac{1}{2m}\log (g_{m,E}^{\omega}(e^{2\pi i\theta})).
\end{align}
Let 
\begin{align}\label{def:B^g_m}
    \mathcal{B}^g_{m,\delta,E}:=\{\theta\in \T: v_{m,E}(\theta)<L_m(\omega,E)-\delta\}.
\end{align}
Note that for $\theta\in \T$, $g_{m,E}^{\omega}(e^{2\pi i\theta})=\|M_{m,E}^{\omega}(\theta)\|_{\mathrm{HS}}^2>0$, hence
\begin{align}\label{eq:B^g_m>0}
    \mathcal{B}^g_{m,\delta,E}=\{\theta\in \T: 0<g_{m,E}^{\omega}(\theta)<e^{2m(L_m(\omega,E)-\delta)}\}.
\end{align}
Therefore each connected component (which is an interval) of $\mathcal{B}^g_{m,\delta,E}$ has two endpoints in $\{\theta\in \T: g_{m,E}^{\omega}(\theta)=e^{2m(L_m-\delta)}\}$. 
Thus the number of connected components of $\mathcal{B}^g_{m,\delta,E}$ is controlled by half of the number of zeros of $g_{m,E}^{\omega}(z)-e^{2m(L_m-\delta)}$.

\subsection{Large deviation estimates}
\begin{lemma}\label{lem:LDT_non_exp}
Let $C_v$ be as in \eqref{def:C_v}.
    Assume $q_{n+1}\leq e^{\delta_1 q_n}$. Then for any $m$ such that $10 q_n \leq m\leq q_{n+1}/5$, with $n$ large enough, we have
    \begin{align}
        \mathrm{mes}(\mathcal{B}^g_{m,1500C_v \delta_1^{1/2},E})\leq e^{-100\delta_1 m}.
    \end{align}
\end{lemma}
As we mentioned, the large deviation estimate can be proved by using almost shift invariance of $v_{m,E}(\theta)$, or the equivalence between $v_{m,E}(\theta)$ and $m^{-1}\log \|M_{m,E}^{\omega}(\theta)\|$.
This proof is close to that in \cite{HZ}, which has a similar weak Liouville condition (in which it was assumed that, roughly speaking, each $q_{n+1}\leq e^{\delta_1 q_n}$ holds for all $n$). We postpone it to Sec. \ref{sec:LDT_non_exp}.

We also prove the following new large deviation estimates, roughly speaking for $m$ of size comparable to $q_n$. The novelty is that it does not require the weak Liouville assumption $q_{n+1}\leq e^{\delta_1 q_n}$.
\begin{lemma}\label{lem:LDT_non_exp_m=qn}
Let $C_v$ be as in \eqref{def:C_v}.
    For any $m$ such that 
    \begin{align}
    10 q_n\geq m\geq 
    \begin{cases}
        5\delta_1^{1/4} q_n^{1-c_0}, \text{ if } q_n\geq e^{\delta_1 q_{n-1}}\\
        5\delta_1^{1/4} q_n, \text{ if } q_n\leq e^{\delta_1 q_{n-1}},
    \end{cases}
    \end{align}
    with $n$ large enough, we have
    \begin{align}
        \mathrm{mes}(\mathcal{B}^g_{m,1000C_v \delta_1^{1/4},E})\leq e^{-100\delta_1 m}.
    \end{align}
\end{lemma}
This lemma also plays an important role in our study of the eigenfunction in the weakly resonant regimes of the strong Liouville sales, e.g. $[\nu q_n, (1-\nu)q_n]$ for small $\nu>0$, at a scale $q_n$ that $q_{n+1}\geq e^{\delta_1 q_n}$. 
We present the proof of Lemma \ref{lem:LDT_non_exp_m=qn} in Sec. \ref{sec:LDT_non_exp_m=qn}.

Next, we turn to the zero count.

\subsection{Zero count}

\ 

Below and throughout the rest of the section, $\kappa:\equiv\kappa(\omega,E)$. 
\begin{lemma}\label{lem:zero_gm}
Let $\eta$ be as in \eqref{def:eta}.
    Let $\mathcal{N}^g_{m,\delta,E}(\varepsilon):=\#\{z\in \mathcal{A}_{e^{2\pi\varepsilon}}: g_{m,E}^{\omega}(z)=e^{2m(L_m(\omega,E)-\delta)}\}$. 
    For any $\delta\geq 1100C_v\delta_1^{1/4}$, $\varepsilon_1=2\eta/(1+2\eta)\varepsilon_0$, and $m$ large enough (depending on $\varepsilon_0,\delta,\eta$) satisfying
    \begin{align}\label{eq:range_m}
    m\in 
        \begin{cases}
            [10q_n, q_{n+1}/5], \text{ if } q_{n+1}\leq \delta_1 q_n,\\
            [5\delta_1^{1/4}q_n^{1-c_0}, 10q_n], \text{ if } q_n\geq e^{\delta_1 q_{n-1}},\\
            [5\delta_1^{1/4}q_n, 10q_n], \text{ if } q_n\leq e^{\delta_1 q_{n-1}},
        \end{cases}
    \end{align}
    we have $\mathcal{N}^g_{m,\delta,E}(\varepsilon_1)\leq 4(1+\eta) \kappa m$.
\end{lemma}
Before proving this lemma, we give a quick corollary.
\begin{corollary}\label{cor:gm_zeros}
Under same condition as Lemma \ref{lem:zero_gm}, the large deviation set  satisfies:
    \begin{align}
        \mathcal{B}^g_{m,1100C_v\delta_1^{1/4},E}=\bigcup_{j=1}^{\tilde{\mathcal{N}}} U_{m,j},
    \end{align}
    with $\tilde{\mathcal{N}}\leq 2(1+\eta)\kappa m$ and $\{U_{m,j}\}_{j=1}^{\tilde{N}}$ are disjoint intervals satisfying 
    \begin{align}\label{eq:Umj_measure_small}
        \mathrm{mes}(U_{m,j})\leq e^{-100\delta_1 m}, 
    \end{align}
    for large enough $m$ satisfying \eqref{eq:range_m}.
\end{corollary}
We give a quick proof of this corollary.
\begin{proof}
    Each interval of $\mathcal{B}^g_{m,1100C_v\delta_1^{1/4},E}$ has two distinct endpoints, both lie in $\{\theta: g_{m,E}^{\omega}(e^{2\pi i\theta})=e^{2m(L_m(\omega,E)-1100C_v\delta_1^{1/4})}\}$.
    Hence the number of intervals is controlled by 
    $$\frac{1}{2}\mathcal{N}_{m,1100C_v\delta_1^{1/4},E}^g(\varepsilon_1)\leq 2(1+\eta)\kappa m.$$
    The measure estimates follow from Lemmas \ref{lem:LDT_non_exp} and \ref{lem:LDT_non_exp_m=qn}.
\end{proof}
Next, we prove Lemma \ref{lem:zero_gm}
\begin{proof}
Within the proof we shall sometimes write $L_m(\omega,E)$ as $L_m$ for simplicity.
Let $R:=e^{2\pi\varepsilon_0}$ and $N:=\mathcal{N}^g_{m,\delta,E}(\varepsilon_0)$. Let $w_1,...,w_N$ be the zeros of $g_{m,E}^{\omega}(z)-e^{2m(L_m-\delta)}$ in $\mathcal{A}_R$ (assuming it is zero free on $\partial \mathcal{A}_R$, otherwise shrink $\varepsilon_0$ to $\varepsilon_0-o(1)$. We omit this small technical adjustment).
Define 
\begin{align}\label{def:GRqn}
G_{R,m,E}(z)=\frac{1}{2m}\sum_{k=1}^{N} G_R(z, w_k),
\end{align}
where $G_R$ is the Green's function in \eqref{def:G}.
Let
\begin{align*}
\tilde{v}_{m,E}(z):=\frac{1}{2m}\log |g_{m,E}^{\omega}(z)-e^{2m(L_m-\delta)}|=2\pi G_{R,m,E}(z)+h_{R,m,E}(z).
\end{align*}
Then $h_{R,m,E}(z)$ is harmonic in $\mathcal{A}_R$.
Let
\begin{align}\label{def:Iqn_GR_2}
   L^{\tilde{v}}_{m}(E,\varepsilon):=&\int_{\T} \tilde{v}_{m,E}(e^{2\pi i(\theta+i\varepsilon)})\, \mathrm{d}\theta, \text{ and }\\
   I^G_{m}(E,\varepsilon):=&\int_{\T} 2\pi G_{R,m,E}(e^{2\pi i(\theta+i\varepsilon)})\, \mathrm{d}\theta \text{ and }\\
   I^h_{m}(E,\varepsilon):=&\int_{\T} h_{R,m,E}(e^{2\pi i(\theta+i\varepsilon)})\, \mathrm{d}\theta
\end{align}
Note that $L^{\tilde{v}}_m(E,\varepsilon)\neq L_m(E,\varepsilon)$, but we will show their difference is negligible in Lemma \ref{lem:Ltv=Lv}.
We first show $I^h_m(E,\varepsilon)$ is constant in $\varepsilon$.
\begin{lemma}\label{lem:Ih_constant}
    There exists $b\in \R$ such that $I^h_{m}(E,\varepsilon)\equiv b$ for $|\varepsilon|\leq \varepsilon_0$.
\end{lemma}
\begin{proof}
We first prove $L^{\tilde{v}}_{m}(E,\varepsilon)$, $I^G_{m}(E,\varepsilon)$ and $I^h_{m}(E,\varepsilon)$ are all even functions in $\varepsilon$.
Similar to the proof of Fact \ref{fact:zero_real}, one observes that since each determinant $P_{k,E}^{\omega}(\theta)$ is real-valued for $\theta\in \T$, we have
\begin{align}
    P_{k,E}^{\omega}(z)=\overline{P_{k,E}^{\omega}(1/\overline{z})}.
\end{align}
This implies
\begin{align}\label{eq:gm_conjugate}
    &\overline{g_{m,E}^{\omega}(1/\overline{z})}\\
    =&(\overline{P_{m,E}^{\omega}(1/\overline{z})})^2+(\overline{P_{m-1,E}^{\omega}(1/\overline{z})})^2+(\overline{P_{m-1,E}^{\omega}(e^{2\pi i\omega}/\overline{z})})^2+(\overline{P_{m-2,E}^{\omega}(e^{2\pi i\omega}/\overline{z})})^2\\
    =&(\overline{P_{m,E}^{\omega}(1/\overline{z})})^2+(\overline{P_{m-1,E}^{\omega}(1/\overline{z})})^2+(\overline{P_{m-1,E}^{\omega}(1/(\overline{ze^{2\pi i\omega}}))})^2+(\overline{P_{m-2,E}^{\omega}(1/\overline{(ze^{2\pi i\omega}}))})^2\\
    =&(P^{\omega}_{m,E}(z))^2+(P^{\omega}_{m-1,E}(z))^2+(P^{\omega}_{m-1,E}(ze^{2\pi i\omega}))^2+(P^{\omega}_{m-2,E}(ze^{2\pi i\omega}))^2=g^{\omega}_{m,E}(z).
\end{align}
Hence $\tilde{v}_{m,E}(e^{2\pi i(\theta+i\varepsilon)})=\tilde{v}_{m,E}(1/e^{2\pi i(-\theta+i\varepsilon)})=\tilde{v}_{m,E}(e^{2\pi i(\theta-i\varepsilon)})$, which implies
\begin{align}\label{eq:Lm_even}
    L_{m}^{\tilde{v}}(E,\varepsilon)=L_{m}^{\tilde{v}}(E,-\varepsilon).
\end{align}
By \eqref{eq:gm_conjugate}, if $z$ is a solution to $g_{m,E}^{\omega}(z)=e^{2m(L_m-\delta)}$, then $g_{m,E}^{\omega}(1/\overline{z})=e^{2m(L_m-\delta)}$.
Hence $1/\overline{z}$ is also a solution, which implies $\bigcup_{k=1}^N\{w_k\}=\bigcup_{k=1}^N\{1/\overline{w}_k\}$.
Therefore, by \eqref{eq:GR_even},
\begin{align}
    G_{R,m,E}(z)=\frac{1}{2m}\sum_{k=1}^N G_R(z,1/\overline{w_k})=\frac{1}{2m}\sum_{k=1}^N G_R(1/\overline{z},w_k)=G_{R,m,E}(1/\overline{z}).
\end{align}
This implies
\begin{align}\label{eq:IG_even}
    I^G_{m}(E,\varepsilon)=I^G_{m}(E,-\varepsilon).
\end{align}
Combining \eqref{eq:Lm_even} with \eqref{eq:IG_even}, we obtain
\begin{align}\label{eq:Ih_even}
    I^h_{m}(E,\varepsilon)=I^h_{m}(E,-\varepsilon).
\end{align}
For any $r\in [1/R,R]$, and $|x+iy|=r$, we define a radial function as follows:
\begin{align}
\tilde{h}(x+iy):=I^h_{m}(E,\frac{\log r}{2\pi})=\int_{\T}h_{R,m,E}(re^{2\pi i\theta})\, \mathrm{d}\theta.
\end{align}
Since $h_{R,m,E}$ is harmonic, $\tilde{h}$ is a radial harmonic function, which implies
\begin{align}
    \tilde{h}(x+iy)=a\log |x+iy|+b=I_{m}^h(E,\frac{\log |x+iy|}{2\pi}),
\end{align}
for some constants $a,b$.
Then, by \eqref{eq:Ih_even},
\begin{align}
    I_{m}^h(E,\varepsilon)=2\pi a\varepsilon+b=-2\pi a\varepsilon+b=I_{m}^h(E,-\varepsilon),
\end{align}
which implies $a=0$. Hence
\begin{align}\label{eq:Ih_constant}
    I^h_{m}(E,\varepsilon)\equiv b, \text{ for any } |\varepsilon|\leq \varepsilon_0,
\end{align}
as claimed.
\end{proof}

We further use the following lemma to bound $L^{\tilde{v}}_{m}(E,0)$ from below.
\begin{lemma}\label{lem:Ltv=Lv}
Under the same condition as Lemma \ref{lem:zero_gm}.
For $\delta\geq 1100C_v \delta_1^{1/4}$, for $m$ large and satisfying \eqref{eq:range_m}, we have
\begin{align}
    L^{\tilde{v}}_m(E,0)\geq L_m(E,0)-1200C_v\delta_1^{1/4}.
\end{align}
\end{lemma}
\begin{proof}
Recall the large deviation estimates of $v_{m,E}$ in Lemmas \ref{lem:LDT_non_exp} and \ref{lem:LDT_non_exp_m=qn}:
\begin{align}\label{eq:bad_measure}
    \mathrm{mes}(\mathcal{B}^g_{m,1000C_v\delta_1^{1/4},E})=\mathrm{mes}(\{\theta:|v_{m,E}(e^{2\pi i\theta})-L_{m}(E,0)|>1000C_v\delta_1^{1/4}\})\leq e^{-100\delta_1 m}.
\end{align}
For $\theta\notin \mathcal{B}^g_{m,1000C_v\delta_1^{1/4},E}$, the following is true
\begin{align}
    e^{2m(L_{m}(E,0)-1000C_v\delta_1^{1/4})}\leq g_{m,E}^{\omega}(e^{2\pi i\theta})\leq e^{2m(L_{m}(E,0)+1000C_v\delta_1^{1/4})},
\end{align}
which implies for $\delta\geq 1100\delta_1^{1/4}$ that
\begin{align}
    \frac{1}{2}e^{2m(L_{m}(E,0)-1000C_v\delta_1^{1/4})}\leq g_{m,E}^{\omega}(e^{2\pi i\theta})-e^{2m(L_m(E,0)-\delta)}\leq 2e^{2m(L_{m}(E,0)+1000C_v\delta_1^{1/4})}.
\end{align}
Hence by \eqref{eq:bad_measure},
\begin{align}\label{eq:tB_mes}
    \mathrm{mes}(\tilde{B}_m):=&\mathrm{mes}(\{\theta: |\tilde{v}_{m,E}(e^{2\pi i\theta})-L_{m}(E,0)|>1100C_v\delta_1^{1/4}\})\\
    \leq &\mathrm{mes}(\mathcal{B}^g_{m,1000C_v\delta_1^{1/4},E})
    \leq e^{-100\delta_1 m}.
\end{align}
This yields a preliminary $L^2$ estimate as below:
\begin{align}\label{eq:tv_L2}
    \|\tilde{v}_{m,E}(e^{2\pi i\theta})-L_{m}(E,0)\|_{L^2(\T)}\leq \tilde{C}_{v,\delta_1,\varepsilon_0}.
\end{align}
To see this, one covers the unit circle $\mathcal{C}_1$ with $\sim \varepsilon_0^{-1}$ many disks of radius $\sim \varepsilon_0$, with centers in $\tilde{B}_m^c$, and then applies the Cartan estimate (Lemma \ref{lem:Cartan}) to each of these disks.

With \eqref{eq:tB_mes} and \eqref{eq:tv_L2} in hands, we obtain
\begin{align}
    &\|\tilde{v}_{m,E}(e^{2\pi i\theta})-L_{m}(E,0)\|_{L^1(\T)}\\
    =&\int_{\tilde{B}_m}|\tilde{v}_{m,E}(e^{2\pi i\theta})-L_{m}(E,0)|\, \mathrm{d}\theta+\int_{\tilde{B}_m^c}|\tilde{v}_{m,E}(e^{2\pi i\theta})-L_{m}(E,0)|\, \mathrm{d}\theta\\
    \leq &(\mathrm{mes}(\tilde{B}_m))^{1/2}\cdot  \|\tilde{v}_m(e^{2\pi i\theta})-L_{m}(E,0)\|_{L^2(\T)}+1100C_v\delta_1^{1/4}\\
    \leq &\tilde{C}_{v,\delta_1,\varepsilon_0}e^{-50\delta_1 m}+1100C_v\delta_1^{1/4}.
\end{align}
Hence 
\begin{align}
    \int_{\T} \tilde{v}_{m,E}(e^{2\pi i\theta})\, \mathrm{d}\theta\geq L_{m}(E,0)-1200C_v\delta_1^{1/4},
\end{align}
as claimed.
\end{proof}

Next, we present the upper bound of $L^{\tilde{v}}_m(E,\varepsilon)$.
\begin{lemma}\label{lem:L_tv_upper}
    For any $\varepsilon\in [0,\varepsilon_0]$, for $m$ large enough, the following is true
    \begin{align}
        L^{\tilde{v}}_{m}(E,\varepsilon)\leq L_m(E,\varepsilon)+o(1).
    \end{align}
\end{lemma}
\begin{proof}
    The following standard uniform upper bound is from Lemma \ref{lem:upper_semi_cont}:
\begin{align}
\frac{1}{m}\log \|M_{m,E}^{\omega}(\theta+i\varepsilon_2)\|\leq L_{m}(E,\varepsilon_2)+o(1), \text{ uniformly in } \theta\in \T.
\end{align}
It implies that uniformly in $\theta\in \T$:
\begin{align}
    |g_{m,E}^{\omega}(e^{2\pi i(\theta+i\varepsilon_2)})|\leq \|M_{m,E}^{\omega}(\theta+i\varepsilon_2)\|_{\mathrm{HS}}^2\leq Ce^{2m(L_{m}(E,\varepsilon_2)+o(1))},
\end{align}
for some absolute constant $C>0$ that arises from the equivalence between the two norms of $2\times 2$ matrices.
Hence, by using that $L_{m}(E,0)\leq L_{m}(E,\varepsilon_2)$ (which follows from the convexity of $L_m(E,\cdot)$), we obtain
\begin{align}
    |g_{m,E}^{\omega}(e^{2\pi i(\theta+i\varepsilon_2)})-e^{2m(L_{m}(E,0)-\delta)}|\leq 2Ce^{2m(L_{m}(E,\varepsilon_2)+o(1))}.
\end{align}
Hence uniformly in $\theta\in \T$, we have
\begin{align}
    \tilde{v}_{m,E}(e^{2\pi(\theta+i\varepsilon_2)})\leq L_{m}(E,\varepsilon_2)+o(1),
\end{align}
This implies the claimed result.
\end{proof}

With these preparations, we now turn to the zero count. The arguments below are very similar to those in \cite{HS1}.
By \cite[(4.24)]{HS1}, the following holds
\begin{align}\label{eq:int_final}
I_{m}^G(E,\varepsilon_1)=-\frac{\pi}{2m} \int_{\varepsilon_1}^{\varepsilon_0} \mathcal{N}^g_{m,\delta,E}(\varepsilon_2)\, \mathrm{d}\varepsilon_2.
\end{align}
Hence for $\varepsilon_0\geq \varepsilon_2>\varepsilon_1>0$,
\begin{align}
    I^G_{m}(E,\varepsilon_2)-I^G_{m}(E,\varepsilon_1)=\frac{\pi}{2m}\int_{\varepsilon_1}^{\varepsilon_2}\mathcal{N}^g_{m,\delta,E}(\varepsilon_3)\,\mathrm{d}\varepsilon_3\geq \frac{\pi}{2m}(\varepsilon_2-\varepsilon_1)\mathcal{N}^g_{m,\delta,E}(\varepsilon_1),
\end{align}
in which the last inequality follows from the monotonicity of $\mathcal{N}^g_{m,\delta,E}(\cdot)$.
Combining the inequality above with Lemma \ref{lem:Ih_constant}, we obtain
\begin{align}\label{eq:Ltv_1-2}
    \frac{\pi}{2m}(\varepsilon_2-\varepsilon_1)\mathcal{N}^g_{m,\delta,E}(\varepsilon_1)\leq L^{\tilde{v}}_{m}(E,\varepsilon_2)-L^{\tilde{v}}_{m}(E,\varepsilon_1).
\end{align}
By convexity of $L^{\tilde{v}}_{m}(E,\cdot)$ and Lemma \ref{lem:Ltv=Lv}, we have
\begin{align}\label{eq:Ltv_1}
    L^{\tilde{v}}_{m}(E,\varepsilon_1)\geq L^{\tilde{v}}_{m}(E,0)\geq L_{m}(E,0)-1200C_v\delta_1^{1/4}. 
\end{align}

Therefore, combining \eqref{eq:Ltv_1}, Lemma \ref{lem:L_tv_upper} with \eqref{eq:Ltv_1-2}, yields
\begin{align}
    \frac{\pi}{2m}(\varepsilon_2-\varepsilon_1)\mathcal{N}^g_{m,\delta,E}(\varepsilon_1)
    \leq &L_{m}(E,\varepsilon_2)-L_{m}(E,0)+1200 C_v\delta_1^{1/4}+o(1).
\end{align}
This implies for $m$ large enough, 
\begin{align}
\frac{\pi}{2m}(\varepsilon_2-\varepsilon_1)\mathcal{N}^g_{m,\delta,E}(\varepsilon_1)
    \leq &L(E,\varepsilon_2)-L(E,0)+1200 C_v\delta_1^{1/4}+o(1)\\
    \leq &2\pi \kappa \varepsilon_2+1200 C_v\delta_1^{1/4}+o(1).
\end{align}
Hence 
\begin{align}
    \mathcal{N}^g_{m,\delta,E}(\varepsilon_1)\leq \frac{4m\kappa}{\varepsilon_2-\varepsilon_1}\left(\varepsilon_2+200 C_v\delta_1^{1/4}+o(1)\right).
\end{align}
Recall that $\eta$ is as in \eqref{def:eta}.
Taking $\varepsilon_2=((\eta/2)^{-1}+1)\varepsilon_1=\varepsilon_0$ and $m$ large enough yields the desired estimates.
This finishes the proof of Lemma \ref{lem:zero_gm}.
\end{proof}

Next, we explore some unique features of the even potentials.
In fact, when $v$ is even, the Dirichlet determinants satisfy:
\begin{align}\label{eq:Pk_even_1}
    P_{k,E}^{\omega}(\theta-\frac{k-1}{2}\omega)=P_{k,E}^{\omega}(-\theta-\frac{k-1}{2}\omega), \text{ for any } k\in \N,
\end{align}
which implies for arbitrary $m\geq 3$ that 
\begin{align}
    \begin{cases}
        P_{m-2,E}^{\omega}(\theta-\frac{m-1}{2}\omega+\omega)=P_{m-2,E}^{\omega}(-\theta-\frac{m-1}{2}\omega+\omega), \text{ and } \\
        (P_{m-1,E}^{\omega}(\theta-\frac{m-1}{2}\omega))^2+(P_{m-1,E}^{\omega}(\theta-\frac{m-1}{2}\omega+\omega))^2\\ \qquad\qquad\qquad=(P_{m-1,E}^{\omega}(-\theta-\frac{m-1}{2}\omega+\omega))^2+(P_{m-1,E}^{\omega}(-\theta-\frac{m-1}{2}\omega))^2
    \end{cases}
\end{align}
Therefore,
\begin{align}
    g_{m,E}^{\omega}(e^{2\pi i(\theta-\frac{m-1}{2}\omega)})=g_{m,E}^{\omega}(e^{2\pi i(-\theta-\frac{m-1}{2}\omega)}).
\end{align}
Combining this with Corollary \ref{cor:gm_zeros}, we conclude that:
\begin{corollary}\label{cor:gm_zeros_even}
Under the same conditions as Corollary \ref{cor:gm_zeros}, we have
\begin{align}
   \mathcal{B}^g_{m,1100C_v\delta_1^{1/4},E}=\bigcup_{j=1}^{\tilde{N}'} (U_{m,j}\bigcup(-U_{m,j}-(m-1)\omega)),
\end{align}
for some $\tilde{N}'\leq \kappa(1+\eta)m$, furthermore each $U_{m,j}$ satisfies $\mathrm{mes}(U_{m,j})\leq e^{-100\delta_1 m}$.
\end{corollary}

\subsection{Exponential decay of generalized eigenfunction}

\ 

Throughout the rest of the section, we assume $\kappa(\omega,E)=1$, namely restrict to the first supercritical stratum.

For any $\Z\ni y\geq 10$, let $m_y:=[y/2]$, and 
\begin{align}
    I^y_1:=\left[-\left[\frac{9}{10}m_y\right], -\left[\frac{3}{5}m_y\right]\right]\cap \Z \text{ and } 
    I^y_2:=\left[y-\left[\frac{9}{10}m_y\right], y-\left[\frac{1}{10}m_y\right]\right]\cap \Z.
\end{align}
Since $y\geq 2m_y$, $I_1^y\cap I_2^y=\emptyset$. Also it is clear that 
\begin{align}\label{eq:card_I1cupI2}
\mathrm{card}(I_1\bigcup I_2)=\mathrm{card}(I_1)+\mathrm{card}(I_2)\geq \frac{11}{10}m_y+O(1)>(1+\eta)m_y,
\end{align}
where we used that $\eta<1/100$ as in \eqref{def:eta}.

Let 
\begin{align}
    \mathcal{I}_n:=\begin{cases}
        [10\delta_1^{1/4}q_n, q_{n+1}/5], \text{ if } q_{n+1}\leq e^{\delta_1 q_n}, \text{ and } q_n\leq e^{\delta_1 q_{n-1}}\\
        [10\delta_1^{1/4} q_n^{1-c_0}, q_{n+1}/5], \text{ if } q_{n+1}\leq e^{\delta_1 q_n}, \text{ and } q_n>e^{\delta_1 q_{n-1}}   
    \end{cases}
\end{align}
Corollary \ref{cor:gm_zeros_even} together with the weak Liouville condition $q_{n+1}\leq e^{\delta_1 q_n}$ implies the following.
\begin{lemma}\label{lem:I2_good_non_exp}
For $y\in \mathcal{I}_n$ with $n$ large enough. For each $k\in I^y_1$, we have
\begin{align}\label{eq:I1_bad_non_exp}
    \theta+k\omega\in \mathcal{B}^g_{m_y,1100C_v\delta_1^{1/4},E}.
\end{align} Furthermore there exists $k_2^y\in I^y_2$ such that 
$$\theta+k_2^y\omega\notin \mathcal{B}^g_{m_y,1100C_v\delta_1^{1/4},E}.$$
\end{lemma}
\begin{proof}
We prove \eqref{eq:I1_bad_non_exp} by contradiction.  Suppose there exists $k_1\in I^y_1$ such that 
$$\theta+k_1\omega\notin \mathcal{B}^g_{m_y,1100 C_v\delta_1^{1/4},E}.$$ Then 
$g_{m_y,E}^{\omega}(\theta+k_1\omega)\geq e^{2m_y(L_{m_y}-1100 C_v\delta_1^{1/4})}$.
This implies, by the definition of $g_{m_y,E}^{\omega}$ as in \eqref{def:g}, that there exists $m_y'\in \{m_y,m_y-1,m_y-2\}$ and $a\in \{0,1\}$ such that 
\begin{align}\label{eq:g_large_implies_P_large}
    |P_{m_y',E}^{\omega}(\theta+(k_1+a)\omega)|\geq \frac{1}{2}e^{m_y(L_{m_y}-1100C_v\delta_1^{1/4})}\geq \frac{1}{2}e^{m_y(L-1200C_v\delta_1^{1/4})},
\end{align}
where we used $L_{m_y}=L+o(1)$ for $y$ large enough (which implies $m_y$ is large enough).
Applying Green's function expansion \eqref{eq:Green_exp} of $\phi_h$ at $h=0, -1$, on the interval $[m_1,m_2]=[k_1+a,k_1+a+m_y'-1]$, and estimating the numerators of the expansion using Lemma \ref{lem:upper_semi_cont}, we arrive at a contradiction that
\begin{align}\label{eq:exp_phi0_weak}
    \qquad 1\leq(\max(|\phi_0|,|\phi_{-1}|)\leq e^{-|k_1|(L-2000 C_v\delta_1^{1/4})}|\phi_{k_1+a-1}|+e^{-|k_1+m_y'|(L-2000C_v\delta_1^{1/4})}|\phi_{k_1+a+m_y'}|\leq e^{-\frac{1}{20}m_y L},
\end{align}
where we used \eqref{eq:Shnol} and that $\min(|k_1|, |k_1+m_y'|)\geq m_y/10-O(1)$ in the last inequality. 
This proves \eqref{eq:I1_bad_non_exp}.

Note that \eqref{eq:I1_bad_non_exp} implies for each $k\in I_1$, there exists $j_k$ such that
\begin{align}
    \theta+k\omega\in (U_{m_y,j_k}\bigcup (-U_{m_y,j_k}-(m_y-1)\omega)).
\end{align}
We need the following repulsion property among $\theta+k\omega$ for $k\in I_1^y\bigcup I_2^y$.
\begin{lemma}\label{lem:theta+komega_notin Ujk1}
    Suppose for some $k\in I_1^y\bigcup I_2^y$, $\theta+k\omega\in (U_{m_y,j}\bigcup (-U_{m_y,j}-(m_y-1)\omega))$, then for any $k'\in I_1^y\bigcup I_2^y\setminus \{k\}$, the following holds:
    \begin{align}
        \theta+k'\omega\notin (U_{m_y,j}\bigcup (-U_{m_y,j}-(m_y-1)\omega)).
    \end{align}
\end{lemma}
\begin{proof}
Without loss of generality, we assume $\theta+k\omega\in U_{m_y,j_k}$. The other case is completely analogous.

We need to distinguish two cases, depending on the size of $y$. 

\underline{Case 1}. If $y\leq q_n/2$. Then $m_y\leq q_n/4$. Clearly this implies
\begin{align}
    0<|k-k'|<y+\frac{4}{5}m_y+2<q_n.
\end{align}
Hence by \eqref{eq:qn_omega_min} and \eqref{eq:qn_omega}, we have
\begin{align}
    \|(\theta+k\omega)-(\theta+k'\omega)\|\geq \|(k-k')\omega\|\geq \|q_{n-1}\omega\|\geq \frac{1}{2q_n}\geq e^{-100\delta_1 m_y}, 
\end{align}
where we used $m_y\geq 5\delta_1\delta_1^{1/4}q_n^{1-c_0}$ in the last inequality. Combining the above with the estimate of $U_{m_y,j_k}$ in Corollary \ref{cor:gm_zeros_even} yields $\theta+k'\omega\notin U_{m_y,j_k}$.

\underline{Case 2}. If $y\geq q_n/2$. Then $m_y\geq q_n/5$. 
Since $k'\neq k$ and $|k-k'|<q_{n+1}$, we have by \eqref{eq:qn_omega_min} and \eqref{eq:qn_omega} that
    \begin{align}\label{eq:k-k1}
        \|\theta+k'\omega-(\theta+k\omega)\|=\|(k'-k)\omega\|\geq \|q_n\omega\|\geq e^{-\delta_1 q_n}/2,
    \end{align}
where we used $q_{n+1}\leq e^{\delta_1 q_n}$ in the last inequality.
Combining this with the measure estimate in Corollary \ref{cor:gm_zeros_even} and that $m_y\geq q_n/5$ yields
$$\mathrm{mes}(U_{m_y,j_k})\leq e^{-100\delta_1 m_y}\leq e^{-20\delta_1 q_n}\ll e^{-\delta_1 q_n}/2,$$
This implies $\theta+k'\omega\notin U_{m_y,j_k}$.
Now it remains to show $\theta+k'\omega\notin (-U_{m_y,j_k}-(m_y-1)\omega)$.
Note that
\begin{align}
        k+k'+m_y-1\in 
        \begin{cases}
        [-\frac{4}{5}m_y, -\frac{1}{5}m_y], \text{ if } k,k'\in I_1^y,\\
        [y-\frac{4}{5}m_y,y+\frac{3}{10}m_y], \text{ if } \text{ exactly one of } k,k'\in I_1^y,\\
        [2y-\frac{4}{5}m_y, 2y+\frac{4}{5}m_y] \text{ if } k,k'\in I_2^y.
        \end{cases}
    \end{align}
    Since $y=2m_y+O(1)$, $|k+k'+m_y-1|\in [m_y/5, 5m_y]$, thus \eqref{eq:theta_non_res} with $\delta'=\delta_1$ implies
    \begin{align}
        \|\theta+k'\omega-(-(\theta+k\omega)-(m_y-1)\omega)\|=\|2\theta+(k+k'+m_y-1)\omega\|
        \geq &e^{-\delta_1|k+k'+m_y-1|}\\
        \geq &e^{-5\delta_1 m_y}\geq \mathrm{mes}(U_{m_y,j_{k_1}}).
    \end{align}
    Hence
    $$\theta+k' \omega\notin -U_{m_y,j}-(m_y-1)\omega\ni -(\theta+k\omega)-(m_y-1)\omega.$$
    This proves the claimed result of Lemma \ref{lem:theta+komega_notin Ujk1}.
    \end{proof}
    Lemma \ref{lem:I2_good_non_exp} follows from combining Lemma \ref{lem:theta+komega_notin Ujk1} with \eqref{eq:card_I1cupI2} and the pigeonhole principle.
\end{proof}
For 
\begin{align}
    y\in 
    \begin{cases}
    [q_n^{1-c_0}, q_{n+1}/10], \text{ if } q_n\geq e^{\delta_1 q_{n-1}}\\
    [\frac{q_n}{10}, \frac{q_{n+1}}{10}], \text{ if } q_n\leq e^{\delta_1 q_{n-1}},
    \end{cases}
\end{align}
expanding $\phi_y$ using the Green's function expansion on $[k^y_2, k^y_2+m_y-1]$ with $k_2^y$ provided by Lemma \ref{lem:theta+komega_notin Ujk1}, we have similar to \eqref{eq:exp_phi0_weak} that
\begin{align}
    |\phi_y|
    \leq \max(e^{-|y-k_2^y|(L-2000C_v\delta_1^{1/4})}|\phi_{k_2^y-1}|, e^{-|k_2^y+m_y-y|(L-2000C_v\delta_1^{1/4})}|\phi_{k_2^y+m_y}|)
    \leq e^{-Ly/40}, 
\end{align}
in which we used $\min(|y-k_2^y|, |k_2^y+m_y-y|)\geq m_y/10\geq y/20$, and used \eqref{eq:Shnol} to bound 
$$\max(|\phi_{k_2^y-1}|, |\phi_{k_2^y+m_y}|)\leq Cy.$$
This is the claimed result.
\end{proof}

\section{Strong Liouville scales}\label{sec:strong}
We first give an overview of this section.
Throughout this section, we assume $q_{n+1}\geq e^{\delta_1 q_n}$. 
Our goal is to study the decay of the generalized eigenfunction $\phi$, roughly speaking, in $[q_n/10, q_{n+1}^{1-c_0}]$. 
The main result of this section is Theorem \ref{thm:AL_main_beta}, which is based on Theorems \ref{thm:AL_weak_regime} and \ref{thm:AL_lqn}, that handle weakly resonant regimes and strongly resonant regimes, respectively.
The strongly resonant regimes $R_{\ell q_n}$, for $|\ell|\leq 10q_{n+1}^{1-c_0}/q_n$, as defined as follows.
If $q_n\leq e^{\delta_1 q_{n-1}}$, let 
\begin{align}\label{def:Rr_1}
    R_{\ell q_n}:=[(\ell-10\delta_1^{1/4})q_n, (\ell+10\delta_1^{1/4})q_n], \text{ and } r_{\ell q_n}:=\sup_{y\in R_{\ell q_n}}|\phi_y|.
\end{align}
If $q_n>e^{\delta_1 q_{n-1}}$, let 
\begin{align}\label{def:Rr_2}
    R_{\ell q_n}:=[\ell q_n-10\delta_1^{1/4} q_n^{1-c_0}, \ell q_n+10\delta_1^{1/4} q_n^{1-c_0}], \text{ and } r_{\ell q_n}:=\sup_{y\in R_{\ell q_n}}|\phi_y|.
\end{align}
A regime in between two consecutive strong ones: $[\ell q_n, (\ell+1)q_n]\setminus (R_{\ell q_n}\bigcup R_{(\ell+1)q_n})$, is called a weakly resonant regime.

The technical core of this section concerns the strongly resonant regimes, for which the study of the weakly resonant regimes serve as preparations.
In fact, sections \ref{sec:zero_fpq}, \ref{sec:zero_fomega}, \ref{sec:even}, \ref{sec:weak_regime} are all preparations for the proof of Theorem \ref{thm:AL_lqn} in Sec. \ref{sec:ef_strong_strong}.
The proof of Theorem \ref{thm:AL_weak_regime}, given in Sec. \ref{sec:weak_regime}, is indeed independent of the sections \ref{sec:zero_fpq}, \ref{sec:zero_fomega} and \ref{sec:even}, but we decide to present it next to Sec. \ref{sec:ef_strong_strong} since they both concern the study of eigenfunctions.

This section deals with the case when there exists a sequence of such strong Liouville scales $q_{n+1}\geq e^{\delta_1 q_n}$, which is true when $\beta(\omega)\geq \delta_1>0$. 
In such case, we can formula the following variant of \eqref{eq:theta_non_res}: 
for any small $\delta'>0$, there exists $c_{\delta'}>0$ such that
\begin{align}\label{eq:theta_non_res_2}
    \|2\theta+n\omega\|\geq c_{\delta'}\, e^{-\delta' |n|}.
\end{align}
Note that this follows from \eqref{eq:theta_non_res} unless $2\theta+n_0\omega\in \Z$ for some $n_0\in \Z$. 
However that would lead to a contradiction to $\theta\in \Theta$. In fact, we would have
\begin{align}
    0=\limsup_{n\to\infty}\frac{-\log \|2\theta+n\omega\|}{|n|}=\limsup_{n\to\infty}\frac{-\log \|(n-n_0)\omega\|}{|n|}=\beta(\omega)\geq \delta_1>0,
\end{align}
hence a contradiction.
We will use this variant \eqref{eq:theta_non_res_2} in Sections \ref{sec:even} and \ref{sec:ef_strong_strong}.

Define $\beta_n$ be such that $\|q_n\omega\|=e^{-\beta_n q_n}$. By \eqref{eq:qn_omega}, we have $\beta_n\geq \delta_1$.
Recall $f^{\omega}_{n,E}$ is roughly speaking the trace of $M^{\omega}_{n,E}$, as in \eqref{def:f=2-tr}.
Let $$v^{\omega}_{q_n,E}(z)=q_n^{-1}\log |f^{\omega}_{q_n,E}(z)|,$$
and $R:=e^{2\pi\varepsilon_0}$, $N:=N_{q_n}(\omega,E,\varepsilon_0)$. Let $w_1,...,w_N$ be the zeros of $f^{\omega}_{q_n,E}(z)$ in $\mathcal{A}_R$ (assume $f^{\omega}_{q_n,E}(z)$ is zero free on $\partial \mathcal{A}_R$).
Define 
\begin{align}\label{def:GRqn_2}
G_{R,q_n,E}^{\omega}(z)=\frac{1}{q_n}\sum_{k=1}^{N} G_R(z, w_k),
\end{align}
where $G_R$ is the Green's function of $\mathcal{A}_R$ as in \eqref{def:G}. Then 
\begin{align}\label{eq:vomega=G+h}
v_{q_n,E}^{\omega}(z)=2\pi G_{R,q_n,E}^{\omega}(z)+h_{R,q_n,E}^{\omega}(z),
\end{align}
where $h_{R,q_n,E}^{\omega}$ is a harmonic function on $\mathcal{A}_R$.

\subsection{Zeros of $f^{p_n/q_n}_{q_n,E}$}\label{sec:zero_fpq}
We start off by studying the zeros of $f^{\omega}$, for $\omega=p_n/q_n$, which are structured due to periodicity. 

Since the trace of $M_{q_n,E}^{p_n/q_n}$ is periodic:
\begin{align}\label{eq:det_Mqn_rational_periodic}
\mathrm{tr}(M_{q_n,E}^{p_n/q_n}(\theta))=\mathrm{tr}(M_{q_n,E}^{p_n/q_n}(\theta+p_n/q_n)),
\end{align} 
$f_{q_n,E}^{p_n/q_n}(e^{2\pi i\cdot})$ is also $1/q_n$-periodic, due to \eqref{def:f=2-tr},
which implies
\begin{fact}\label{fact:zero_periodic}
If $z_0\in \C$ is a zero of $f_{q_n,E}^{p_n/q_n}(z)$, namely $f_{q_n,E}^{p_n/q_n}(z_0)=0$, then for any $1\leq j\leq q_n-1$, $z_0e^{2\pi ij p_n/q_n}$ is also a zero.
\end{fact}

Next, we estimate the total number of zeros of $f_{q_n,E}^{p_n/q_n}(z)$ lying near the unit circle.
For any $0\leq \varepsilon\leq \varepsilon_0$, define
\begin{align}
    N_n(\omega,E,\varepsilon):=\#\{e^{-2\pi\varepsilon}\leq |z|\leq e^{2\pi\varepsilon}:\, z \text{ is a zero of } f_{n,E}^{\omega}(z)\}.
\end{align}
Fact \ref{fact:zero_periodic} yields immediately that:
\begin{corollary}\label{cor:dqn_multiple}
For any $\varepsilon\geq 0$,
$N_{q_n}(p_n/q_n,E,\varepsilon)$ is a multiple of $q_n$.
\end{corollary}
Let
\begin{align}\label{def:Iqn_GR}
   I_{q_n}^{p_n/q_n,G}(E,\varepsilon):=\int_{\T} 2\pi G_{R,q_n,E}^{p_n/q_n}(e^{2\pi i(\theta+i\varepsilon)})\, \mathrm{d}\theta
\end{align}
be the integral of the Green's function, and
\begin{align}\label{def:Lv_qn}
    L_{q_n}^{p_n/q_n,v}(E,\varepsilon):= \int_{\T} v_{q_n,E}^{p_n/q_n}(e^{2\pi i(\theta+i\varepsilon)})\, \mathrm{d}\theta.
\end{align}
\begin{lemma}\label{lem:zero_count_rational}
The harmonic part of $v^{p_n/q_n}_{q_n,E}$ as in \eqref{eq:vomega=G+h} satisfies $h_{R,q_n,E}^{p_n/q_n}=v_{q_n,E}^{p_n/q_n}$ on $\partial \mathcal{A}_R$, and the following holds uniformly in $z\in \mathcal{A}_R$:
\begin{align}\label{eq:harmonic=constant}
    h_{R,q_n,E}^{p_n/q_n}(z)=L(\omega,E,\varepsilon_0)+o(1).
\end{align}
For any small $\delta>0$, for $n$ large enough (depending on $\delta,\varepsilon_0$), the zero count satisfies
\begin{align}
    N_{q_n}(p_q/q_n,E,\varepsilon_0/2)=N_{q_n}(p_n/q_n,E,\delta/2)=2q_n.
\end{align}
\end{lemma}

\begin{proof}
We start with the harmonic part.
Since for $\varepsilon\in \{\pm\varepsilon_0, \pm\varepsilon_0/2,\pm\delta/2, \pm\delta/4\}$, $(\omega,M_E(\cdot+i\varepsilon))$ is regular, Theorem \ref{thm:dominated} implies that for $n$ large enough and uniformly in $\theta\in \T$, 
\begin{align}
    \frac{1}{q_n}\log (\rho(M_{q_n,E}^{p_n/q_n}(\theta+i\varepsilon)))=L(\omega,E,\varepsilon)+o(1),
\end{align}
where $\rho(M)$ is the spectral radius of $M$.
This implies, due to $\rho(M)\leq \|M^2\|^{1/2}\leq \|M\|$, that
\begin{align}\label{eq:lower_M2_qn}
    &\frac{1}{q_n}\min\left(\log \|M_{q_n,E}^{p_n/q_n}(\theta+i\varepsilon)\|, \frac{1}{2}\log \|(M_{q_n,E}^{p_n/q_n}(\theta+i\varepsilon))^2\|\right)\\
    \geq &\frac{1}{q_n}\log (\rho(M_{q_n,E}^{p_n/q_n}(\theta+i\varepsilon)))\notag\\
    \geq &L(\omega,E,\varepsilon)+o(1),
\end{align}
uniformly in $\theta\in \T$, for $n$ large enough.

By Lemma \ref{lem:upper_semi_cont} and \eqref{eq:f=M2/M}, we conclude that for $n$ large enough,
\begin{align}\label{eq:upper_M_qn}
    v_{q_n,E}^{p_n/q_n}(e^{2\pi i(\theta+i\varepsilon)})\leq \frac{1}{q_n}\log \|M_{q_n,E}^{p_n/q_n}(\theta+i\varepsilon)\|+o(1)\leq L(\omega,E,\varepsilon)+o(1),
\end{align}
uniformly in $\theta\in \T$.
Note that \eqref{eq:lower_M2_qn} and \eqref{eq:upper_M_qn} together with \eqref{eq:f=M2/M} imply, uniformly in $\theta\in \T$, that
\begin{align}\label{eq:vqn_lower_eps}
    v_{q_n,E}^{p_n/q_n}(e^{2\pi i(\theta+i\varepsilon)})\geq \frac{1}{q_n}\log \frac{\|(M_{q_n,E}^{p_n/q_n}(\theta+i\varepsilon))^2\|}{\|M_{q_n,E}^{p_n/q_n}(\theta+i\varepsilon)\|}+o(1)\geq L(\omega,E,\varepsilon)+o(1).
\end{align}
Combining \eqref{eq:upper_M_qn} with \eqref{eq:vqn_lower_eps} yields
\begin{align}\label{eq:vqn_uniform=L}
    v^{p_n/q_n}_{q_n,E}(e^{2\pi i(\theta+i\varepsilon)})=L(\omega,E,\varepsilon)+o(1),
\end{align}
for $\varepsilon\in \{\pm\varepsilon_0, \pm\varepsilon_0/2,\pm\delta/2, \pm\delta/4\}$, uniformly in $\theta\in \T$.
This clearly implies the integral satisfies:
\begin{align}\label{eq:Lpq=L+o1}
    L^{p_n/q_n,v}_{q_n}(E,\varepsilon)=L(\omega,E,\varepsilon)+o(1), \text{ for } \varepsilon\in \{\pm\varepsilon_0, \pm\varepsilon_0/2,\pm\delta/2, \pm\delta/4\}.
\end{align}

Since $h_{R,q_n,E}^{p_n/q_n}(z)=v^{p_n/q_n}_{q_n,E}(z)$ for $z\in \partial \mathcal{A}_R$, by \eqref{eq:vqn_uniform=L} and the max/min principle for harmonic functions,
\begin{align}\label{eq:harmonic}
    h_{R,q_n,E}^{p_n/q_n}(z)=L(\omega,E,\varepsilon_0)+o(1),
\end{align}
holds uniformly in $z\in \mathcal{A}_R$, thus proving \eqref{eq:harmonic=constant}.

Next, we estimate the number of zeros.
By \eqref{eq:harmonic=constant} and \eqref{eq:Lpq=L+o1}, we have for any $$\varepsilon_2,\varepsilon_1\in \{\varepsilon_0, \varepsilon_0/2,\delta/2, \delta/4\},$$
and for $n$ large enough that
\begin{align}\label{eq:Iqn_eps2-eps1}
    I_{q_n}^{p_n/q_n,G}(E,\varepsilon_2)-I_{q_n}^{p_n/q_n,G}(E,\varepsilon_1)
    =&L^{p_n/q_n,v}_{q_n}(E,\varepsilon_2)-L^{p_n/q_n,v}_{q_n}(E,\varepsilon_1)+o(1) \notag\\
    =&L(\omega,E,\varepsilon_2)-L(\omega,E,\varepsilon_1)+o(1).
\end{align}
By \cite[(4.24)]{HS1}, 
\begin{align}\label{eq:424_HS2}
    I_{q_n}^{p_n/q_n,G}(E,\varepsilon)=2\pi\int_{\T} G_{R,q_n,E}^{p_n/q_n}(e^{2\pi i(\theta+i\varepsilon)})\, \mathrm{d}\theta=-\frac{\pi}{q_n}\int_{\varepsilon}^{\varepsilon_0} N_{q_n}(p_n/q_n,E,\varepsilon)\, \mathrm{d}\varepsilon.
\end{align}
Combining \eqref{eq:Iqn_eps2-eps1} with \eqref{eq:424_HS2}, and using the $L(\omega,E,\varepsilon)=L(\omega,E,0)+2\pi\varepsilon$, we have
\begin{align}
\int_{\varepsilon_0/2}^{\varepsilon_0}N_{q_n}(p_n/q_n,E,\varepsilon)\, \mathrm{d}\varepsilon
=&\frac{q_n}{\pi}(L(\omega,E,\varepsilon_0)-L(\omega,E,\varepsilon_0/2)+o(1))\\
=&q_n\varepsilon_0(1+o(1)).
\end{align}
This implies, by the monotonicity of $N_{q_n}(p_n/q_n,E,\varepsilon)$ in $\varepsilon$, that
\begin{align}
    N_{q_n}(p_n/q_n,E,\varepsilon_0/2)\leq 2q_n(1+o(1)).
\end{align}
By Corollary \ref{cor:dqn_multiple}, we have
\begin{align}\label{eq:Nq<2q}
    N_{q_n}(p_n/q_n,E,\varepsilon_0/2)\leq 2q_n.
\end{align}
Similarly
\begin{align}
    \int_{\delta/4}^{\delta/2}N_{q_n}(p_n/q_n,E,\varepsilon)\, \mathrm{d}\varepsilon
    =&\frac{q_n}{\pi} (L(\omega,E,\delta)-L(\omega,E,\delta/2)+o(1))\\
    =&\frac{q_n\delta}{2}(1+o(1)),
\end{align}
which implies
\begin{align}
    N_{q_n}(p_n/q_n,E,\delta/2)\geq 2q_n(1+o(1)),
\end{align}
This combined with Corollary \ref{cor:dqn_multiple} yields
\begin{align}
    N_{q_n}(p_n/q_n,E,\delta/2)\geq 2q_n.
\end{align}
Taking the upper bound \eqref{eq:Nq<2q} into account, we conclude that for $n$ large enough,
\begin{align}\label{eq:Ntau=Neps0=2dqn}
    N_{q_n}(p_n/q_n,E,\delta/2)=N_{q_n}(p_n/q_n,E,\varepsilon_0/2)=2q_n,
\end{align}
as claimed.
\end{proof}

Facts \ref{fact:zero_real}, \ref{fact:zero_periodic} and Lemma \ref{lem:zero_count_rational} yield the following immediately:
\begin{lemma}\label{lem:ration_zero_expression}
    There exists $z_1^{(p_n/q_n)}=r_1e^{2\pi i(\theta_1+i\varepsilon_1)},z_2^{(p_n/q_n)}=r_2e^{2\pi i(\theta_2+i\varepsilon_2)}\in \mathcal{A}_{\exp(\pi\delta)}$ (it is possible that $z_1^{(p_n/q_n)}=z_2^{(p_n/q_n)}$), with $r_1r_2=1$ and $r_2\leq r_1$, such that the zeros of $f_{q_n,E}^{p_n/q_n}(z)$ in $\mathcal{A}_{\exp(\pi\varepsilon_0)}$ are 
    \begin{align}\label{def:zero_set_rational}
    \mathcal{Z}_{q_n}(p_n/q_n,E):=\bigcup_{j=0}^{q_n-1}\{z_1^{(p_n/q_n)}e^{2\pi i jp_n/q_n}, z_2^{(p_n/q_n)} e^{2\pi i jp_n/q_n}\}.
    \end{align}
\end{lemma}
The periodic structure of zeros implies the following control of $\Gamma_R$ in \eqref{def:G}.
Let 
\begin{align}\label{def:Gamma_R_qn}
    \Gamma^{p_n/q_n}_{R,q_n,E}(z):=\frac{1}{q_n}\sum_{w\in \mathcal{Z}_{q_n}(p_n/q_n,E)} \Gamma_R(z,w).
\end{align}
\begin{lemma}\label{lem:Gamma_qn}
    For any $\delta>0$, for $n$ large enough, uniformly in $z\in \mathcal{A}_R$, the following holds:
    \begin{align}
        \left|\Gamma^{p_n/q_n}_{R,q_n,E}(z)+\frac{\log (|z|R)}{2\pi}\right|\leq \delta.
    \end{align}
\end{lemma}
\begin{proof}
For each $s=1,2$, we study
    \begin{align}
        \frac{1}{q_n}\sum_{j=0}^{q_n-1} \Gamma_R(z, z^{(p_n/q_n)}_s e^{2\pi ijp_n/q_n}).
    \end{align}
    Recall that $\Gamma_R(z,w)$ is as in \eqref{def:GammaR}, there exists some $k_0=k_0(\delta)$ such that uniformly in $z\in \mathcal{A}_R$ and $w\in \mathcal{A}_{\exp(\pi\delta)}$,
    \begin{align}
    \Gamma_R(z,w)=\frac{\log( |z|/ R) \log (|w|/ R)}{4\pi\log R}+\frac{1}{2\pi}\log \left( \frac{\prod_{k=1}^{k_0} |1-\frac{1}{R^{4k}}\frac{z}{w}| \cdot |1-\frac{1}{R^{4k}} \frac{w}{z}|}{R\cdot \prod_{k=1}^{k_0} |1-\frac{1}{R^{4k-2}}w\overline{z}|\cdot |1-\frac{1}{R^{4k-2}} \frac{1}{\overline{z}w}|}\right)+\xi_1,
    \end{align}
    where $|\xi_1|\leq \delta/4$.
    It is easy to show, for each $1\leq k\leq k_0$ that uniformly in $z\in \mathcal{A}_R$,
    \begin{align}\label{eq:GammaR_11}
        &\frac{1}{q_n}\sum_{j=0}^{q_n-1}\log \left|1-\frac{1}{R^{4k}}\frac{z}{z^{(p_n/q_n)}_s e^{2\pi ijp_n/q_n}}\right|\notag\\
        =&\frac{1}{q_n}\sum_{j=0}^{q_n-1}\log \left|1-\frac{1}{R^{4k}}\frac{z}{r_s e^{2\pi i(\theta_s+j/q_n)}}\right|\notag\\
        =&\int_{\T} \log \left|1-\frac{1}{R^{4k}}\frac{z}{r_se^{2\pi i\theta}}\, \mathrm{d}\theta\right|+O\left(\frac{|z|}{R^{4k}r_s-|z|}\right)\frac{1}{q_n},
    \end{align}
    and similarly
    \begin{align}
        &\frac{1}{q_n}\sum_{j=0}^{q_n-1}\log \left|1-\frac{1}{R^{4k}}\frac{z^{(p_n/q_n)}_s e^{2\pi ijp_n/q_n}}{z}\right|=\int_{\T}\left|1-\frac{1}{R^{4k}}\frac{r_se^{2\pi i\theta}}{z}\, \mathrm{d}\theta\right|+O\left(\frac{r_s}{R^{4k}|z|-r_s}\right)\frac{1}{q_n},\\
        &\frac{1}{q_n}\sum_{j=0}^{q_n-1}\log \left|1-\frac{1}{R^{4k-2}}z^{(p_n/q_n)}_s e^{2\pi ijp_n/q_n}\overline{z}\right|=\int_{\T}\left|1-\frac{1}{R^{4k-2}}r_se^{2\pi i\theta}\overline{z}\, \mathrm{d}\theta\right|+O\left(\frac{r_s|z|}{R^{4k-2}-r_s|z|}\right)\frac{1}{q_n},\\
        &\frac{1}{q_n}\sum_{j=0}^{q_n-1}\log \left|1-\frac{1}{R^{4k-2}}\frac{1}{\overline{z} z^{(p_n/q_n)}_s e^{2\pi ijp_n/q_n}}\right|=\int_{\T}\left|1-\frac{1}{R^{4k-2}}\frac{1}{\overline{z}r_se^{2\pi i\theta}}\, \mathrm{d}\theta\right|+O\left(\frac{1}{R^{4k-2}r_s|z|-1}\right)\frac{1}{q_n}.
    \end{align}
    Therefore, for each $1\leq k\leq k_0$, the sums can be approximated, uniformly in $z\in \mathcal{A}_R$, by the integrals. This implies for $q_n$ large enough, depending on $R,\delta$, that 
    \begin{align}\label{eq:GammaR_1}
        \delta/2 \geq &\left|\frac{1}{q_n}\sum_{j=0}^{q_n-1} \Gamma_R(z, z^{(p_n/q_n)}_s e^{2\pi ijp_n/q_n})
        -\int_{\T}\Gamma_R(z, r_s e^{2\pi i\theta})\, \mathrm{d}\theta\right|\\
        =&\left|\frac{1}{q_n}\sum_{j=0}^{q_n-1} \Gamma_R(z, z^{(p_n/q_n)}_s e^{2\pi ijp_n/q_n})-\frac{\log (|z|/R)}{4\pi \log R}\log(r_s/R)+\frac{\log R}{2\pi}\right|,
    \end{align}
    where we used \eqref{eq:int_HR} in the last line.
    This implies, by triangle inequality and $\log r_1+\log r_2=0$, that
    \begin{align}
        \left|\Gamma^{p_n/q_n}_{R,q_n,E}(z)+\frac{\log (|z|R)}{2\pi}\right|\leq \delta,
    \end{align}
    as claimed.
\end{proof}
Combining the control of $h^{p_n/q_n}_{R,q_n,E}$ in Lemma \ref{lem:zero_count_rational} with the control of $\Gamma^{p_n/q_n}_{R,q_n,E}$ in Lemma \ref{lem:Gamma_qn}, we have that for some $|\xi|\leq 2\delta$, 
\begin{align}\label{eq:vqn_Adelta_1}
    v^{p_n/q_n}_{q_n,E}(z)
    =&2\pi G^{p_n/q_n}_{R,q_n,E}(z)+h^{p_n/q_n}_{R,q_n,E}(z)\notag\\
    =&\frac{1}{q_n}\left(\sum_{w\in \mathcal{Z}_{q_n}(p_n/q_n,E)}\log |z-w|\right)-\log (|z|R)+L(\omega,E,\varepsilon_0)+\xi\notag\\
    =&\frac{1}{q_n}\left(\sum_{w\in \mathcal{Z}_{q_n}(p_n/q_n,E)}\log |z-w|\right)-\log |z|+L(\omega,E)+\xi,
\end{align}
where we used $L(\omega,E,\varepsilon_0)=L(\omega,E)+2\pi\varepsilon_0=L(\omega,E)+\log R$ in the last line.

The following lemma controls the sum of logarithmic potential part via its minimum term, in \eqref{eq:vqn_Adelta_1} above.
\begin{lemma}\label{lem:log_min}
    For $s\in \{1,2\}$, the following holds uniformly in $z\in \mathcal{A}_R$:
\begin{align}\label{eq:vqn_Adelta_11}
    \left|q_n^{-1}\sum_{s=1}^2 \sum_{j=0}^{q_n-1}\log |z-z^{(p_n/q_n)}_s e^{2\pi ijp_n/q_n}|-q_n^{-1}\sum_{s=1}^2 \min_{j=0}^{q_n-1}\log |z-z^{(p_n/q_n)}_s e^{2\pi ijp_n/q_n}|\right.\\
    \left.-\sum_{s=1}^2 \int_{\T}\log |z-r_s e^{2\pi i\theta}|\, \mathrm{d}\theta\right|\leq \delta.
\end{align}
\end{lemma}
\begin{proof}
Fix an arbitrary $s\in \{1,2\}$.
Let $z=r_ze^{2\pi i\theta_z}$, where $r_z>0$ and $\theta_z\in \T$. 
We partition $\T$ as follows:
\begin{align}
    \T=\bigcup_{k=-[q_n/2]}^{q_n-[q_n/2]-1} \left[\theta_z-\frac{1}{2q_n}+\frac{k}{q_n}, \theta_z+\frac{1}{2q_n}+\frac{k}{q_n}\right)=:Q_k.
\end{align}
For each $k\in [-[q_n/2], q_n-[q_n/2]-1]$, there exists a unique $j_k$ such that $\theta_s+j_kp_n/q_n\in Q_k$.
Clearly $j_0$ is such that
\begin{align}
    \log |z-z_s^{(p_n/q_n)}e^{2\pi i j_0p_n/q_n}|=\min_{j=0}^{q_n-1}\log |z-z_s^{(p_n/q_n)}e^{2\pi i j p_n/q_n}|.
\end{align}
To control the non-minimum terms, for each $k\in [-[q_n/2], q_n-[q_n/2]-1]\setminus \{0\}$ and $\theta\in Q_k$, we have that
\begin{align}
    |\theta-(\theta_s+j_kp_n/q_n)|\leq 1/q_n, \text{ and } \min(|\theta_z-(\theta_s+j_kp_n/q_n)|, |\theta_z-\theta|)\geq (2|k|-1)/q_n.
\end{align}
This implies
\begin{align}
    &q_n^{-1}\log |z-z_s^{(p_n/q_n)}e^{2\pi ij_kp_n/q_n}|-\int_{Q_k}\log |z-r_se^{2\pi i\theta}|\, \mathrm{d}\theta\\
    =&
    \int_{Q_k}\log \left|1+\frac{r_s (e^{2\pi i\theta}-e^{2\pi i(\theta_s+j_kp_n/q_n)})}{z-r_se^{2\pi i\theta}}\right|\, \mathrm{d}\theta\\
    \leq &q_n^{-1} \sup_{\theta\in Q_k} \left|\frac{r_s(e^{2\pi i\theta}-e^{2\pi i(\theta_s+j_kp_n/q_n)})}{z-r_se^{2\pi i\theta}}\right|\\
    \leq &q_n^{-1}\frac{\pi |\theta-(\theta_s+j_kp_n/q_n)|}{2|\theta_z-\theta|}\leq q_n^{-1}\frac{\pi}{4|k|-2}.
\end{align}
Summing up in $k$ and $s$, we conclude that for $q_n$ large enough that \eqref{eq:vqn_Adelta_11} holds as claimed.
\end{proof}
By Jensen's formula we have that, recall that $r_1r_2=1$ and $r_2\leq r_1$,
\begin{align}\label{eq:vqn_Adelta_12}
    \sum_{s=1}^2 \int_{\T} \log |z-r_s e^{2\pi i\theta}|\, \mathrm{d}\theta=
    \begin{cases}
        0\, \text{ if } |z|<r_2\\
        \log (r_1|z|)\, \text{ if } r_2\leq |z|\leq r_1\\
        2\log |z|, \text{ if } |z|>r_1
    \end{cases}
\end{align}
Combining \eqref{eq:vqn_Adelta_1}, \eqref{eq:vqn_Adelta_11}, \eqref{eq:vqn_Adelta_12} with the fact that $1\leq r_1\leq e^{\pi\delta}$ (due to Lemma \ref{lem:ration_zero_expression}) yields:
\begin{corollary}\label{cor:vqn_Adelta_2}
    For any $z\in \mathcal{A}_{\exp(2\pi\delta)}$, for $n$ large enough,
\begin{align}\label{eq:vqn_Adelta_2}
    \left|v_{q_n,E}^{p_n/q_n}(z)-\frac{1}{q_n}\left(\sum_{s=1}^2 \min_{j=0}^{q_n-1}\log |z-z^{(p_n/q_n)}_s e^{2\pi ijp_n/q_n}|\right)-L(\omega,E)\right|\leq 10\delta,
\end{align}
and for any $z\in \mathcal{A}_{\exp(2\pi\delta)}\setminus \left(\bigcup_{s=1}^2 \bigcup_{j=0}^{q_n-1} B_{\exp(-\delta q_n)}(z_s^{(p_n/q_n)}e^{2\pi ijp_n/q_n}))\right)$, 
\begin{align}
    v^{p_n/q_n}_{q_n,E}(z)\geq L(\omega,E)-12\delta.
\end{align}
\end{corollary}

Note that by Lemma \ref{lem:ration_zero_expression}, $\mathcal{Z}_{q_n}(p_n/q_n,E)\subset \mathcal{A}_{\exp(\pi\delta)}$, hence
\begin{align}\label{eq:balls_in_Adelta}
    \bigcup_{s=1}^2 \bigcup_{j=0}^{q_n-1} B_{4\exp(-\delta q_n)}(z_s^{(p_n/q_n)}e^{2\pi ijp_n/q_n}))\subset \mathcal{A}_{\exp(2\pi \delta)}.
\end{align}

\subsection{Zeros of $f^{\omega}_{q_n,E}$}\label{sec:zero_fomega}
Next, we study the zeros of $f^{\omega}_{q_n,E}$ for the irrational frequency $\omega$.
In the following, we write $L(\omega,E)$ as $L$ for simplicity.
We will show:
\begin{lemma}\label{lem:ap_zero_omega}
There exists $z^{\omega}_{q_n,j,s}\in \mathcal{A}_{\exp(2\pi\delta)}$, $j\in \{0,...,q_n-1\}$ and $s=1,2$ such that the set $\mathcal{Z}_{q_n}(\omega,E)$ of zeros of $f^{\omega}_{q_n,E}$ in $\mathcal{A}_{\exp(2\pi\delta)}$ is given by  
\begin{align}\label{eq:Zqn_omega}
\mathcal{Z}_{q_n}(\omega,E)=\bigcup_{s=1}^2\bigcup_{j=0}^{q_n-1}\{z^{\omega}_{q_n,j,s}\}
\end{align}
Furthermore, the zeros are structured (almost $1/q_n$-periodic), in the sense that for any $j,k\in \{0,...,q_n-1\}$, the following holds
\begin{align}\label{eq:r_omega_periodic}
    z^{\omega}_{q_n,j,s}=z^{\omega}_{q_n,k,s}e^{2\pi i(j-k)p_n/q_n}+O(e^{-\delta q_n}).
\end{align}
If we denote $z^{\omega}_{q_n,j,s}:=r_{j,s}^{\omega} e^{2\pi i\theta_{j,s}^{\omega}}$ for $s=1,2$ and $j\in \{0,1,...,q_n-1\}$, then 
\begin{align}\label{eq:r_theta_periodic}
    r_{j,s}^{\omega}=r_{k,s}^{\omega}+O(e^{-\delta q_n}) \text{ and } \theta_{j,s}^{\omega}=\theta_{k,s}^{\omega}+\frac{(j-k)p_n}{q_n}+O(e^{-\delta q_n}),
\end{align}
for $j,k\in \{0,...,q_n-1\}$ and $s=1,2$. 
\end{lemma}
\begin{proof}
We distinguish into two different cases:

\underline{Case 1.} If $z_1^{(p_n/q_n)}\in B_{2\exp(-\delta q_n)}(z_2^{(p_n/q_n)}e^{2\pi ij_0p_n/q_n})$ for some $j_0\in \{0,1,...,q_n-1\}$. Let $$B_{q_n,0}:=B_{4\exp(-\delta q_n)}(z_2^{(p_n/q_n)}e^{2\pi ij_0p_n/q_n}),$$ and $B_{q_n,j}:=B_{q_n,0}e^{2\pi ijp_n/q_n}$.
It is clear from Corollary \ref{cor:vqn_Adelta_2} that 
\begin{align}\label{eq:fqn_lower_111}
    |f_{q_n,E}^{p_n/q_n}(z)|\geq e^{(L-12\delta)q_n}, \text{ for any } z\in \mathcal{A}_{\exp(2\pi\delta)}\setminus(\bigcup_{j=0}^{q_n-1}B_{q_n,j}).
\end{align}
By \eqref{eq:f=M+M}, Lemma \ref{lem:upper_semi_cont} and the standard telescoping argument, for $q_{n+1}\geq e^{30\delta q_n}$, we have uniformly in $\mathcal{A}_{\exp(2\pi\delta)}$ that
\begin{align}\label{eq:telescoping}
    |f^{\omega}_{q_n,E}(z)-f^{p_n/q_n}_{q_n,E}(z)|\leq 2\|M^{\omega}_{q_n,E}(z)-M^{p_n/q_n}_{q_n,E}(z)\|
    \leq &e^{(L(\omega,E,\delta)+\delta)q_n}\|q_n\omega\|\\
    \leq &e^{(L+2\pi\delta+\delta)q_n}\|q_n\omega\|\\
    \leq &e^{(L-20\delta)q_n},
\end{align}
where we used $\|q_n\omega\|\leq q_{n+1}^{-1}\leq e^{-30\delta q_n}$ (see \eqref{eq:qn_omega}). 
Combining this with \eqref{eq:fqn_lower_111} yields the following for $z\notin \bigcup_{j=0}^{q_n-1}B_{q_n,j}\subset \mathcal{A}_{\exp(2\pi \delta)}$:
\begin{align}\label{eq:fomega-fpq}
    |f^{\omega}_{q_n,E}(z)-f^{p_n/q_n}_{q_n,E}(z)|\leq \frac{1}{2}|f^{p_n/q_n}_{q_n,E}(z)|.
\end{align}
Rouche's theorem, applied to each ball $B_{q_n,j}$, implies the two analytic functions $f^{\omega}_{q_n,E}$ and $f^{p_n/q_n}_{q_n,E}$ have the same number of zeros in each ball $B_{q_n,j}$. Thus $f^{\omega}_{q_n,E}$ has exactly two zeros, denoted by $z^{\omega}_{q_n,j,1}, z^{\omega}_{q_n,j,2}$, in each $B_{q_n,j}$, $j=0,...,q_n-1$.
Furthermore, $f^{\omega}_{q_n,E}$ has no zero in $\mathcal{A}_{\exp(2\pi\delta)}\setminus(\bigcup_{j=0}^{q_n-1}B_{q_n,j})$.

\underline{Case 2.} If $z_1^{(p_n/q_n)}\notin \bigcup_{j=0}^{q_n-1}B_{2\exp(-\delta q_n)}(z_2^{(p_n/q_n)}e^{2\pi ijp_n/q_n})$. For each $s=1,2$, and $j\in \{0,1,...,q_n-1\}$, let $$B_{q_n,j}^{s}:=B_{\exp(-\delta q_n)}(z^{(p_n/q_n)}_s  e^{2\pi ijp_n/q_n}).$$ 
It is clear that in this case $B^s_{q_n,j_1}\cap B^{s'}_{q_n,j_2}=\emptyset$ for any $s\neq s'$ and any $j_1,j_2\in \{0,1,...,q_n-1\}$.
By Corollary \ref{cor:vqn_Adelta_2} we obtain that 
\begin{align}
    |f_{q_n,E}^{p_n/q_n}(z)|\geq e^{(L-12\delta)q_n}, \text{ for any } z\in \mathcal{A}_{\exp(2\pi\delta)}\setminus (\bigcup_{s=1}^2\bigcup_{j=0}^{q_n-1}B^s_{q_n,j}).
\end{align}
Similar to the Case 1 above, for $q_{n+1}\geq e^{30\delta q_n}$, we have
\begin{align}
    |f_{q_n,E}^{\omega}(z)-f_{q_n,E}^{p_n/q_n}(z)|\leq \frac{1}{2}|f_{q_n,E}^{p_n/q_n}(z)|,
\end{align}
for any $z\in \mathcal{A}_{\exp(2\pi\delta)}\setminus (\bigcup_{s=1}^2 \bigcup_{j=0}^{q_n-1} B_{q_n,j}^s)$.
Thus $f^{\omega}_{q_n,E}$ has exactly one zero, denoted by $z^{\omega}_{q_n,j,s}$, in each $B_{q_n,j}^s$, $j=0,...,q_n-1$, $s=1,2$, and it has no zero in $\mathcal{A}_{\exp(2\pi\delta)}\setminus (\bigcup_{s=1}^2 \bigcup_{j=0}^{q_n-1} B_{q_n,j}^s)$.

The zeros in both two cases clearly satisfy \eqref{eq:r_omega_periodic} and thus \eqref{eq:r_theta_periodic}.
\end{proof}
The zeros of the strong Liouville scale, similar to its rational approximation, are structured (almost $1/q_n$-periodic) as in \eqref{eq:r_omega_periodic}. 
Therefore, similar to Corollary \ref{cor:vqn_Adelta_2}, we obtain:
\begin{lemma}\label{lem:vqn_Adelta_3_omega}
Let 
\begin{align}\label{eq:fix_delta=delta1/30}
\delta=\delta_1/30.
\end{align}
For $n$ large enough and $q_{n+1}\geq e^{\delta_1 q_n}$.
    For any $z\in \mathcal{A}_{\exp(2\pi\delta)}$,
    \begin{align}
        \left|v^{\omega}_{q_n,E}(z)-\frac{1}{q_n}\left(\sum_{s=1}^2 \min_{j=0}^{q_n-1}\log |z-z^{\omega}_{q_n,j,s}|\right)-L(\omega,E)\right|\leq 10\delta,
    \end{align}
    in which $\mathcal{Z}_{q_n}(\omega,E)=\bigcup_{s=1}^2\bigcup_{j=0}^{q_n-1}\{z^{\omega}_{q_n,j,s}\}$ is the set of zeros of $f^{\omega}_{q_n,E}$ in $\mathcal{A}_{\exp(2\pi\delta)}$.
    Furthermore, for any $z\in \mathcal{A}_{\exp(2\pi\delta)}\setminus \left(\bigcup_{s=1}^2 \bigcup_{j=0}^{q_n-1} B_{\exp(-\delta q_n)}(z^{\omega}_{q_n,j,s})\right)$,
    \begin{align}
        v^{\omega}_{q_n,E}(z)\geq L(\omega,E)-12\delta.
    \end{align}
\end{lemma}
\begin{proof}
We discuss the proof briefly. 
Defining $\Gamma^{\omega}_{R,q_n,E}$ similar to \eqref{def:Gamma_R_qn} as
\begin{align}
    \Gamma^{\omega}_{R,q_n,E}(z):=\frac{1}{q_n}\sum_{w\in \mathcal{Z}_{q_n}(\omega,E)}\Gamma_R(z,w).
\end{align}
Then due to the almost periodicity of $\mathcal{Z}_{q_n}(\omega,E)$ provided by Lemma \ref{lem:ap_zero_omega}, similar to Lemma \ref{lem:Gamma_qn}, the following estimates hold uniformly in $z\in \mathcal{A}_R$,
\begin{align}
    \left|\Gamma^{\omega}_{R,q_n,E}(z)+\frac{\log (|z|R)}{2\pi}\right|\leq \delta.
\end{align}
Next, we show an analogue of \eqref{eq:harmonic=constant} holds as follows:
\begin{align}\label{eq:h=constant_omega}
    h^{\omega}_{R,q_n,E}(z)=L(\omega,E,\varepsilon_0)+o(1), \text{ uniformly in } z\in \mathcal{A}_R.
\end{align}
By \eqref{eq:harmonic=constant}, for $z\in \partial\mathcal{A}_R$ the following holds:
\begin{align}\label{eq:h-h_1}
    |f^{p_n/q_n}_{q_n,E}(z)|=e^{q_n(L(\omega,E,\varepsilon_0)+o(1))}.
\end{align}
By the telescoping argument as in \eqref{eq:telescoping}, we have for $z\in \partial \mathcal{A}_R$ that
\begin{align}\label{eq:h-h_2}
    |f^{p_n/q_n}_{q_n,E}(z)-f^{\omega}_{q_n,E}(z)|\leq 2\|M^{p_n/q_n}_{q_n,E}(z)-M^{\omega}_{q_n,E}(z)\|
    \leq &e^{(L(\omega,E,\varepsilon_0)+\delta)q_n}\|q_n\omega\|\\
    \leq &e^{(L(\omega,E,\varepsilon_0)-29\delta)q_n},
\end{align}
for $q_{n+1}\geq e^{\delta_1 q_n}=e^{30\delta q_n}$. Then \eqref{eq:h=constant_omega} follows from combining \eqref{eq:h-h_1} with \eqref{eq:h-h_2}.
Finally, it remains to note that an analogue of Lemma \ref{lem:log_min}, which controls the sum of the logarithmic potentials via their minimum terms, is true for $\omega$, again due to the almost periodicity of the zeros provided by Lemma \ref{lem:ap_zero_omega}.
\end{proof}

\subsection{The consequence of even potentials and non-resonance of $\theta$}\label{sec:even}

Our goal in this subsection is to prove Lemma \ref{lem:zero_omega_reflection} below.
As a preparation, we first show $e^{2\pi i(\theta-[q_n/2]\omega)}$ is $e^{-\delta q_n}$ close to one of the zeros in $\mathcal{Z}_{q_n}(\omega,E)$.
\begin{lemma}\label{lem:-q/2_small}
For $n$ large enough, and $q_{n+1}\geq e^{\delta_1 q_n}$, the following holds:
    \begin{align}
        e^{2\pi i(\theta-[q_n/2]\omega)}\in \bigcup_{s=1}^2 \bigcup_{j=0}^{q_n-1} B_{\exp(-\delta q_n)}(z^{\omega}_{q_n,j,s}).
    \end{align}
\end{lemma}
\begin{proof}
    Proof by contradiction. Assume otherwise, then by Lemma \ref{lem:vqn_Adelta_3_omega}, 
    \begin{align}
        |f^{\omega}_{q_n,E}(e^{2\pi i(\theta-[q_n/2])\omega})|\geq e^{(L-12\delta)q_n}.
    \end{align}
    By \eqref{eq:M=P} and \eqref{eq:f=M+M}, we have
    \begin{align}
        |P_{q_n,E}^{\omega}(\theta-[q_n/2]\omega)-P_{q_n-2,E}^{\omega}(\theta-[q_n/2]\omega+\omega)|\geq e^{(L-12\delta)q_n}-2.
    \end{align}
    Assume without loss of generality that
    \begin{align}\label{eq:P_-q/2_large}
        |P_{q_n,E}^{\omega}(\theta-[q_n/2]\omega)|\geq e^{(L-14\delta)q_n}.
    \end{align}
    Applying the Green's function expansion \eqref{eq:Green_exp} to $\phi_0$ on the interval $[-[q_n/2], -[q_n/2]+q_n-1]$, we obtain
    \begin{align}
        |\phi_0|\leq \frac{|P^{\omega}_{-[q_n/2]+q_n-1,E}(\theta+\omega)|}{|P_{q_n,E}^{\omega}(\theta-[q_n/2]\omega)|}|\phi_{-[q_n/2]-1}|+\frac{|P^{\omega}_{[q_n/2],E}(\theta-[q_n/2]\omega)|}{|P^{\omega}_{q_n,E}(\theta-[q_n/2]\omega)|}|\phi_{-[q_n/2]+q_n}|.
    \end{align}
    Combining the lower bound of the denominator in \eqref{eq:P_-q/2_large} with the standard upper bound of the numerator in \eqref{eq:M=P} and Lemma \ref{lem:upper_semi_cont}, we have
    \begin{align}
        |\phi_0|\leq e^{-(L-30\delta)\frac{q_n}{2}} \max(|\phi_{[-q_n/2]+q_n}|, |\phi_{-[q_n/2]-1}|).
    \end{align}
    Finally using that $\phi$ is a generalized solution satisfying \eqref{eq:Shnol}, we conclude that
    \begin{align}
        |\phi_0|\leq e^{-(L-35\delta)\frac{q_n}{2}}.
    \end{align}
    The same argument implies the same bound for $|\phi_{-1}|$, which leads to a contradiction to \eqref{eq:Shnol}.
\end{proof}
Lemma \ref{lem:-q/2_small} implies that for some $j_0\in \{0,...,q_n-1\}$ and $s_0\in \{1,2\}$, 
\begin{align}\label{eq:theta_close_j0}
r^{\omega}_{j_0,s_0}=1+O(e^{-\delta q_n}), \text{ and } \|\theta-[q_n/2]\omega-\theta^{\omega}_{j_0,s_0}\|=O(e^{-\delta q_n}).
\end{align}
Without loss of generality, we assume $s_0=1$.
We will show:
\begin{lemma}\label{lem:zero_omega_reflection}
For $n$ large enough, and $q_{n+1}\geq e^{\delta_1 q_n}$, the following holds:
    \begin{align}
        \mathcal{Z}_{q_n}(\omega,E)=\bigcup_{j=0}^{q_n-1}\{z^{\omega}_{q_n,j,1}, e^{-2\pi i(q_n-1)\omega}/z^{\omega}_{q_n,j,1}\}.
    \end{align}
    Furthermore, $z^{\omega}_{q_n,j,1}$ is ``far away" from $e^{-2\pi i(q_n-1)\omega}/z^{\omega}_{q_n,k,1}$, for any $j,k\in \{0,...,q_n-1\}$ in the sense that:
    \begin{align}\label{eq:thetak-thetaj>-1/25}
    \|\theta_{k,1}^{\omega}-(-\theta_{j,1}^{\omega}-(q_n-1)\omega)\|\geq e^{-\delta q_n/25},
\end{align}
\end{lemma}
\begin{proof}
We have for any $k,j\in \{0,...,q_n-1\}$, by \eqref{eq:r_theta_periodic},
\begin{align}\label{eq:theta_k-j}
    \|\theta_{k,1}^{\omega}-(-\theta_{j,1}^{\omega}-(q_n-1)\omega)\|
    =&\|\theta_{k,1}^{\omega}+\theta_{j,1}^{\omega}+(q_n-1)\omega\|\notag\\
    =&\|2\theta_{j_0,1}^{\omega}+\frac{(k+j-2j_0)p_n}{q_n}+(q_n-1)\omega\|+O(e^{-\delta q_n})\notag\\
    =&\|2\theta_{j_0,1}^{\omega}+(k+j-2j_0-1)\omega\|+O(e^{-\delta q_n}),
\end{align}
where we used $\|q_n\omega\|\leq e^{-\delta_1 q_n}\ll e^{-\delta q_n}$ (see \eqref{eq:qn_omega}), and $|k+j-2j_0|\leq 2q_n$, which implies 
$$\left|(k+j-2j_0)\left(\frac{p_n}{q_n}-\omega\right)\right|=\frac{|k+j-2j_0|}{q_n}\|q_n\omega\|\leq 2e^{-\delta_1 q_n}\ll e^{-\delta q_n}.$$
Combining \eqref{eq:theta_close_j0} with \eqref{eq:theta_k-j} yields
\begin{align}\label{eq:2theta_1}
    \|\theta_{k,1}^{\omega}-(-\theta_{j,1}^{\omega}-(q_n-1)\omega)\|=&\|2\theta+(k+j-2j_0-1-2[q_n/2])\omega\|+O(e^{-\delta q_n}).
\end{align}
It is easy to see that $|k+j-2j_0-1-2[q_n/2]|\leq 3q_n$, hence by \eqref{eq:theta_non_res_2} with $\delta'=\delta/100$, we have
\begin{align}\label{eq:2theta_2}
    \|2\theta+(k+j-2j_0-1-2[q_n/2])\omega\|\geq c_{\delta}e^{-\delta |k+j-2j_0-1-2[q_n/2]|/100}\geq c_{\delta} e^{-3\delta q_n/100}\gg e^{-\delta q_n}.
\end{align}
Thus \eqref{eq:2theta_1} together with \eqref{eq:2theta_2} implies
\begin{align}
    \|\theta_{k,1}^{\omega}-(-\theta_{j,1}^{\omega}-(q_n-1)\omega)\|\geq e^{-\delta q_n/25},
\end{align}
in particular 
\begin{align}\label{eq:theta_k_neq_-theta_j}
    \left(\bigcup_{j=0}^{q_n-1}\{\theta_{j,1}^{\omega}\}\right)\cap \left(\bigcup_{j=0}^{q_n-1}\{-\theta_{j,1}^{\omega}-(q_n-1)\omega\}\right)=\emptyset.
\end{align}
Recall that we assume the potential is an even function. This implies that, see \eqref{eq:Pk_even_1},
\begin{align}
    \mathrm{tr}(M^{\omega}_{q_n,E}(\theta-\frac{q_n-1}{2}\omega))
    =&P_{q_n}^{\omega}(\theta-\frac{q_n-1}{2}\omega)-P_{q_n-2}^{\omega}(\theta-\frac{q_n-1}{2}\omega+\omega)\\
    =&P_{q_n}^{\omega}(-\theta-\frac{q_n-1}{2}\omega)-P_{q_n-2}^{\omega}(-\theta-\frac{q_n-1}{2}\omega+\omega)\\
    =&\mathrm{tr}(M^{\omega}_{q_n,E}(-\theta-\frac{q_n-1}{2}\omega)).
\end{align}
By \eqref{def:f=2-tr}, this further implies
\begin{align}
    f^{\omega}_{q_n,E}(e^{2\pi i(\theta-(q_n-1)\omega/2)})=f^{\omega}_{q_n,E}(e^{2\pi i(-\theta-(q_n-1)\omega/2)})
\end{align}
This means the two analytic function below coincide on the unit circle $z\in \mathcal{C}_1$, hence are identical:
\begin{align}\label{eq:fqn_even}
    f_{q_n,E}^{\omega}(z)=f_{q_n,E}^{\omega}(e^{-2\pi i(q_n-1)\omega}/z).
\end{align}
Then if $z$ is a zero of $f_{q_n,E}^{\omega}(z)$, $e^{-2\pi i (q_n-1)\omega}/z$ is also a zero.
Note that \eqref{eq:theta_k_neq_-theta_j} implies $\bigcup_{j=0}^{q_n-1}\{z^{\omega}_{q_n,j,1}\}$ and $\bigcup_{j=0}^{q_n-1}\{e^{-2\pi i(q_n-1)\omega}/z^{\omega}_{q_n,j,1}\}$ are distinct (and in total $2q_n$) zeros. 
Therefore Lemma \ref{lem:zero_omega_reflection} follows from \eqref{eq:Zqn_omega} in Lemma \ref{lem:vqn_Adelta_3_omega}.
\end{proof}

\subsection{Eigenfunctions in the weakly resonant regime}\label{sec:weak_regime}
The main difficulty in proving Anderson localization lies in the strongly resonant regime, which are the locations of the local peaks of the eigenfunctions. To address the strong regimes in a sharp way, one needs to first control the weakly resonant regimes in terms of its adjacent peaks. 
The following lemma shows, roughly speaking, the eigenfunction in a weakly resonant regime can be dominated by its values in its two adjacent strongly resonant regimes.

\begin{theorem}\label{thm:AL_weak_regime}
For any $\ell\in \Z$, and large enough $y\in [\ell q_n, (\ell+1)q_n]\setminus (R_{\ell q_n}\bigcup R_{(\ell+1)q_n})$ satisfying $|y|\leq 10q_{n+1}^{1-c_0}$, the following holds:
\begin{align}
    |\phi_y|\leq \max
    \begin{cases}
    e^{-(L-2000C_v\delta_1^{1/4})\cdot \mathrm{dist}(y_0, R_{\ell q_n})}\cdot r_{\ell q_n},\\ 
    e^{-(L-2000C_v\delta_1^{1/4})\cdot \mathrm{dist}(y_0, R_{(\ell+1) q_n})}\cdot r_{(\ell+1) q_n}.
    \end{cases}
\end{align}
\end{theorem}
\begin{proof}
    The proof is similar to, but more difficult than, that of the weak Liouville case in Sec. \ref{sec:non_exp}.
    
For $y\in [\ell q_n, (\ell+1)q_n]$, let $m_y:=\mathrm{dist}(y, q_n \Z)/2$. 
    Let 
    \begin{align}\label{def:I1_I2_weak_regime}
        I_1^y:=\left[-\left[\frac{9}{10}m_y\right], -\left[\frac{3}{5}m_y\right]\right], \text{ and }
        I_2^y:=\left[y-\left[\frac{9}{10}m_y\right], y-\left[\frac{1}{10}m_y\right]\right].
    \end{align}
    By Corollary \ref{cor:gm_zeros_even}, we have that:
    \begin{lemma}\label{lem:weak_regime_interval}
        The large deviation set $\mathcal{B}^g_{m_y,1100C_v\delta_1^{1/4},E}$ satisfies:
        \begin{align}
            \mathcal{B}^g_{m_y,1100C_v\delta_1^{1/4},E}=\bigcup_{j=1}^{\tilde{N}'}(U_{m_y,j}\bigcup(-U_{m_y,j}-(m_y-1)\omega)),
        \end{align}
        for some $\tilde{N}'\leq (1+\eta)m_y$ with $\eta$ as in \eqref{def:eta}. Furthermore each $U_{m_y,j}$ satisfies $\mathrm{mes}(U_{m_y,j})\leq e^{-100\delta_1 m_y}$.
    \end{lemma}
    \begin{proof} 
    Note that if $q_n\leq e^{\delta_1 q_{n-1}}$, then $y\in [(\ell+10\delta_1^{1/4})q_n, (\ell+1-10\delta_1^{1/4})q_n]$, and $q_n/4\geq m_y\geq 5\delta_1^{1/4}q_n$.
If $q_n\geq e^{\delta_1 q_{n-1}}$, then $y\in [\ell q_n+10\delta_1^{1/4} q_n^{1-c_0}, (\ell+1)q_n-10\delta_1^{1/4} q_n^{1-c_0}]$, and $q_n/4\geq m_y\geq 5\delta_1^{1/4} q_n^{1-c_0}$.
We have verified the conditions in Corollary \ref{cor:gm_zeros_even}, therefore Corollary \ref{cor:gm_zeros_even} directly implies the claimed result.
\end{proof}
    Then one can show:
    \begin{lemma}\label{lem:non_res_repulsion}
For any $k\in I_1^y$, we have 
\begin{align}
    \theta+k\omega\in \mathcal{B}^g_{m_y, 1100C_v\delta_1^{1/4},E}.
\end{align} Furthermore there exists $k_2^y\in I_2^y$ such that
\begin{align}\label{eq:k2_not_in_B}
    \theta+k_2^y\omega\notin \mathcal{B}^g_{m_y, 1100C_v\delta_1^{1/4},E}.
\end{align}
    \end{lemma}
    \begin{proof}
        The first claim regarding $I_1^y$ is very similar to that of Lemma \ref{lem:-q/2_small}, which we shall leave for the readers.
        
        To prove \eqref{eq:k2_not_in_B}, it suffices to prove the following repulsion property:
        \begin{lemma}\label{lem:repulsion_2}
            If for some $k\in I_1^y\bigcup I_2^y$, $$\theta+k\omega\in (U_{m_y,j}\bigcup(-U_{m_y,j}-(m_y-1)\omega)$$ for some $j\in \{0,...,q_n-1\}$. Then the following holds for any $k'\in I_1^y\bigcup I_2^y\setminus \{k\}$:
        \begin{align}
            \theta+k'\omega\notin (U_{m_y,j}\bigcup(-U_{m_y,j}-(m_y-1)\omega)).
        \end{align}
        \end{lemma}
        \begin{proof}
        Without loss of generality, we assume $\theta+k\omega\in U_{m_y,j}$ (the other case is completely analogous). 
        
        \underline{Case 1}. If $k,k'$ belong to $I_1^y$ or $I_2^y$ simultaneously, then 
        $0<|k-k'|\leq 4m_y/5\leq q_n/5$, where we used $m_y\leq q_n/4$. 
        Hence by \eqref{eq:qn_omega_min} and \eqref{eq:qn_omega}, we have
        \begin{align}
            \|\theta+k'\omega-(\theta+k\omega)\|\geq \|q_{n-1}\omega\|\geq \frac{1}{2q_n}\gg e^{-100\delta_1 m_y}\geq \mathrm{mes}(U_{m_y,j}),
        \end{align}
        since $m_y\geq 5\delta_1^{1/4}q_n^{1-c_0}$.
        This implies $\theta+k'\omega\notin U_{m_y,j}$.

        To show $\theta+k'\omega\notin (-U_{m_y,j}-(m_y-1)\omega)$, we further distinguish two cases. 

        \underline{Case 1.1}. If $k,k'\in I_1^y$, it is easy to check that
        \begin{align}
            k'+k+m_y-1\in [-\frac{4}{5}m_y,-\frac{1}{5}m_y].
        \end{align}
        Then by \eqref{eq:theta_non_res} with $\delta'=\delta_1$, when $y$ is large (hence $n$ is large and $m_y$ is large),
        \begin{align}
            &\|\theta+k'\omega-(-\theta-k\omega-(m_y-1)\omega)\|\\
            =&\|2\theta+(k'+k+m_y-1)\omega\|\geq e^{-\delta_1 |k'+k+m_y-1|}\geq e^{-4\delta_1 m_y/5}\gg e^{-100\delta_1 m_y}\geq \mathrm{mes}(-U_{m_y,j}-(m_y-1)\omega).
        \end{align}
        This implies 
        \begin{align}\label{eq:theta+k'_notin}
            \theta+k'\omega\notin -U_{m_y,j}-(m_y-1)\omega\ni -(\theta+k\omega)-(m_y-1)\omega.
        \end{align}
        \underline{Case 1.2}. If $k,k'\in I_2^y$. Since $|y|\leq 10q_{n+1}^{1-c_0}$, $|\ell|q_n<20q_{n+1}^{1-c_0}$, we obtain by \eqref{eq:qn_omega} that,
        \begin{align}\label{eq:2l_qn}
            \max(|2\ell|, |2\ell+2|)\|q_n\omega\|\leq \frac{4\max(|\ell|,1)}{q_{n+1}}\leq 80 q_{n+1}^{-c_0}\leq 80 e^{-\delta_1c_0q_n},
        \end{align}
        where we used $q_{n+1}\geq e^{\delta_1 q_n}$ in the last inequality.

        We further distinguish two cases depending on if $y\leq (\ell+1/2)q_n$.
        
        \underline{Case 1.2.1}. If $\ell q_n\leq y\leq (\ell+1/2)q_n$. 
        In this case $y-\ell q_n=2m_y$.
        This implies
        \begin{align}
            k+k'+m_y-1-2\ell q_n\in [2y-\frac{4}{5}m_y-2\ell q_n-1, 2y+\frac{4}{5}m_y-2\ell q_n-1]\subset [3m_y, 5m_y].
        \end{align}
        Hence by \eqref{eq:theta_non_res} with $\delta'=\delta_1c_0/10<\delta_1/10$, we obtain for $y$ large enough that
        \begin{align}\label{eq:121}
            \|2\theta+(k+k'+m_y-1-2\ell q_n)\omega\|\geq e^{-5\delta' m_y}\geq e^{-5\delta'q_n/4} \gg 80e^{-\delta_1c_0q_n}\geq |2\ell|\|q_n\omega\|,
        \end{align}
        in which we used \eqref{eq:2l_qn} in the last inequality.
        By triangle inequality, \eqref{eq:121} implies 
        \begin{align}\label{eq:121_3}
            &\|\theta+k'\omega-(-\theta-k\omega-(m_y-1)\omega)\|\notag\\
            \geq &\|\theta+k'\omega-(-\theta-k\omega-(m_y-1)\omega)-2\ell q_n\omega\|-|2\ell|\|q_n\omega\|\notag\\
            =&\|2\theta+(k'+k+m_y-1-2\ell)q_n)\omega\|-|2\ell|\|q_n\omega\|\notag\\
            \geq &e^{-5\delta'm_y}/2 \gg e^{-100\delta_1 m_y}\geq \mathrm{mes}(-U_{m_y,j}-(m_y-1)\omega).
        \end{align}
        This verifies \eqref{eq:theta+k'_notin} for Case 1.2.1.

        \underline{Case 1.2.2}. If $(\ell+1/2)q_n\leq y\leq (\ell+1)q_n$. In this case $(\ell+1)q_n-y=2m_y$.  The estimates below are similar to those in Case 1.2.1.
        First, we have
        \begin{align}
            k+k'+m_y-1-2(\ell+1)q_n\in [2y-\frac{4}{5}m_y-2(\ell+1)q_n-1, 2y+\frac{4}{5}m_y-2(\ell+1)q_n-1]\subset [-5m_y, -3m_y].
        \end{align}
        This implies similar to \eqref{eq:121} that
        \begin{align}\label{eq:122}
            \|2\theta+(k+k'+m_y-1-2(\ell+1)q_n)\omega\|\geq e^{-5\delta' m_y}\geq e^{-5\delta' q_n/4}\gg |2(\ell+1)|\|q_n\omega\|.
        \end{align}
        This, by triangle inequality, implies, similar to \eqref{eq:121_3}, that
        \begin{align}
            \|\theta+k'\omega-(-\theta-k\omega-(m_y-1)\omega)\|\gg \mathrm{mes}(-U_{m_y,j}-(m_y-1)\omega).
        \end{align}
        Thus we have verified \eqref{eq:theta+k'_notin} in Case 1.2.2.

        \underline{Case 2}. If $k\in I_1^y$ and $k'\in I_2^y$ (the other case when $k\in I_2^y$ and $k'\in I_1^y$ is analogous), we distinguish two cases below:
        
        \underline{Case 2.1}. If $y\leq \ell q_n+q_n/2$. In this case $y-\ell q_n=2m_y$, hence
        \begin{align}
            k'-\ell q_n-k\in [y-\ell q_n-\frac{3}{10}m_y, y-\ell q_n+\frac{4}{5}m_y]\subset [m_y,3m_y]\subset [m_y,3q_n/4].
        \end{align}
           Therefore, by \eqref{eq:qn_omega_min},
            \begin{align}\label{eq:a1}
                \|\theta+k\omega-(\theta+k'\omega)+\ell q_n\omega\|=\|(k-(k'-\ell q_n))\omega\|\geq \|q_{n-1}\omega\|\geq \frac{1}{2q_n}\gg 40e^{-\delta_1c_0q_n}.
            \end{align}
            Combining this with \eqref{eq:2l_qn} yields 
            \begin{align}\label{eq:a2}
                \|\theta+k\omega-(\theta+k'\omega)\|\geq \|\theta+k\omega-(\theta+k'\omega)+\ell q_n\omega\|-\|\ell q_n\omega\|\geq \frac{1}{4q_n}
                \gg &e^{-100\delta_1 m_y}\notag\\ \geq &\mathrm{mes}(U_{m_y,j}),
            \end{align}
            where we used $m_y\geq 5\delta_1^{1/4}q_n^{1-c_0}$.
            Clearly, \eqref{eq:a2} implies $\theta+k'\omega\notin U_{m_y,j}$.

            We also have, by $y-\ell q_n=2m_y$ that
            \begin{align}
                k+k'+m_y-1-\ell q_n\in [y-\frac{4}{5}m_y-1-\ell q_n, y+\frac{3}{10}m_y-1-\ell q_n]\subset [m_y, 3m_y].
            \end{align}
            By \eqref{eq:theta_non_res} with $\delta'=\delta_1c_0/10<\delta_1/10$, we have for $y$ large enough that
            \begin{align}
                \|2\theta+(k+k'+m_y-1-\ell q_n)\omega\|\geq e^{-3\delta'm_y}\geq e^{-3\delta'q_n/4}\gg 40e^{-\delta_1c_0q_n}\geq |\ell|\|q_n\omega\|,
            \end{align}
            where we used \eqref{eq:2l_qn} in the last inequality.
            Combining this with the triangle inequality yields
            \begin{align}
                &\|\theta+k'\omega-(-\theta-k\omega-(m_y-1)\omega)\|\\
                =&\|2\theta+(k+k'+m_y-1)\omega\|\\
                \geq &\|2\theta+(k+k'+m_y-1-\ell q_n)\|-|\ell|\|q_n\omega\|\\
                \geq &e^{-3\delta'm_y}/2\gg e^{-100\delta_1m_y}\geq \mathrm{mes}(-U_{m_y,j}-(m_y-1)\omega).
            \end{align}
            This implies $\theta+k'\omega\notin (-U_{m_y,j}-(m_y-1)\omega)$.
            
        \underline{Case 2.2}. If $y\geq \ell q_n+q_n/2$.  
        This case is very similar to Case 2.1. 
        Note that in this case, $(\ell+1)q_n-y=2m_y$, thus
        \begin{align}
            k'-(\ell+1)q_n-k
            \in [y-(\ell+1)q_n-\frac{3}{10}m_y, y-(\ell+1)q_n+\frac{4}{5}m_y]
            \subset [-3m_y, -m_y]\subset [-3q_n/4,-m_y].
        \end{align}
        This implies, by \eqref{eq:qn_omega_min} that
        \begin{align}
            \|\theta+k\omega-(\theta+k'\omega)+(\ell+1)q_n\omega\|\geq \|q_{n-1}\omega\|\geq \frac{1}{2q_n}
        \end{align}
        Then similar to \eqref{eq:a2} above, we conclude that
        \begin{align}
            \|\theta+k\omega-(\theta+k'\omega)\|\geq \frac{1}{4q_n}\gg e^{-100\delta_1m_y}\geq \mathrm{mes}(U_{m_y,j}).
        \end{align}
        This shows $\theta+k'\omega\notin U_{m_y,j}$.

        We also have
        \begin{align}
            k+k'+m_y-1-(\ell+1)q_n\in [y-\frac{4}{5}m_y-1-(\ell+1)q_n, y+\frac{3}{10}m_y-1-\ell q_n]\subset [-3m_y,-m_y].
        \end{align}
        By \eqref{eq:theta_non_res} with $\delta'=\delta_1c_0/10<\delta_1/10$, the following holds:
        \begin{align}
            \|2\theta+(k+k'+m_y-(\ell+1)q_n)\omega\|\geq e^{-3\delta' m_y}\geq e^{-3\delta'q_n/4}\gg 40e^{-\delta_1c_0q_n}\geq |\ell+1|\|q_n\omega\|,
        \end{align}
        where we used \eqref{eq:2l_qn} in this last inequality. 
        This, together with the triangle inequality, yields
        \begin{align}
            &\|\theta+k'\omega-(-\theta-k\omega-(m_y-1)\omega)\|\\
            =&\|2\theta+(k+k'+m_y-1)\omega\|\\
            \geq &e^{-3\delta'm_y}/2\gg e^{-100\delta_1m_y}\geq \mathrm{mes}(-U_{m_y,j}-(m_y-1)\omega).
        \end{align}
        This implies $\theta+k'\omega\notin (-U_{m_y,j}-(m_y-1)\omega)$.
        Therefore we have proved the claimed result of Lemma \ref{lem:repulsion_2}.
    \end{proof}
    Finally, note that 
    $$\mathrm{card}(I_1^y)+\mathrm{card}(I_2^y)\geq \frac{11}{10}m_y-2>(1+\eta)m_y\geq \tilde{N}',$$
    where $\tilde{N}'$ is as in Lemma \ref{lem:weak_regime_interval}. Lemma \ref{lem:non_res_repulsion} then follows from the pigeonhole principle.
\end{proof}
Returning to the proof of Theorem \ref{thm:AL_weak_regime}, by Lemma \ref{lem:non_res_repulsion}, we conclude similar to \eqref{eq:g_large_implies_P_large} that there exists $m_y'\in \{m_y,m_y-1,m_y-2\}$ and $a\in \{0,1\}$ such that 
\begin{align}
    |P_{m_y'}^{\omega}(\theta+(k_2^y+a)\omega)|\geq \frac{1}{2}e^{m_y(L_{m_y}-1000C_v\delta_1^{1/4})}.
\end{align}
Denoting $\Gamma_y:=[k_2^y+a, k_2^y+a+m_y'-1]$ and $\partial \Gamma_y:=\{k_2^y+a-1, k_2^y+a+m_y'\}$. 
Expanding $\phi_y$ using Green's function expansion on the interval $\Gamma_y$, and denoting $y=y_0$, we have similar to \eqref{eq:exp_phi0_weak} that:
\begin{align}
    |\phi_{y_0}|\leq \max_{y_1\in \partial \Gamma_{y_0}} e^{-\mathrm{dist}(y_0,y_1) \cdot (L-2000C_v\delta_1^{1/4})}|\phi_{y_1}|.
\end{align}
Clearly one expand on $\phi_{y_1}$ as long as $y_1\notin R_{\ell q_n}\bigcup R_{(\ell+1)q_n}$.
Iterating such expansion yields
\begin{align}\label{eq:iterate_Green_phi_y0}
    |\phi_{y_0}|\leq \sup_{(y_1, y_2,..., y_t)\in S} e^{-\sum_{j=1}^t\mathrm{dist}(y_{j-1},y_{j})\cdot (L-2000C_v\delta_1^{1/4})} |\phi_{y_t}|,
\end{align}
where $S$ is collection of admissible chains such that either
\begin{align}
    \begin{cases}
        y_t\in R_{\ell q_n}, \text{ or } \\
        y_t\in R_{(\ell+1)q_n}, \text{ or } \\
        t=t_0, \text{ and } y_j\notin R_{\ell q_n}\bigcup R_{(\ell+1)q_n}, \text{ for any } 1\leq j\leq t_0,
    \end{cases}
\end{align}
where the stopping time
\begin{align}
    t_0:=
    \begin{cases}
        [10/\delta_1^{1/4}]+1, \text{ if } q_n\leq e^{\delta_1 q_{n-1}}, \text{ or } \\
        [10 q_n^{c_0}/\delta_1^{1/4}]+1, \text{ if } q_n\geq e^{\delta_1 q_{n-1}}.
    \end{cases}
\end{align}
Note that if we do artificially terminate the process at $t=t_0$, then for each $0\leq j< t_0$, we have $\phi_{y_j}\notin R_{\ell q_n}\bigcup R_{(\ell+1)q_n}$ and hence
\begin{align}
    \mathrm{dist}(y_j, q_n\Z)\geq 
    \begin{cases}
        10\delta_1^{1/4} q_n, \text{ if } q_n\leq e^{\delta_1 q_{n-1}}\\
        10\delta_1^{1/4} q_n^{1-c_0}, \text{ if } q_n\geq e^{\delta_1 q_{n-1}}.
    \end{cases}
\end{align}
Therefore, due to our choice of the $I_2^{y_j}$ interval as in \eqref{def:I1_I2_weak_regime},
\begin{align}
    \mathrm{dist}(y_j,y_{j+1})\geq \frac{1}{10}m_{y_j}=\frac{1}{20}\mathrm{dist}(y_j, q_n\Z)\geq     \begin{cases}
        \delta_1^{1/4} q_n/2, \text{ if } q_n\leq e^{\delta_1 q_{n-1}}\\
        \delta_1^{1/4} q_n^{1-c_0}/2, \text{ if } q_n\geq e^{\delta_1 q_{n-1}}.
    \end{cases}
\end{align}
Then
\begin{align}
    \sum_{j=1}^{t_0}\mathrm{dist}(y_{j-1},y_{j})
    \geq &t_0\cdot
    \begin{cases}
        \delta_1^{1/4} q_n/2, \text{ if } q_n\leq e^{\delta_1 q_{n-1}}\\
        \delta_1^{1/4} q_n^{1-c_0}/2, \text{ if } q_n\geq e^{\delta_1 q_{n-1}}
    \end{cases}\\
    \geq &4q_n.
\end{align}
This implies such chain of expansion $(y_1,y_2,...,y_{t_0})\in S$ contributes at most
\begin{align}
    e^{-4(L-2000C_v\delta_1^{1/4})q_n}|\phi_{y_{t_0}}|\leq e^{-4(L-2000C_v\delta_1^{1/4})q_n}\cdot b_{\ell,\ell+1,q_n},
\end{align}
in which
\begin{align}
    b_{\ell,\ell+1,q_n}:=\sup_{k\in [\ell q_n, (\ell+1)q_n]\setminus (R_{\ell q_n}\bigcup R_{(\ell+1)q_n})} |\phi_k|.
\end{align}

If a chain $(y_1,y_2,...,y_t)\in S$ is such that $y_{t}\in R_{\ell q_n}$, then along this chain, by triangle inequality,
\begin{align}
    \sum_{j=1}^{t}\mathrm{dist}(y_{j-1}, y_j)\geq \mathrm{dist}(y_0, R_{\ell q_n}).
\end{align}
The contribution along this chain is at most
\begin{align}
    e^{-(L-2000C_v\delta_1^{1/4})\cdot \mathrm{dist}(y_0, R_{\ell q_n})}\cdot r_{\ell q_n},
\end{align}
where we dominate $|\phi_{y_t}|$ by $r_{\ell q_n}$ since $y_{t}\in R_{\ell q_n}$.

If a chain $(y_1,y_2,...,y_t)\in S$ is such that $y_t\in R_{(\ell+1)q_n}$, then similar to the case above, the contribution along this chain is at most
\begin{align}
    e^{-(L-2000C_v\delta_1^{1/4})\cdot \mathrm{dist}(y_0, R_{(\ell+1) q_n})}\cdot r_{(\ell+1) q_n}.
\end{align}
Combining the three cases above with \eqref{eq:iterate_Green_phi_y0}, we conclude that for any $y\in [\ell q_n, (\ell+1)q_n]\setminus (R_{\ell q_n}\bigcup R_{(\ell+1)q_n})$, 
\begin{align}\label{eq:phiy_exp_1}
    |\phi_y|\leq \max
    \begin{cases}e^{-(L-2000C_v\delta_1^{1/4})\cdot \mathrm{dist}(y_0, R_{\ell q_n})}\cdot r_{\ell q_n}\\ 
    e^{-(L-2000C_v\delta_1^{1/4})\cdot \mathrm{dist}(y_0, R_{(\ell+1) q_n})}\cdot r_{(\ell+1) q_n} \\
    e^{-4(L-2000C_v\delta_1^{1/4})q_n}\cdot b_{\ell,\ell+1,q_n}.
    \end{cases}
\end{align}
In particular, we can take $y$ such that $|\phi_y|=b_{\ell,\ell+1,q_n}$, then the inequality above yields
\begin{align}\label{eq:b_exp}
    b_{\ell,\ell+1,q_n}\leq \max
    \begin{cases}
    e^{-(L-2000C_v\delta_1^{1/4})\cdot \mathrm{dist}(y_0, R_{\ell q_n})}\cdot r_{\ell q_n}\\ 
    e^{-(L-2000C_v\delta_1^{1/4})\cdot \mathrm{dist}(y_0, R_{(\ell+1) q_n})}\cdot r_{(\ell+1) q_n}.
    \end{cases}
\end{align}
Combining \eqref{eq:phiy_exp_1} with \eqref{eq:b_exp} yields
\begin{align}
    |\phi_y|\leq \max
    \begin{cases}
    e^{-(L-2000C_v\delta_1^{1/4})\cdot \mathrm{dist}(y_0, R_{\ell q_n})}\cdot r_{\ell q_n}\\ 
    e^{-(L-2000C_v\delta_1^{1/4})\cdot \mathrm{dist}(y_0, R_{(\ell+1) q_n})}\cdot r_{(\ell+1) q_n},
    \end{cases}
\end{align}
as claimed.
\end{proof}

\subsection{Eigenfunction in the strongly resonant regimes}\label{sec:ef_strong_strong}
We are now ready to study the eigenfunction in a strongly resonant regime $R_{\ell q_n}$. 
The goal of this subsection is to prove the following:
\begin{theorem}\label{thm:AL_lqn}
    For $n$ large enough, if $q_{n+1}\geq e^{\delta_1 q_n}$, then the following holds for any $1\leq |k|\leq 5q_{n+1}^{1-c_0}/q_n$:
    \begin{align}
        r_{kq_n}\leq e^{-(1-o(1))(L-\beta_n)|k|q_n},
    \end{align}
    where $o(1)\in (0,1/50)$.
\end{theorem}
In fact, $o(1)\sim (5L+1000C_v)\delta_1^{1/4}$. Hence by choosing smaller $\delta_1$, one can make $o(1)$ arbitrarily small.

\begin{proof}
Recall $\delta=\delta_1/30$ as in \eqref{eq:fix_delta=delta1/30}.
We need the following lemma:
\begin{lemma}\label{lem:qn/2_small_other_large}
    If for some $|j|\leq q_n$,
    \begin{align}\label{eq:assume_vqn_small}
        v_{q_n,E}^{\omega}(e^{2\pi i(\theta-j\omega)})<L(\omega,E)-\beta_n-12\delta,
    \end{align} then for any $k\in \Z\setminus \{0\}$, $|k|\leq q_{n+1}^{1-c_0}$, we have 
    \begin{align}
        v_{q_n,E}^{\omega}(e^{2\pi i(\theta+(kq_n-j)\omega)})\geq L(\omega,E)+\frac{\log |k|}{q_n}-\beta_n-12\delta.
    \end{align}
\end{lemma}
\begin{proof}
    By Lemma \ref{lem:zero_omega_reflection}, \eqref{eq:assume_vqn_small} implies that 
    \begin{align}\label{eq:jell<betan}
        \left(\min_{\ell=0}^{q_n-1}q_n^{-1}\log |e^{2\pi i(\theta-j\omega)}-z^{\omega}_{q_n,\ell,1}|+\min_{\ell=0}^{q_n-1} q_n^{-1}\log |e^{2\pi i(\theta-j\omega)}-e^{-2\pi i(q_n-1)\omega}/z^{\omega}_{q_n,\ell,1}|\right)<-\beta_n-2\delta.
    \end{align}
    Due to non-resonance of $\theta$, only one of these two minimums can be less than $-\delta /25$, see \eqref{eq:thetak-thetaj>-1/25}, which forces this minimum to be less than $-\beta_n-\delta$.
    Assume without loss of generality that for some $\ell_0\in \{0,1,...,q_n-1\}$,
    \begin{align}
        |e^{2\pi i(\theta-j\omega)}-z^{\omega}_{q_n,\ell_0,1}|\leq e^{-(\beta_n+\delta)q_n},
    \end{align}
    which implies
    \begin{align}\label{eq:theta_j<-beta}
        \|\theta-j\omega-\theta^{\omega}_{\ell_0,1}\|\leq e^{-(\beta_n+\delta)q_n}.
    \end{align}
    
    We first show $\theta-j\omega+kq_n\omega$ is away from $\bigcup_{\ell=0}^{q_n-1}\{-\theta_{\ell,1}^{\omega}-(q_n-1)\omega\}$.
    
    By \eqref{eq:r_theta_periodic}, we have for any $\ell\in \{0,...,q_n-1\}$,
    \begin{align}\label{eq:theta-j--theta}
        \|\theta-j\omega-(-\theta^{\omega}_{\ell,1}-(q_n-1)\omega)\|
        =&\|\theta+\theta_{\ell,1}^{\omega}+(q_n-j-1)\omega\|\notag\\
        =&\|\theta+\theta_{\ell_0,1}^{\omega}+\frac{(\ell-\ell_0)p_n}{q_n}+(q_n-j-1)\omega+O(e^{-\delta q_n})\|\notag\\
        =&\|\theta+\theta_{\ell_0,1}^{\omega}+(\ell-\ell_0-j-1)\omega+O(e^{-\delta q_n})\|,
    \end{align}
    where we used by \eqref{eq:qn_omega} that
    \begin{align}
        |\ell-\ell_0|\cdot \left|\frac{p_n}{q_n}-\omega\right|\leq \frac{|\ell-\ell_0|}{q_n}\|q_n\omega\|\leq 2\|q_n\omega\|\leq 2e^{-\delta_1 q_n}\ll e^{-\delta q_n}.
    \end{align}
    Combining \eqref{eq:theta_j<-beta} with \eqref{eq:theta-j--theta} yields
    \begin{align}\label{eq:b1}
        \|\theta-j\omega-(-\theta^{\omega}_{\ell,1}-(q_n-1)\omega)\|=\|2\theta+(\ell-\ell_0-2j-1)\omega+O(e^{-\delta q_n})\|.
    \end{align}
    Since $|\ell-\ell_0-2j-1|\leq 3q_n$, by \eqref{eq:theta_non_res_2} with $\delta'=\delta_1c_0/100<\delta_1/100=3\delta/10$, we obtain
    \begin{align}\label{eq:b2}
        \|2\theta+(\ell-\ell_0-2j-1)\omega\|\geq c_{\delta'}\, e^{-\delta'|\ell-\ell_0-2j-1|}\geq c_{\delta'}\, e^{-3\delta' q_n}\gg e^{-\delta q_n},
    \end{align}
    for $n$ large enough.
    Therefore \eqref{eq:b1} and \eqref{eq:b2} imply
    \begin{align}\label{eq:b3}
        \|\theta-j\omega-(-\theta^{\omega}_{\ell,1}-(q_n-1)\omega)\|\geq \frac{1}{2}c_{\delta'}\, e^{-3\delta' q_n}.
    \end{align}
    Since $|k|\leq q_{n+1}^{1-c_0}$, we have by \eqref{eq:qn_omega} that
    \begin{align}\label{eq:b4}
        \|kq_n\omega\|\leq |k|\|q_n\omega\|\leq \frac{|k|}{q_{n+1}}\leq q_{n+1}^{-c_0}\leq e^{-\delta_1 c_0 q_n}\ll c_{\delta'}e^{-3\delta'q_n},
    \end{align}
    for $n$ large enough.
    Combining \eqref{eq:b3} with \eqref{eq:b4} yields
    \begin{align}
        \|\theta+(kq_n-j)\omega-(-\theta^{\omega}_{\ell,1}-(q_n-1)\omega)\|\geq \frac{1}{4}c_{\delta'}\, e^{-\delta' q_n}\gg e^{-\delta q_n}.
    \end{align}
    This implies
    \begin{align}\label{eq:b5_0}
        &\min_{\ell=0}^{q_n-1} q_n^{-1}\log |e^{2\pi i(\theta+(kq_n-j)\omega)}-e^{-2\pi i(q_n-1)\omega}/z^{\omega}_{q_n,\ell,1}|\\
        \geq &\min_{\ell=0}^{q_n-1}q_n^{-1}\log \|\theta+(kq_n-j)\omega-(-\theta^{\omega}_{\ell,1}-(q_n-1)\omega)\|\geq -\delta.
    \end{align}

    Next, we show $\theta-j\omega+kq_n\omega$ is away from $\bigcup_{\substack{\ell=0\\ \ell\neq \ell_0}}^{q_n-1}\{\theta_{\ell,1}^{\omega}\}$.

    By \eqref{eq:theta_j<-beta} and \eqref{eq:r_theta_periodic}, we conclude that for any $\ell\in \{0,1,...,q_n-1\}\setminus \{\ell_0\}$,
    \begin{align}
        \|\theta-j\omega-\theta^{\omega}_{\ell,1}\|\geq \frac{1}{2q_n}.
    \end{align}
    Combining this with \eqref{eq:b4} yields
    \begin{align}\label{eq:b55}
        \|\theta-j\omega+kq_n\omega-\theta^{\omega}_{\ell,1}\|\geq \frac{1}{2q_n}-e^{-\delta_1c_0q_n}\geq \frac{1}{4q_n}.
    \end{align}

    Next, it suffices to study $\|\theta+(kq_n-j)\omega-\theta_{\ell_0,1}^{\omega}\|$.
    By \eqref{eq:qn_omega}, for $1\leq |k|\leq q_{n+1}^{1-c_0}$,
    \begin{align}
        \|kq_n\omega\|=|k|\|q_n\omega\|\geq \|q_n\omega\|\geq \frac{1}{2}e^{-\beta_n q_n}\gg e^{-(\beta_n+\delta)q_n}.
    \end{align}
    Therefore by triangle inequality and \eqref{eq:theta_j<-beta}, we have 
    \begin{align}\label{eq:b6}
        \|\theta+(kq_n-j)\omega-\theta_{\ell_0,1}^{\omega}\|
        \geq |k|\|q_n\omega\|-\|\theta-j\omega-\theta_{\ell_j,1}^{\omega}\|
        \geq &|k|\|q_n\omega\|-e^{-(\beta_n+\delta)q_n}\notag\\
        \geq &\frac{1}{2}|k|e^{-\beta_n q_n},
    \end{align}
    and by further combining with \eqref{eq:b4},
    \begin{align}\label{eq:b7}
        \|\theta+(kq_n-j)\omega-\theta_{\ell_j,1}^{\omega}\|\leq |k|\|q_n\omega\|+\|\theta-j\omega-\theta_{\ell_j,1}^{\omega}\|\leq e^{-\delta_1c_0q_n}+e^{-(\beta_n+\delta)q_n}\ll \frac{1}{q_n}.
    \end{align}
    Combining \eqref{eq:b55}, \eqref{eq:b6} with \eqref{eq:b7}, we conclude that
    \begin{align}
        \min_{\ell=0}^{q_n-1}q_n^{-1}\log \|\theta+(kq_n-j)\omega-\theta^{\omega}_{\ell,1}\|
        =&q_n^{-1}\log \|\theta+(kq_n-j)\omega-\theta^{\omega}_{\ell_0,1}\|\\
        \geq &\frac{\log |k|}{q_n}-\beta_n-\delta.
    \end{align}
    This implies
    \begin{align}\label{eq:b9_0}
        \min_{\ell=0}^{q_n-1}q_n^{-1}\log |e^{2\pi i(\theta+(kq_n-j)\omega)}-z^{\omega}_{q_n,\ell,1}|
        \geq &\min_{\ell=0}^{q_n-1}q_n^{-1}\log \|\theta+(kq_n-j)\omega-\theta^{\omega}_{\ell,1}\|\\
        \geq &\frac{\log |k|}{q_n}-\beta_n-\delta.
    \end{align}
    Finally, combining \eqref{eq:b5_0} with \eqref{eq:b9_0} yields
    \begin{align}
        &\left(\min_{\ell=0}^{q_n-1}q_n^{-1}\log |e^{2\pi i(\theta+(kq_n-j)\omega)}-z^{\omega}_{q_n,\ell,1}|+\min_{\ell=0}^{q_n-1} q_n^{-1}\log |e^{2\pi i(\theta+(kq_n-j)\omega)}-e^{-2\pi i(q_n-1)\omega}/z^{\omega}_{q_n,\ell,1}|\right)\\
        &\qquad\qquad\qquad\qquad\qquad\qquad\qquad\qquad\qquad\qquad\qquad\qquad\geq \frac{\log |k|}{q_n}-\beta_n-2\delta.
    \end{align}
    This combined with Lemma \ref{lem:vqn_Adelta_3_omega} implies the claimed result.
\end{proof}
Next, we prove 
\begin{lemma}\label{lem:v_qn/2_small}
    For $n$ large enough, 
    \begin{align}
        v_{q_n,E}^{\omega}(e^{2\pi i(\theta-[q_n/2]\omega)})<L(\omega,E)-\beta_n-12\delta.
    \end{align}
\end{lemma}
\begin{proof}
We write $L(\omega,E)$ as $L$ for simplicity. 
    Suppose by contradiction that
    \begin{align}
        v_{q_n,E}^{\omega}(e^{2\pi i(\theta-[q_n/2]\omega)})\geq L-\beta_n-12\delta.
    \end{align}
    By \eqref{eq:f=M+M}, this implies
    \begin{align}
        |P_{q_n,E}^{\omega}(\theta-[q_n/2]\omega)-P_{q_n-1,E}^{\omega}(\theta+\omega-[q_n/2]\omega)|\geq e^{(L-\beta_n-12\delta)q_n}-2.
    \end{align}
    Hence for some $a,b\in \{0,1\}$,
    \begin{align}
        |P_{q_n-a}^{\omega}(\theta+(b-[q_n/2])\omega)|\geq e^{(L-\beta_n-14\delta)q_n}.
    \end{align}
    Then by Green's function expansion \eqref{eq:Green_exp} of $\phi_0$ on the interval $[b-[q_n/2],b-[q_n/2]+q_n-a-1]$, we obtain, similar to \eqref{eq:exp_phi0_weak}, that
    \begin{align}\label{eq:phi0}
        |\phi_0|
        \leq &e^{-(\frac{L}{2}-\beta_n-15\delta)q_n}\cdot \max(|\phi_{b-[q_n/2]-1}|, |\phi_{b-[q_n/2]+q_n-a}|).
    \end{align}
    Note that both $b-[q_n/2]-1=-q_n/2+O(1)$ and $b-[q_n/2]+q_n-a=q_n/2+O(1)$ are in the weakly resonant regime. Hence one can control $\phi$ at those values via Theorem \ref{thm:AL_weak_regime}:
    \begin{align}\label{eq:phi0_2}
        \max(|\phi_{b-[q_n/2]-1}|, |\phi_{b-[q_n/2]+q_n-a}|)
        \leq &e^{-(L-2000C_v\delta_1^{1/4})(1-10\delta_1^{1/4})\frac{q_n}{2}}\cdot \max(r_{-q_n}, r_0, r_{q_n}).
    \end{align}
    Combining \eqref{eq:phi0} with \eqref{eq:phi0_2} yields
    \begin{align}
        |\phi_0|\leq e^{-(L(1-5\delta_1^{1/4})-\beta_n-15\delta-1000C_v\delta_1^{1/4}(1-10\delta_1^{1/4}))q_n} \cdot \max(r_{-q_n},r_0,r_{q_n}).
    \end{align}
    Finally bounding $\max(r_{-q_n},r_0,r_{q_n})\leq Cq_n$ by \eqref{eq:Shnol} yields 
        \begin{align}\label{eq:phi0_3}
        |\phi_0|\leq e^{-(L(1-5\delta_1^{1/4})-\beta_n-16\delta-1000C_v\delta_1^{1/4}(1-10\delta_1^{1/4}))q_n}.
    \end{align}
    Clearly, due to our choice of $\delta_1$, see \eqref{def:delta_1}, the exponential exponent
    $$-(L(1-5\delta_1^{1/4})-\beta_n-16\delta-1000C_v\delta_1^{1/4}(1-10\delta_1^{1/4}))=-(1-o(1))(L-\beta_n)<0,$$
    with some small constant $o(1)\in (0,1/50)$.
    Hence \eqref{eq:phi0_3} can be rewritten as:
    \begin{align}\label{eq:phi0_4}
        |\phi_0|\leq e^{-(1-o(1))(L-\beta_n)q_n}.
    \end{align}
    Since the same argument applies to $\phi_{-1}$, we arrive at a contradiction with $\max(|\phi_0|,|\phi_{-1}|)= 1$ as in \eqref{eq:Shnol}.
\end{proof}
Combining Lemmas \ref{lem:qn/2_small_other_large} and \ref{lem:v_qn/2_small} yields for any $k\neq 0$ and $|k|<10 q_{n+1}^{1-c_0}/q_n$ that
\begin{align}
    v_{q_n,E}^{\omega}(e^{2\pi i(kq_n-[q_n/2])\omega})\geq L(\omega,E)+\frac{\log |k|}{q_n}-\beta_n-12\delta.
\end{align}
Following the same argument as in the proof of Lemma \ref{lem:v_qn/2_small} above, expanding $|\phi_{kq_n+m}|=r_{k q_n}$, for some $kq_n+m\in R_{kq_n}$, we have
\begin{align}
    r_{kq_n}\leq |k|^{-1} e^{-(1-o(1))(L-\beta_n)q_n}\cdot \max(r_{(k-1)q_n},r_{kq_n},r_{(k+1)q_n}).
\end{align}
Since the exponential exponent is negative, we conclude that
\begin{align}
    r_{kq_n}\leq |k|^{-1} e^{-(1-o(1))(L-\beta_n)q_n}\cdot \max(r_{(k-1)q_n},r_{(k+1)q_n}).
\end{align}
For any $1\leq |k_0|\leq 5q_{n+1}^{1-c_0}/q_n$, one can iterate such expansion until one reaches $k=0$ or $|k|=k_1:=[10q_{n+1}^{1-c_0}/q_n]$.
If one reaches $k=0$, then the contribution is controlled by 
\begin{align}
    r_{k_0q_n}
    \leq e^{-(1-o(1))(L-\beta_n)|k_0|q_n}\cdot r_0
    \leq e^{-(1-o(1))(L-\beta_n)|k_0|q_n},
\end{align}
where we dominate $r_0$ by $Cq_n$ using \eqref{eq:Shnol}.
If one reaches $k_1$, then the contribution is controlled by
\begin{align}
    r_{k_0q_n}\leq e^{-(1-o(1))(L-\beta_n)|k_0-k_1|q_n}\cdot C k_1,
\end{align}
where we used \eqref{eq:Shnol} to control $\max(r_{k_1q_n}, r_{-k_1q_n})\leq Ck_1q_n$.
Since $|k_0-k_1|\geq \max(|k_0|, k_1/2)$, we conclude that
\begin{align}
    r_{k_0q_n}\leq e^{-(1-o(1))(L-\beta_n)|k_0|q_n}.
\end{align}
This is the claimed result.
\end{proof}
Finally, combining our analysis in the weakly resonant and strongly resonant regimes, we obtain the following:
\begin{theorem}\label{thm:AL_main_beta}
    For $n$ large enough. If $q_{n+1}\geq e^{\delta_1 q_n}$, then for any $y$ satisfying
    \begin{align}
    \begin{cases}
        q_n/10\leq |y|\leq 5q_{n+1}^{1-c_0}, \text{ if } q_n\leq e^{\delta_1 q_{n-1}}\\
        q_n^{1-c_0}/10\leq |y|\leq 5q_{n+1}^{1-c_0}, \text{ if } q_n\geq e^{\delta_1 q_{n-1}}
    \end{cases}
    \end{align}
    we have 
    \begin{align}
        |\phi_y|\leq e^{-(1-o(1))(L-\beta_n)|y|},
    \end{align}
    where $o(1)\in (0, 1/50)$.
\end{theorem}

\section{Large deviation estimates}\label{sec:LDT}
\subsection{Review of some basic estimates}
Let $v_{m,E}(\theta):=\frac{1}{2m}\log (g_{m,E}(e^{2\pi i\theta}))$, where $g_{m,E}$ is as in \eqref{def:g}.
In the rest of the section, we shall omit the dependence in $E$ since it is fixed.
It is easy to see the following holds for some constant $C_{v,1}=C(\|v\|_{\T_{\varepsilon_0}})>0$:
\begin{align}\label{eq:vm_shift_invariance}
\left|v_m(\theta)-v_m(\theta+\omega)\right|\notag
=&\frac{1}{2m}\left|\log \frac{\|M^{\omega}_{2m}(\theta)\|_{\mathrm{HS}}}{\|M^{\omega}_{2m}(\theta+\omega)\|_{\mathrm{HS}}}\right|\notag\\
\leq &\frac{C}{2m}+\frac{1}{2m}\left|\log \frac{\|M_{2m}^{\omega}(\theta)\|}{\|M_{2m}^{\omega}(\theta+\omega)\|}\right|\notag\\
\leq &\frac{C}{2m}+\frac{1}{2m}\left(\log \|M^{\omega}(\theta)\|+\log \|M^{\omega}(\theta+2m\omega)\|\right)\notag\\
\leq &\frac{C_{v,1}}{m},
\end{align}
where $C>0$ in the second line is an absolute constant arising from the equivalence of the Hilbert-Schmidt norm and the operator norm.

We recall some basic estimates for the Fej\'er kernel:
\[
F_R(k)=\sum_{|j|<R}\frac{R-|j|}{R^2}e^{2\pi i kj\omega}.
\]
The following estimates can be found in \cite{HZ}, with their proofs in \cite[Appendix E]{HZ}.
Below, $p/q$ is an arbitrary continued fractional approximant of $\omega$, as in \ref{sec:continued_fraction}.
\begin{align}\label{eq:FR_1}
0\leq F_R(k)\leq \min(1, \frac{2}{1+R^2\|k\omega\|^2}),
\end{align}
\begin{align}\label{eq:FR_2}
    \sum_{1\leq |k|<q/4} \frac{1}{1+R^2\|k\omega\|^2}\leq 2\pi \frac{q}{R},
\end{align}
and
\begin{align}\label{eq:FR_3}
    \sum_{\ell q/4\leq k<(\ell+1)q/4} \frac{1}{1+R^2\|k\omega\|^2}\leq 2+2\pi \frac{q}{R}.
\end{align}
For a proof of a variant of \eqref{eq:FR_3}, see Sec. \eqref{eq:sum_ellqn}.

We will always use \eqref{eq:FR_1} to bound the Fejer kernel without explicitly referring to it throughout the rest of this section.

We also have two basic estimates for the Fourier coefficients:
\begin{align}\label{eq:Fourier_1/k}
|\hat{v}_m(k)|\leq \frac{C_{v,2}}{|k|},\,  k\neq 0,
\end{align}
for some constant $C_{v,2}=C(\|v\|_{\T_{\varepsilon_0}},\varepsilon_0)>0$. 
Here we used that $v_{m,E}(\theta)\geq 0$ for $\theta\in \T$ and 
\begin{align}
    |g_{m}^{\omega}(e^{2\pi i\theta})|\leq \|M_{m}^{\omega}(\theta)\|_{\mathrm{HS}}^2,
\end{align}
for $\theta\in \T_{\varepsilon_0}$, which implies 
\begin{align}
    v_{m}(\theta)\leq \frac{1}{m}\log \|M_{m}^{\omega}(\theta)\|_{\mathrm{HS}}\leq C\|v\|_{\T_{\varepsilon_0}},
\end{align}
for some absolute constant $C>0$ and uniformly in $\theta\in \T_{\varepsilon_0}$.

The next lemma was first proved in \cite{HZ}, and is useful in particular for small values of $k$.
\begin{lemma}\cite[Lemma 2.4]{HZ}\label{lem:hatu_small_k}
\begin{align}\label{eq:Fourier_1/mk}
|\hat{v}_{m}(k)|\leq \frac{C_{v,3}}{m \|k\omega\|},\, k \neq 0,
\end{align}
for some constant $C_{v,3}>0$.
\end{lemma}
\begin{proof}
    The proof is short. 
    By \eqref{eq:vm_shift_invariance}, 
    \begin{align}
        |\hat{v}_m(k)-e^{2\pi ik\omega}\hat{v}_m(k)|\leq \|v_m(\cdot)-v_m(\cdot+\omega)\|_{L^{\infty}(\T)}\leq \frac{C_{v,1}}{m},
    \end{align}
    this clearly implies \eqref{eq:Fourier_1/k} taking into account that 
    $$|1-e^{2\pi ik\omega}|=2|\sin(\pi k\omega)|\geq 4\|k\omega\|.$$
    This proves Lemma \ref{lem:hatu_small_k}.
\end{proof}
We let 
\begin{align}\label{def:C_v}
    C_v:=\max(C_{v,1},C_{v,2},C_{v,3},1),
\end{align}
where $C_{v,1}, C_{v,2}, C_{v,3}$ are as in \eqref{eq:vm_shift_invariance}, \eqref{eq:Fourier_1/k} and \eqref{eq:Fourier_1/mk} respectively.

\subsection{Proof of Lemma \ref{lem:LDT_non_exp}}\label{sec:LDT_non_exp}

The first large deviation theorem we prove is the following, which implies Lemma \ref{lem:LDT_non_exp} by choosing $\delta=10 \delta_1^{1/2}$.
\begin{lemma}
    For any constant $\delta\in (0,1)$, for large enough $m>m_0(\delta)$ satisfying $10 q_n <m<q_{n+1}/5$, the following holds:
    \begin{align}
        \mathrm{mes}\left(\left\{\theta\in \T: |v_{m,E}(\theta)-L_m(\omega,E)|\geq C_v\left(145\delta+\frac{12\log q_{n+1}}{\delta q_n}\right)\right\}\right)\leq e^{-\delta^2 m}.
    \end{align}
\end{lemma}
\begin{proof} 
We consider $R=[\delta m]$, and 
\begin{align}
    v_m^{(R)}(\theta):=\sum_{|j|<R}\frac{R-|j|}{R^2}v_m(\theta+j\omega).
\end{align}
First note that the zeroth Fourier coefficient is almost $L_m$:
\begin{align}\label{eq:v0=Lm}
    \hat{v}_m(0)=\frac{1}{m}\int_{\T}\log \|M_{m}^{\omega}(\theta)\|_{\mathrm{HS}}\, \mathrm{d}\theta=L_m+\frac{O(1)}{m},
\end{align}
where $O(1)$ is bounded by an absolute constant, due to the equivalence between the Hilbert-Schmidt norm and the operator norm of $2\times 2$ matrices.
Next, we consider
\begin{align}
    v_m(\theta)-\hat{v}_m(0)
    =&v_m(\theta)-v_m^{(R)}(\theta)=:U_1(\theta)\\
    &+v_m^{(R)}(\theta)-\hat{v}_m(0)
\end{align}
Fourier expanding the second line above yields, note the zeroth coefficient cancels:
\begin{align}
v_m^{(R)}(\theta)-\hat{v}_m(0)=
    &\sum_{1\leq |k|\leq \delta^{-2}}\hat{v}_m(k)F_R(k)e^{2\pi i k\theta}=:U_2(\theta)\\
    &+\sum_{\delta^{-2}<|k|<q_n/4}\hat{v}_m(k)F_R(k)e^{2\pi i k\theta}=:U_3(\theta)\\
    &+\sum_{q_n/4\leq |k|<q_{n+1}/4}\hat{v}_m(k)F_R(k)e^{2\pi i k\theta}:=U_4(\theta)\\
    &+\sum_{q_{n+1}/4\leq |k|<e^{4\delta^2 m}}\hat{v}_m(k)F_R(k)e^{2\pi i k\theta}=:U_5(\theta)\\
    &+\sum_{|k|\geq e^{4\delta^2 m}}\hat{v}_m(k)F_R(k)e^{2\pi i k\theta}=:U_6(\theta).
\end{align}
By \eqref{eq:vm_shift_invariance}, we obtain
\begin{align}\label{eq:U1_1}
\|U_1\|_{L^{\infty}(\T)}=\|v_m-v_m^{(R)}\|_{L^{\infty}(\T)}\leq C_v\frac{R}{m}\leq C_v \delta.
\end{align}
Regarding $U_2$, we have by \eqref{eq:Fourier_1/mk} that
\begin{align}\label{eq:U2_1}
    \|U_2\|_{L^{\infty}(\T)}\leq \sum_{1\leq |k|\leq \delta^{-2}} \frac{C_v}{m\|k\omega\|}\leq \frac{2C_v\delta^{-2}}{m}\cdot \max_{1\leq |k|\leq \delta^{-2}} \frac{1}{\|k\omega\|}\leq \delta,
\end{align}
provided $m$ is large enough.
We have by \eqref{eq:FR_2} and \eqref{eq:Fourier_1/k} that
    \begin{align}\label{eq:U3_1}
        \|U_3\|_{L^{\infty}(\T)}\leq C_v\delta^2 \sum_{1\leq |k|<q_n/4}\frac{1}{1+R^2\|k\omega\|_{\T}^2}\leq 2\pi C_v \delta^2 \frac{q_n}{R}\leq C_v\delta,
    \end{align}
    where we used $R=[\delta m]\geq 9\delta q_n$.
    Regarding $U_4$, we conclude by \eqref{eq:Fourier_1/k} and \eqref{eq:FR_3} that
\begin{align}\label{eq:U4_1}
    \|U_4\|_{L^{\infty}(\T)}
    \leq &4C_v\sum_{\ell=1}^{q_{n+1}/q_n} \sum_{k\in [\ell q_n/4, (\ell+1)q_n/4)}\frac{1}{\ell q_n} \frac{1}{1+R^2\|k\omega\|^2}\notag\\
    \leq &4C_v\sum_{\ell=1}^{q_{n+1}/q_n} \frac{1}{\ell q_n} \left(2+2\pi \frac{q_n}{R}\right)\notag\\
    \leq &4C_v \left(\frac{2\log q_{n+1}}{q_n}+\frac{2\pi\log q_{n+1}}{9\delta q_n}\right)\notag\\
    \leq &\frac{12 C_v \log q_{n+1}}{\delta q_n},
\end{align}
in which we used again $R\geq 9\delta q_n$.
We also obtain by \eqref{eq:FR_3} and \eqref{eq:Fourier_1/k} that
    \begin{align}\label{eq:U5_1}
        \|U_5\|_{L^{\infty}(\T)}
        \leq &4C_v\sum_{\ell=1}^{4e^{4\delta^2 m}/q_{n+1}} \frac{1}{\ell q_{n+1}}\sum_{k\in [\ell q_{n+1}/4, (\ell+1)q_{n+1})/4} \frac{1}{1+R^2\|k\omega\|^2}\notag\\
        \leq &4C_v\sum_{\ell=1}^{4e^{4\delta^2 m}/q_{n+1}}\frac{1}{\ell q_{n+1}} \left(2+2\pi \frac{q_{n+1}}{R}\right)\notag\\
        \leq &4C_v\left(\frac{10\delta^2 m}{q_{n+1}}+\frac{10\pi \delta^2 m}{R}\right)\leq 140 C_v\delta,
    \end{align}
    where we used $m\leq q_{n+1}/5$ and $R=[\delta m]$.
Finally it remains to note that by \eqref{eq:Fourier_1/k} that
\begin{align}\label{eq:U6_1}
    \|U_6\|_{L^2}^2=\sum_{|k|>e^{4\delta^2 m}} |\hat{v}_m(k)|^2 |F_R(k)|^2\leq C_v^2 \sum_{|k|>e^{4\delta^2 m}} \frac{1}{|k|^2}\leq 2C_v^2 e^{-4\delta^2 m}.
\end{align}
Combining \eqref{eq:v0=Lm}, \eqref{eq:U1_1}, \eqref{eq:U2_1}, \eqref{eq:U3_1}, \eqref{eq:U4_1} with \eqref{eq:U6_1}, we have for $m$ large enough,
\begin{align}
    \|v_m-L_m-U_6\|_{L^{\infty}(\T)}\leq |\hat{v}_m(0)-L_m|+\sum_{j=1}^5 \|U_j\|_{L^{\infty}(\T)}\leq C_v\left(144\delta+\frac{12\log q_{n+1}}{\delta q_n}\right).
\end{align}
Combining this with \eqref{eq:U6_1} and the Chebyshev's inequality, we conclude that
\begin{align}
    &\mathrm{mes}\left(\left\{\theta: |v_m(\theta)-L_m|>C_v\left(145\delta+\frac{12\log q_{n+1}}{\delta q_n}\right)\right\}\right)\\
    \leq & \mathrm{mes}(\{\theta: |U_6(\theta)|>C_v\delta\})\leq \frac{1}{C_v\delta}\|U_6\|_{L^2(\T)}\leq \sqrt{2}\delta^{-1}e^{-2\delta^2m}\leq e^{-\delta^2m},
\end{align}
as claimed.
\end{proof}

\subsection{Proof of Lemma \ref{lem:LDT_non_exp_m=qn}}\label{sec:LDT_non_exp_m=qn}
Let $\delta_1$ be as in \eqref{def:delta_1}.
The second large deviation we prove is the following, it implies Lemma \ref{lem:LDT_non_exp_m=qn} by choosing $\delta=(400\delta_1)^{1/4}$ and $c_0=10 \delta_1^{1/4}/(1+\beta(\omega))$ (note that for $n$ large, $(\log q_n)/q_{n-1}<1+\beta(\omega))$.
\begin{lemma}\label{lem:LDT_non_exp_m=qn'}
For any constants $c_0, \delta\in (0,1)$, for $n$ be large enough. For $m\in \N$ such that 
\begin{align}\label{eq:m_sim_qn_Lem63}
    \delta^{-1} q_n\geq m\geq 
    \begin{cases}
       \delta q_n^{1-c_0}, \text{ if } q_n\geq e^{\delta_1 q_{n-1}}\\
       \delta q_n, \text{ if } q_n\leq e^{\delta_1 q_{n-1}},
    \end{cases}
\end{align}
the following holds:
    \begin{align}
    \mathrm{mes}\left(\left\{\theta\in \T: |v_{m,E}(\theta)-L_m(\omega,E)|\geq C_v \left(170\delta+4c_0\frac{\log q_n}{q_{n-1}}\right)\right\}\right)\leq e^{-\delta^4 m/4}.
    \end{align} 
\end{lemma}

\begin{proof} 
We consider $R=[\delta m ]$, and 
\begin{align}
    v_m^{(R)}(\theta):=\sum_{|j|<R}\frac{R-|j|}{R^2}v_m(\theta+j\omega).
\end{align}

Next, we consider
\begin{align}
    v_m(\theta)-L_m
    =&\hat{v}_m(0)-L_m\\
    &+v_m(\theta)-v_m^{(R)}(\theta)=:U_1(\theta)\\
    &+\sum_{1<|k|\leq \delta^2 q_n}\hat{v}_m(k)F_R(k)e^{2\pi i k\theta}=:U_2(\theta)\\
    &+\sum_{\delta^2 q_n\leq |k|<e^{\delta^4 m}}\hat{v}_m(k)F_R(k)e^{2\pi i k\theta}=:U_3(\theta)\\
    &+\sum_{|k|\geq e^{\delta^4 m}}\hat{v}_m(k)F_R(k)e^{2\pi i k\theta}=:U_4(\theta)
\end{align}
By \eqref{eq:v0=Lm}, we obtain
\begin{align}\label{eq:v0=Lm_2}
    |\hat{v}_m(0)-L_m|\leq \delta,
\end{align}
for $m$ large enough.
By \eqref{eq:vm_shift_invariance}, we conclude that
\begin{align}\label{eq:U1_2}
\|U_1\|_{L^{\infty}(\T)}=\|v_m-v_m^{(R)}\|_{L^{\infty}(\T)}\leq C_v\frac{R}{m}\leq C_v \delta.
\end{align}
\begin{lemma}\label{lem:U2_2}
Regarding $U_2$, the following holds:
\begin{align}
    \|U_2\|_{L^{\infty}(\T)}\leq C_v\left(55\delta+4c_0\frac{\log q_n}{q_{n-1}}\right).
\end{align}
\end{lemma}
\begin{proof}
Let $q_{n-\ell}$ be such that $q_{n-\ell}\leq \delta^2 q_n<q_{n-\ell+1}$.
Note that $\delta^2 q_n<q_n$, hence 
\begin{align}\label{eq:qn-ell<n-1}
    q_{n-\ell+1}\leq q_n, \text{ and } q_{n-\ell}\leq q_{n-1}.
\end{align}
Note that when $m$ is large, $q_n$ is large and then both $q_{n-\ell+1},q_{n-\ell}$ are large.
Let 
\begin{align}\label{def:j0}
[\delta^2 q_n]=j_0 q_{n-\ell}+r, \text{ with } 0\leq r<q_{n-\ell}.
\end{align}
Note that by \eqref{eq:qn+1=qn+qn-1}, $q_{n-\ell+1}=a_{n-\ell+1}q_{n-\ell}+q_{n-\ell-1}>j_0q_{n-\ell}+r$, hence 
\begin{align}\label{eq:j0<a_n-ell+1}
j_0\leq a_{n-\ell+1}.
\end{align}

We have
\begin{align}
    &\left|\sum_{1\leq |k|\leq \delta^2 q_n}\hat{v}_m(k)F_R(k)e^{2\pi ik\theta}\right|\\
    \leq& \sum_{|j|=1}^{j_0}|\hat{v}_m(jq_{n-\ell})| |F_R(jq_{n-\ell})|+\sum_{j=1}^{j_0}\, \sum_{(j-1)q_{n-\ell}<|k|<jq_{n-\ell}}|\hat{v}_m(k)| |F_R(k)|\\
    &+\sum_{j_0q_{n-\ell}<|k|\leq j_0q_{n-\ell}+r}|\hat{v}_m(k)| |F_R(k)|\\
    \leq &\sum_{|j|=1}^{j_0}|\hat{v}_m(jq_{n-\ell})| |F_R(jq_{n-\ell})|+\sum_{0<|k|<q_{n-\ell}}|\hat{v}_m(k)| |F_R(k)|\\
    &+\sum_{j=2}^{j_0+1}\, \sum_{(j-1)q_{n-\ell}<|k|<jq_{n-\ell}}|\hat{v}_m(k)| |F_R(k)|\\
    =&:I_1+I_2+I_3,
\end{align}
respectively.
First, note that for any $|j|\leq a_{n-\ell+1}$, 
\begin{align}\label{eq:jqn-ell}
    \|jq_{n-\ell}\omega\|=\|j\|q_{n-\ell}\omega\|\|=|j|\|q_{n-\ell}\omega\|.
\end{align}
In fact by \eqref{eq:qn_omega} and \eqref{eq:qn+1=qn+qn-1}, for $|j|\leq a_{n-\ell+1}$,
\begin{align}
    |j|\|q_{n-\ell}\omega\|\leq \frac{a_{n-\ell+1}}{q_{n-\ell+1}}\leq \frac{1}{q_{n-\ell}}<\frac{1}{2},
\end{align}
which implies $\|j\|q_{n-\ell}\omega\|\|=|j|\|q_{n-\ell}\omega\|$ as claimed in \eqref{eq:jqn-ell}.

Next, we estimate $I_1$. By \eqref{eq:Fourier_1/k} and \eqref{eq:jqn-ell}, 
\begin{align}\label{eq:I1_1}
    I_1
    \leq &4C_v\sum_{j=1}^{j_0}\frac{1}{jq_{n-\ell}}\frac{1}{1+R^2 j^2\|q_{n-\ell}\omega\|^2} \notag\\
    \leq &4C_v\sum_{j=1}^{j_0}\frac{1}{jq_{n-\ell}}\frac{1}{1+\delta^2 m^2 j^2\|q_{n-\ell}\omega\|^2}\notag\\
    \leq &4C_v\sum_{j=1}^{j_0}\frac{1}{jq_{n-\ell}}\frac{1}{1+\delta^2 m^2 j^2/(2q_n)^2},
\end{align}
where we used $R=[\delta m]$ and $\|q_{n-\ell}\omega\|\geq 1/(2q_{n-\ell+1})\geq 1/(2q_n)$, due to \eqref{eq:qn_omega} and \eqref{eq:qn-ell<n-1}.

Next, we need to divide into two different cases:

\underline{Case $I_1$-1}. If $q_n\leq e^{\delta_1 q_{n-1}}$. In this case we have $m\geq \delta q_n$ according to \eqref{eq:m_sim_qn_Lem63}.
We can bound $I_1$ in \eqref{eq:I1_1} as follows:
\begin{align}\label{eq:I1_2}
    I_1
    \leq &4C_v\sum_{j=1}^{j_0}\frac{1}{jq_{n-\ell}}\frac{1}{1+\delta^4 j^2/4}\notag\\
    \leq &\frac{4C_v}{q_{n-\ell}}\left(1+\int_1^{\infty} \frac{1}{x(1+\delta^4 x^2/4)}\, \mathrm{d}x \right)\notag\\
    \leq &\frac{4C_v}{q_{n-\ell}}\left(1+\int_{\delta^2/2}^1 \frac{1}{x(1+x^2)}\, \mathrm{d}x+\int_1^{\infty} \frac{1}{x(1+x^2)}\, \mathrm{d}x\right)\notag\\
    \leq &\frac{4C_v}{q_{n-\ell}}\left(1+\int_{\delta^2/2}^1 \frac{1}{x}\, \mathrm{d}x+\int_0^{\infty} \frac{1}{1+x^2}\, \mathrm{d}x\right)\notag\\
    \leq &\frac{4C_v(3+\log (2\delta^{-2}))}{q_{n-\ell}}\leq \delta,
\end{align}
provided $m$ is large.

\underline{Case $I_1$-2}. If $q_n\geq e^{\delta_1 q_{n-1}}$. In this case we have $m\geq \delta q_n^{1-c_0}$ according to \eqref{eq:m_sim_qn_Lem63}. 
Note in this case $\delta^2 q_n\geq \delta^2 e^{\delta_1q_{n-1}}>q_{n-1}$, provided $n$ is large enough.
Hence $n-\ell\geq n-1$, which yields, when combined with \eqref{eq:qn-ell<n-1}, that $q_{n-\ell}=q_{n-1}$.
We bound $I_1$ in \eqref{eq:I1_1} as follows:
\begin{align}\label{eq:I1_3}
I_1
    \leq &4C_v \sum_{j=1}^{j_0}\frac{1}{jq_{n-1}}\frac{1}{1+\delta^4 q_n^{-2c_0} j^2/4} \notag\\
    \leq &\frac{4C_v}{q_{n-1}}\left(1+\sum_{j=2}^{j_0}\frac{1}{j}\frac{1}{1+\delta^4 q_n^{-2c_0}j^2/4}\right)\\
    \leq &\frac{4C_v}{q_{n-1}}\left(1+\int_1^\infty \frac{1}{x(1+\delta^4 q_n^{-2c_0} x^2/4)}\, \mathrm{d}x\right) \notag\\
    \leq &\frac{4C_v}{q_{n-1}}\left(1+\int_{\delta^2 q_n^{-c_0}/2}^1\frac{1}{x(1+x^2)}\, \mathrm{d}x+\int_1^{\infty}\frac{1}{1+x^2}\, \mathrm{d}x\right)\notag \\
    \leq &4C_v\left(\frac{3}{q_{n-1}}+\frac{\log (2\delta^{-2}q_n^{c_0})}{q_{n-1}}\right)\notag\\
    \leq &4C_vc_0\frac{\log q_n}{q_{n-1}}+\delta,
\end{align}
provided $n$ is large enough.

Next, we study $I_2$.
By \eqref{eq:Fourier_1/mk}, 
\begin{align}\label{eq:I2_11}
    I_2=\sum_{1\leq |k|<q_{n-\ell}}|\hat{v}_m(k)| |F_R(k)|
    \leq &4 C_v\sum_{1\leq k<q_{n-\ell}}\frac{1}{m\|k\omega\|}\frac{1}{1+\delta^2 m^2\|k\omega\|^2}.
\end{align}
For $0<|k|<q_{n-\ell}$, we obtain by \eqref{eq:qn_omega_min} and \eqref{eq:qn_omega} that
\begin{align}\label{eq:I2_komega}
    \|k\omega\|\geq \|q_{n-\ell-1}\omega\|\geq \frac{1}{2q_{n-\ell}}.
\end{align}
This implies the $I_2$ as in \eqref{eq:I2_11} can be bounded by:
\begin{align}\label{eq:I2_111}
    I_2\leq \frac{8C_vq_{n-\ell}}{m} \sum_{1\leq k<q_{n-\ell}} \frac{1}{1+\delta^2 m^2\|k\omega\|^2}.
\end{align}
We decompose the sum over $\{1,2,...,q_{n-\ell}-1\}$ into sums of two subsets: $\{1,2,...,q_{n-\ell}-1\}=K_1\bigcup K_2$, where 
\begin{align}\label{def:K1_K2_sets}
    K_1:=&\{k\in \{1,2,...,q_{n-\ell}-1\}: k\omega-[k\omega]\in (0,1/2)\}, \text{ and } \\
    K_2:=&\{k\in \{1,2,...,q_{n-\ell}-1\}: k\omega-[k\omega]\in (1/2,1)\}
\end{align}
For $k_1\neq k_2$ such that $\{k_1,k_2\}\subset K_1$, clearly $\|k_j\omega\|=k_j\omega-[k_j\omega]$ holds for $j=1,2$, and hence
$$\|\|k_1\omega\|-\|k_2\omega\|\|=\|(k_1-k_2)\omega\|.$$
Combining this with $\|k_1\omega\|-\|k_2\omega\|\in (-1/2, 1/2)$ and that for $x\in (-1/2,1/2)$, $|x|=\|x\|$, we have
\begin{align}\label{eq:I2_K1}
    |\|k_1\omega\|-\|k_2\omega\||=\|k_1\omega-k_2\omega\|\geq \|q_{n-\ell-1}\omega\|\geq \frac{1}{2q_{n-\ell}},
\end{align}
where we used $0<|k_1-k_2|<q_{n-\ell}$ and the estimate similar to \eqref{eq:I2_komega}.
Combining \eqref{eq:I2_komega} with \eqref{eq:I2_K1}, we conclude that $\{\|k\omega\|\}_{k\in K_1}$ are non-negative terms, which are at least $1/(2q_{n-\ell})$ spaced with the smallest being at least $1/(2q_{n-\ell})$. 
This implies
\begin{align}\label{eq:I2_sum_K1}
    \sum_{k\in K_1}\frac{1}{1+\delta^2 m^2 \|k\omega\|^2}\leq \sum_{s=1}^{q_{n-\ell}}\frac{1}{1+\delta^2 m^2 s^2/(2q_{n-\ell})^2}.
\end{align}

The sum in $K_2$ is similar. In fact for $\{k_1\neq k_2\}\subset K_2$, 
$\|k_j\omega\|=[k_j\omega]+1-k_j\omega$ holds for $j=1,2$. Hence, similar to \eqref{eq:I2_K1}, we have
\begin{align}\label{eq:I2_K2}
    |\|k_1\omega\|-\|k_2\omega\||=\|k_1\omega-k_2\omega\|\geq \|q_{n-\ell-1}\omega\|\geq \frac{1}{2q_{n-\ell}}.
\end{align}
This implies similar to \eqref{eq:I2_sum_K1} that
\begin{align}\label{eq:I2_sum_K2}
     \sum_{k\in K_2}\frac{1}{1+\delta^2 m^2 \|k\omega\|^2}\leq \sum_{s=1}^{q_{n-\ell}}\frac{1}{1+\delta^2 m^2 s^2/(2q_{n-\ell})^2}.   
\end{align}
Combining \eqref{eq:I2_sum_K1}, \eqref{eq:I2_sum_K2} with \eqref{eq:I2_111}, we obtain
\begin{align}\label{eq:I2_1111}
    I_2\leq &\frac{16C_vq_{n-\ell}}{m}\sum_{s=1}^{q_{n-\ell}}\frac{1}{1+\delta^2 m^2s^2/(2q_{n-\ell})^2} \notag\\
    \leq &\frac{16C_vq_{n-\ell}}{m}\int_0^{\infty}\frac{1}{1+\delta^2 m^2 x^2/(2q_{n-\ell})^2}\, \mathrm{d}x \notag \\
    \leq &\frac{32C_vq_{n-\ell}^2}{\delta m^2}\int_0^{\infty}\frac{1}{1+x^2}\, \mathrm{d}x \notag\\
    \leq &\frac{16\pi C_v q_{n-\ell}^2}{\delta m^2}.
\end{align}
If $q_n\geq e^{\delta_1 q_{n-1}}$, we bound $m\geq \delta q_n^{1-c_0}\geq \delta e^{\delta_1(1-c_0)q_{n-1}}\geq q_{n-1}^2\geq q_{n-\ell}^2$ for $m$ large enough, where we used $q_{n-1}\geq q_{n-\ell}$ as in \eqref{eq:qn-ell<n-1}. 
This implies the following bound on $I_2$ as in \eqref{eq:I2_1111}:
\begin{align}\label{eq:I2_1}
    I_2\leq \frac{16\pi C_v}{\delta q_{n-\ell}^2}<\delta,
\end{align}
for $m$ large enough.

If $q_n\leq e^{\delta_1 q_{n-1}}$, we bound $m\geq \delta q_n$ and $q_{n-\ell}\leq \delta^2 q_n$ (according to the definition of $q_{n-\ell}$), then
\begin{align}\label{eq:I2_2}
    I_2\leq \frac{16\pi C_v\delta^4 q_n^2}{\delta^3 q_n^2}\leq 16\pi C_v\delta.
\end{align}

Next, we consider $I_3$. We distinguish two cases.

\underline{Case $I_3$-1.} If $j_0\leq 10$.
We simply estimate, via \eqref{eq:FR_3} and \eqref{eq:Fourier_1/k} (we divide $((j-1)q_{n-\ell}, jq_{n-\ell})$ into four intervals of length $q_{n-\ell}/4$ and apply \eqref{eq:FR_3} to each of these four) that
\begin{align}\label{eq:I3_case1}
    I_3 \leq \sum_{j=2}^{10}\, \sum_{(j-1)q_{n-\ell}<|k|<jq_{n-\ell}}|\hat{v}_m(k)||F_R(k)|
    \leq &C_v \sum_{j=2}^{10}\frac{8}{q_{n-\ell}}\left(8+8\pi\frac{q_{n-\ell}}{R}\right)\notag\\
    \leq &C_v\left(\frac{576}{q_{n-\ell}}+\frac{576\pi}{R}\right)<\delta,
\end{align}
provided $m$ is large enough.

\underline{Case $I_3$-2.} If $j_0>10$.
We obtain by \eqref{eq:Fourier_1/k} that for each $2\leq j\leq j_0+1$,
\begin{align}\label{eq:I3_111}
    \sum_{(j-1)q_{n-\ell}<|k|<jq_{n-\ell}}|\hat{v}_m(k)||F_R(k)|
    \leq &\frac{4C_v}{(j-1)q_{n-\ell}}\, \sum_{(j-1)q_{n-\ell}<k<jq_{n-\ell}}\frac{1}{1+\delta^2 m^2\|k\omega\|^2}.
\end{align}
Since $|k-(j-1)q_{n-\ell}|<q_{n-\ell}$, 
\begin{align}\label{eq:k-(j-1)}
    \|k\omega-(j-1)q_{n-\ell}\omega\|\geq \|q_{n-\ell-1}\omega\|.
\end{align}
For each $2\leq j\leq [j_0/2]$, using $j_0\leq a_{n-\ell+1}$ and that $a_{n-\ell+1}\|q_{n-\ell}\omega\|<\|q_{n-\ell-1}\omega\|$ (see \eqref{eq:qn-1_norm=qn+qn+1}), we have by \eqref{eq:jqn-ell} that
\begin{align}\label{eq:j-1q_small}
    \|(j-1)q_{n-\ell}\omega\|=(j-1)\|q_{n-\ell}\omega\|\leq (\frac{j_0}{2}-1)\|q_{n-\ell}\omega\|
    \leq \frac{1}{2}\|q_{n-\ell-1}\omega\|.
\end{align}
Combining \eqref{eq:k-(j-1)} with \eqref{eq:j-1q_small}, we have by triangle inequality that for 
$(j-1)q_{n-\ell}<k<jq_{n-\ell}$, 
\begin{align}\label{eq:I3_2_min}
    \|k\omega\|\geq \frac{1}{2}\|q_{n-\ell-1}\omega\|\geq \frac{1}{4q_{n-\ell}}.
\end{align}
Similar to \eqref{def:K1_K2_sets}, 
For each $j\leq j_0+1$, we define
\begin{align}\label{def:K3_K4}
    K_3:=&\{k\in ((j-1)q_{n-\ell},jq_{n-\ell}): k\omega-[k\omega]\in (0,1/2)\}, \text{ and } \\
    K_4:=&\{k\in ((j-1)q_{n-\ell},jq_{n-\ell}): k\omega-[k\omega]\in (1/2,1)\}
\end{align}
Similar to \eqref{eq:I2_K1} and \eqref{eq:I2_K2}, we can obtain pairwise spacing of size $\|q_{n-\ell-1}\omega\|$, among $\{\|k\omega\|\}_{k\in K_3}$ and $\{\|k\omega\|\}_{k\in K_4}$, respectively. Together with a control of the minimum value in \eqref{eq:I3_2_min}, we conclude that $\{\|k\omega\|\}_{k\in K_3}$ (and $\{\|k\omega\|\}_{k\in K_4}$) are at least $1/(2q_{n-\ell})$ spaced and the smallest being at least $1/(4q_{n-\ell})$.
Thus we can bound \eqref{eq:I3_111} as follows: 
\begin{align}
    \sum_{(j-1)q_{n-\ell}<|k|<jq_{n-\ell}}|\hat{v}_m(k)||F_R(k)|
    \leq &\sum_{r=3}^4 \frac{4C_v}{(j-1)q_{n-\ell}}\, \sum_{k\in K_r}\frac{1}{1+\delta^2 m^2\|k\omega\|^2}\\
    \leq &\frac{8C_v}{(j-1)q_{n-\ell}}\sum_{s=1}^{q_{n-\ell}}\frac{1}{1+\delta^2 m^2 s^2/(4q_{n-\ell})^2}\\
    \leq &\frac{8C_v}{(j-1)q_{n-\ell}}\int_0^{\infty}\frac{1}{1+\delta^4 q_n^{2-2c_0} x^2/(4q_{n-\ell})^2}\, \mathrm{d}x\\
    \leq &\frac{16\pi C_v}{(j-1)\delta^2 q_n^{1-c_0}},
\end{align}
where we used $m\geq \delta q_n^{1-c_0}$, which is satisfied in both cases (see \eqref{eq:m_sim_qn_Lem63}).
Therefore, the estimate above yields
\begin{align}\label{eq:I3_1}
    &\sum_{2\leq j\leq [j_0/2]}\sum_{(j-1)q_{n-\ell}<|k|<jq_{n-\ell}}|\hat{v}_m(k)||F_R(k)|\notag\\
    \leq &\sum_{2\leq j\leq [j_0/2]}\frac{16\pi C_v}{(j-1)\delta^2 q_n^{1-c_0}}\leq \frac{16\pi C_v\log j_0}{\delta^2 q_n^{1-c_0}}\leq \frac{16\pi C_v\log(\delta^2 q_n)}{\delta^2 q_n^{1-c_0}}<\delta,
\end{align}
for $n$ large enough. Note that we controlled $j_0\leq \delta^2 q_n$ above, due to \eqref{def:j0}.

For each $j$ such that $[j_0/2]<j\leq j_0+1$, define $K_3,K_4$ as in \eqref{def:K3_K4} above, one can show that $\{\|k\omega\|\}_{k\in K_3}$ (and $\{\|k\omega\|\}_{k\in K_4}$) are at least $1/(2q_{n-\ell})$ spaced and the smallest term being at least $0$ (which is a trivial lower bound).
Therefore we can bound \eqref{eq:I3_111} as follows: 
\begin{align}\label{eq:I3_2'}
    \sum_{(j-1)q_{n-\ell}<|k|<jq_{n-\ell}}|\hat{v}_m(k)||F_R(k)|
    \leq &\sum_{r=3}^4 \frac{4C_v}{(j-1)q_{n-\ell}}\, \sum_{k\in K_r}\frac{1}{1+\delta^2 m^2\|k\omega\|^2}\notag\\
    \leq &\frac{8C_v}{(j-1)q_{n-\ell}}\left(1+\sum_{s=1}^{q_{n-\ell}}\frac{1}{1+\delta^2 m^2 s^2/(2q_{n-\ell})^2}\right)\notag\\
    \leq &\frac{8C_v}{(j-1)q_{n-\ell}}\left(1+\int_0^{\infty} \frac{1}{1+\delta^2 m^2 x^2/(2q_{n-\ell})^2}\, \mathrm{d}x\right) \notag\\
    \leq &\frac{8C_v}{(j-1)q_{n-\ell}}\left(1+\frac{\pi q_{n-\ell}}{\delta m}\right).
\end{align}
This implies
\begin{align}\label{eq:I3_2}
    \sum_{j=[j_0/2]+1}^{j_0+1}\, \sum_{(j-1)q_{n-\ell}<|k|<jq_{n-\ell}}|\hat{v}_m(k)||F_R(k)|\leq 8C_v\log 2\left(\frac{1}{q_{n-\ell}}+\frac{\pi}{\delta m}\right)<\delta,
\end{align}
for $m$ large enough.
Combining \eqref{eq:I3_1} with \eqref{eq:I3_2}, we conclude in Case $I_3$-2 that
\begin{align}\label{eq:I3_case2}
I_3\leq 2\delta.
\end{align}

Combining the estimates of $I_1,I_2,I_3$ in \eqref{eq:I1_2}, \eqref{eq:I1_3}, \eqref{eq:I2_1}, \eqref{eq:I2_2}, \eqref{eq:I3_case1}, \eqref{eq:I3_case2} yields the claimed result for $U_2$.
\end{proof}

\begin{lemma}\label{lem:U3_2}
Regarding $U_3$, the following holds:
\begin{align}
    \|U_3\|_{L^{\infty}(\T)}\leq 110 C_v\delta.
\end{align}
\end{lemma}
\begin{proof}
To prove this lemma we need the following estimate, which is a modification of \eqref{eq:FR_3}.
\begin{align}\label{eq:sum_ellqn}
    \sum_{\ell \delta^2 q_n<k<(\ell+1)\delta^2 q_n} \frac{1}{1+R^2\|k\omega\|^2}\leq 2+\frac{2\pi q_n}{R}.
\end{align}
To see why this is true:
for any $\{k_1\neq k_2\}\subset (\ell \delta^2 q_n, (\ell+1)\delta^2 q_n)$, we have $0<|k_1-k_2|<\delta^2 q_n<q_n$. Hence by \eqref{eq:qn_omega_min} and \eqref{eq:qn_omega},
\begin{align}
\|(k_1-k_2)\omega\|\geq \|q_{n-1}\omega\|\geq \frac{1}{2q_n}.
\end{align}
Define
\begin{align}\label{def:K7_K8}
    K_5:=&\{k\in (\ell \delta^2 q_n,(\ell+1)\delta^2 q_n): k\omega-[k\omega]\in (0,1/2)\}, \text{ and } \\
    K_6:=&\{k\in (\ell \delta^2 q_n,(\ell+1)\delta^2 q_n): k\omega-[k\omega]\in (1/2,1)\}.
\end{align}
Similar to \eqref{eq:I2_K1}, \eqref{eq:I2_K2}, we have $\{\|k\omega\|\}_{k\in K_5}$ (and $\{\|k\omega\|\}_{k\in K_6}$) have pairwise spacing at least $\|q_{n-1}\omega\|\geq 1/(2q_n)$, with the smallest term being at least $0$. 
Hence
\begin{align}
    \sum_{\ell \delta^2 q_n<k<(\ell+1)\delta^2 q_n} \frac{1}{1+R^2\|k\omega\|^2}
    =&\sum_{r=5,6}\sum_{k\in K_r}\frac{1}{1+R^2\|k\omega\|^2}\\
    \leq &2\left(1+\sum_{s=1}^{\delta^2 q_n}\frac{1}{1+R^2s^2/(2q_n)^2}\right)\\
    \leq &2\left(1+\int_0^{\infty} \frac{1}{1+R^2x^2/(2q_n)^2}\, \mathrm{d}x\right)\\
    \leq &2\left(1+\frac{\pi q_n}{R}\right).
\end{align}
This proves \eqref{eq:sum_ellqn}. Clearly, \eqref{eq:sum_ellqn} combined with \eqref{eq:Fourier_1/k} implies
\begin{align}
    \|U_3\|_{L^{\infty}(\T)}
    \leq &\frac{4C_v}{\delta^2 q_n}
    \sum_{\ell=1}^{e^{\delta^4 m}/(\delta^2 q_n)}\frac{1}{\ell} \sum_{\ell \delta^2 q_n\leq k<(\ell+1)\delta^2 q_n}\frac{1}{1+R^2\|k\omega\|^2} \notag\\
    \leq &8 C_v\delta^2 m\left(\frac{1}{q_n}+\frac{\pi}{R}\right)\\
    \leq &80 C_v\delta+8\pi C_v\delta\leq 110 C_v\delta,
\end{align}
in which we used $m\leq \delta^{-1} q_n$ and $R=[\delta m]$.
This proves the claimed result for $U_3$.
\end{proof}
For $U_4$, we obtain by \eqref{eq:Fourier_1/k} that
\begin{align}\label{eq:U4_2}
    \|U_4\|_{L^2(\T)}^2=\sum_{|k|>e^{\delta^4 m}} |\hat{v}_m(k)|^2 |F_R(k)|^2\leq C_v^2 \sum_{|k|>e^{\delta^4 m}} \frac{1}{|k|^2}\leq 2C_v^2 e^{-\delta^4 m}.
\end{align}
Combining \eqref{eq:v0=Lm_2}, \eqref{eq:U1_2}, Lemmas \ref{lem:U2_2} and \eqref{lem:U3_2} with the Chebyshev's inequality and \eqref{eq:U4_2}, we conclude that
\begin{align}
    &\mathrm{mes}\left(\left\{\theta: |v_m(\theta)-L_m|>C_v\left(170\delta+4c_0\frac{\log q_n}{q_{n-1}}\right)\right\}\right)\\
    \leq &\mathrm{mes}(\{\theta: |U_4(\theta)|\geq 3C_v\delta\})\\
    \leq &\frac{1}{3C_v\delta}\|U_4\|_{L^2(\T)}\leq e^{-\delta^4 m/4},
\end{align}
as claimed.
\end{proof}

\end{document}